\documentclass[letterpaper, 11pt]{amsart}

\usepackage{microtype}
\usepackage{graphicx}
\usepackage{booktabs}
\usepackage[margin=1.206in]{geometry}

\usepackage{mathtools,algorithm,algorithmic}
\usepackage{tikz} 
\usepackage{scrextend}
\usepackage{amsmath,amsthm,amssymb,url,amscd}
\usepackage{amsfonts}

\usepackage[scaled=.98,sups,osf]{XCharter}
\usepackage[scaled=1.04,varqu,varl]{inconsolata}
\usepackage[type1]{cabin}
\usepackage[charter,vvarbb,scaled=1.05]{newtxmath}

\usepackage{multirow}
\usepackage{enumerate}
\usepackage{accents}
\usepackage{amscd}
\usepackage{float}
\usepackage{bm}
\usepackage{bbm}
\usepackage{caption}
\usepackage{subcaption}
\usepackage[colorlinks=true, linkcolor = citeblue, citecolor=citeblue, urlcolor=citeblue]{hyperref}
 \usepackage{cleveref}
\usepackage{xcolor}
\usepackage{setspace}
\usepackage{pgfplots}
\pgfplotsset{compat=newest}
\usetikzlibrary{plotmarks}
\usetikzlibrary{arrows.meta}
\usepgfplotslibrary{patchplots}
  
\xdefinecolor{citeblue}{RGB}{  0,45,154}
  
\pgfplotsset{plot coordinates/math parser=false}
\newlength\figureheight
\newlength\figurewidth

\DeclareMathOperator{\trace}{tr}

\DeclareMathOperator{\Id}{Id}

\DeclareMathOperator*{\argmin}{arg\,min}

\DeclareMathOperator{\rank}{rank}

\DeclareMathOperator{\prox}{prox}
\DeclareMathOperator{\R}{\mathbb{R}}
\DeclareMathOperator{\N}{\mathbb{N}}

\DeclareMathOperator{\diag}{diag}
\DeclareMathOperator{\dg}{dg}

\DeclareMathOperator{\Ex}{\mathbb{E}}
\renewcommand\f[1]{\mathbf{#1}}
\newcommand\y[1]{\mathcal{#1}}

\newcommand\Rdd{\R^{d_1 \times d_2}}

\newcommand\hk{^{(k)}}
\newcommand\hkk{^{(k+1)}}

\newcommand\bb[1]{\mathbb{#1}}
\newtheorem{theorem}{Theorem}[section]

\newtheorem{lemma}{Lemma}[section]

\newtheorem{remark}{Remark}[section]

\newtheorem{definition}{Definition}[section]

\usepackage{hyperref}

\title{A Scalable Second Order Method for Ill-Conditioned Matrix Completion from Few Samples}
\author[C. K\"ummerle \and C. Mayrink Verdun]{Christian K\"ummerle$^*$ \and Claudio Mayrink Verdun$^\dagger$}
\thanks{$^*$Department of Applied Mathematics \& Statistics, Johns Hopkins University, Baltimore, USA (\href{mailto:kuemmerle@jhu.edu}{kuemmerle@jhu.edu}).}
\thanks{$^\dagger$Department of Mathematics and Department of Electrical and Computer Engineering, Technical University of Munich, Munich, Germany (\href{mailto:claudio.verdun@tum.de}{claudio.verdun@tum.de})}

\begin{document}

\maketitle

\begin{abstract}

We propose an iterative algorithm for low-rank matrix completion that can be interpreted as an iteratively reweighted least squares (IRLS) algorithm, a saddle-escaping smoothing Newton method or a variable metric proximal gradient method applied to a non-convex rank surrogate. It combines the favorable data-efficiency of previous IRLS approaches with an improved scalability by several orders of magnitude. We establish the first local convergence guarantee from a minimal number of samples for that class of algorithms, showing that the method attains a local quadratic convergence rate. Furthermore, we show that the linear systems to be solved are well-conditioned even for very ill-conditioned ground truth matrices.  We provide extensive experiments, indicating that unlike many state-of-the-art approaches, our method is able to complete very ill-conditioned matrices with a condition number of up to $10^{10}$ from few samples, while being competitive in its scalability.

\end{abstract}

\section{Introduction}\label{intro}

In different areas of machine learning and signal processing, low-rank models have turned out to be a powerful tool for the acquisition, storage and computation of information. 
In many of these applications, an important sub-problem is to infer the low-rank model from partial or incomplete data \cite{Davenport16,ChiLuChen19}. 

This problem is called \emph{low-rank matrix completion}: Given a matrix $\f{X}^0 \in \Rdd$ of rank-$r$ and an index set $\Omega \subset [d_1] \times [d_2]$, the task is to reconstruct $\f{X}^0$ just from the knowledge of $\Omega$ and $P_{\Omega}(\f{X}^0)$, where $P_{\Omega}: \Rdd \to \mathbb{R}^m$ is the subsampling operator that maps a matrix to the set of entries indexed by $\Omega$. It is well-known that this can be reformulated \cite{Recht10} as the NP-hard \emph{rank minimization} problem
\begin{equation}\label{rank_equation}
    \min_{\f{X} \in \R^{d_1 \times d_2}} \rank(\f{X}) \quad \mbox{ subject to } P_{\Omega}(\f{X}) = P_{\Omega}(\f{X}^0).
\end{equation}

From an optimization point of view, \eqref{rank_equation} is particularly difficult to handle due to two properties: its \emph{non-convexity} and its \emph{non-smoothness}. A widely studied approach in the literature replaces the $\rank(\f{X})$ by the (convex) nuclear norm $\|\f{X}\|_* = \sum_{i=1}^d \sigma_i(\f{X})$ \cite{FHB03}, which is the tightest convex envelope of the rank, as an objective. For this approach, a mature theory has been developed that includes performance guarantees for a near-optimal sample complexity \cite{CandesTao10,Chen15} and robustness to noise \cite{CandesPlan10,Chen20noisy}. 

However, from a practical point of view, using such a convex relaxation to find a low-rank completion is \emph{computationally very demanding}, as even first-order solvers have an per-iteration arithmetic complexity that is at least cubic in the dimensions of $\f{X}^0$ \cite{ChiLuChen19}. Thus, convex relaxations are of little use in large-scale applications of the model such as in recommender systems \cite{koren_bell_volinsky}, where even storing the dense matrix $\f{X}^0 \in \R^{d_1 \times d_2}$ is prohibitive. Another important, but less well-known issue is that a convex relaxation is \emph{typically not as data efficient} as certain other algorithms \cite{TannerWei13,BauchNadler20}, i.e., nuclear norm minimization typically necessitates a larger amount of samples $m$ than other methods, measured by the quotient $\rho:= m /(d_1+d_2 -r)$ (oversampling ratio) between $m$ and the number of degrees of freedom of $\f{X}^0$, to identify $\f{X}^0$ correctly \cite{Amelunxen14}.

To overcome these drawbacks, a variety of alternative approaches have been proposed and studied. Among the most popular ones are ``non-convex'' algorithms based on matrix factorization \cite{Burer03} with objective
\begin{equation} \label{eq:UV:recommender}
J(\f{U},\f{V}) \!\!:=\!\! \left\| P_{\Omega}(\f{U}\f{V}^*)\! -\! P_{\Omega}(\f{X}^0) \right\|_F^2\! + \!\frac{\lambda}{2} \left( \|\f{U}\|_F^2\! +\!\|\f{V}\|_F^2\right)
\end{equation}
for $\lambda \geq 0$, which use (projected) gradient descent on the two factor matrices \cite{SunL15,ZhengL16,MaWangChiChen19}, or related methods. These methods are much more scalable than those optimizing a convex rank surrogate, while also allowing for a theoretical analysis, see \cite{ChiLuChen19} for a recent survey. Furthermore, among the most data-efficient methods for low-rank completion are those that minimize a smooth objective over the Riemannian manifold of fixed rank matrices \cite{Vandereycken13,wei_cai_chan_leung,boumal_absil_15,BauchNadler20}. These approaches are likewise scalable and often able to reconstruct the low-rank matrix from fewer samples $m$ than a convex formulation, but strong performance guarantees have remained elusive so far.

In many instances of our problem, such as in the discretization of PDE-based inverse problems with Fredholm equations \cite{cloninger_czaja_bai_basser} or in spectral estimation problems modeled by structured low-rank matrices, it is an additional difficulty that the matrix of interest $\f{X}^0$ is severely \emph{ill-conditioned}, i.e., $\kappa= \sigma_1(\f{X}^0)/\sigma_r(\f{X}^0)$ might be very large (up to $\kappa = 10^{15}$ in spectral estimation \cite{fannjiang_liao}).
\subsection*{Our contribution} In this paper, we propose and analyze the algorithm \emph{Matrix Iteratively Reweighted Least Squares} (\texttt{MatrixIRLS}) that is designed to find low-rank completions that are potentially very ill-conditioned, allowing for a scalable implementation. It is based on the minimization of quadratic models of a sequence of continuously differentiable, \emph{non-convex ``relaxations''} of the rank function $\rank(\f{X})$. We note that, while being severely non-convex, our method is fundamentally different from a typical non-convex approach with an objective such as \eqref{eq:UV:recommender}.

Let $D = \max(d_1,d_2)$ and $d=\min(d_1,d_2)$. From a theoretical angle, we establish that if the $m$ sampled entries are distributed uniformly at random and if $m = \Omega(\mu_0 r D \log(D))$, with high probability, \texttt{MatrixIRLS} \emph{exhibits local convergence} to $\f{X}^0$ \emph{with a local quadratic convergence rate}, where $\mu_0$ is an incoherence factor. This sample complexity does not depend on the condition number $\kappa$, is \emph{optimal} under the sampling model and improves, to the best of our knowledge, on the state-of-the-art of any algorithmic sample complexity result for low-rank matrix completion---albeit, with the caveat that unlike many other results, our guarantee is inherently \emph{local}.
 
Furthermore, we show that the algorithm can be implemented in a per-iteration cost that is sub-quadratic in $D$, without the need of storing dense $(d_1 \times d_2)$ matrices. We show that under the random sampling model, the linear systems to be solved in the main computational step of \texttt{MatrixIRLS} are well-conditioned even close to the ground truth, unlike the systems of comparable IRLS algorithms in the literature \cite{Daubechies10,Fornasier11,Mohan10,KS18}. 
 
The data-efficiency and scalability of our method compared to several state-of-the-art methods is finally explored in numerical experiments involving simulated data.

\section{MatrixIRLS for log-det rank surrogate}\label{surrogate}
The starting point of the derivation of our method is the observation that minimizing a \emph{non-convex surrogate} objective $F$ with more regularity than $\rank(\f{X})$ can lead to effective methods for solving \eqref{rank_equation} that may combine some of the aforementioned properties, e.g., if $F$ is chosen as a \emph{log-determinant} \cite{Fazel02,Candes13}, Schatten-$p$ quasi-norm (with $0 <p <1$) \cite{Giampouras20} or a smoothed clipped absolute deviation (SCAD) of the singular values \cite{Mazumder20}. In particular, it has been observed in several works \cite{Fazel02,Candes13} that optimizing the smoothed log-det objective $\sum_{i=1}^d \log(\sigma_i(\f{X}+\epsilon \f{I}))$ for some $\epsilon > 0$ can lead to less biased solutions than a nuclear norm minimizer---very generally, it can be shown that a minimizer of non-convex spectral functions such as the smoothed log-det objective coincides as least as often with the rank minimizer as the convex nuclear norm minimizer \cite{Foucart18}. Relevant algorithmic approaches to minimize non-convex rank surrogates include iterative thresholding methods \cite{Mazumder20}, iteratively reweighted least squares \cite{Fornasier11,Mohan10,KS18} and iteratively reweighted nuclear norm \cite{Lu2015nonconvex} algorithms.

However, finding the global minimizer of a non-convex and non-smooth rank surrogate can be very challenging, as the existence of sub-optimal local minima and saddle points might deter the success of many local optimization approaches. Furthermore, applications such as in recommender systems \cite{koren_bell_volinsky} require solving very high-dimensional problem instances so that it is impossible to store full matrices, let alone to calculate many singular values of these matrices, ruling out the applicability of many of the existing methods for non-convex surrogates. A major shortcoming is, finally, also that the available convergence theory for such algorithms is still very immature---a convergence theory quantifying the sample complexity or convergence rates is, to the best of our knowledge, not available for any method of this class. 

To derive our method, let now $\epsilon > 0$ and $F_{\epsilon}:\Rdd \to \R$ be the \emph{smoothed log-det objective} defined as
$F_{\epsilon}(\f{X}):= \sum_{i=1}^d f_{\epsilon}(\sigma_i(\f{X}))$
with $d=\min(d_1,d_2)$ and
\begin{equation} \label{eq:smoothing:Fpeps}	
f_{\epsilon}(\sigma) =
\begin{cases}
 \log|\sigma|, & \text{ if } \sigma \geq \epsilon, \\
 \log(\epsilon) + \frac{1}{2}\Big( \frac{\sigma^2}{\epsilon^2}-1\Big), & \text{ if } \sigma < \epsilon.
 \end{cases}
 \end{equation}
It can be shown that that $F_{\epsilon}$ is continuously differentiable with $\epsilon^{-2}$-Lipschitz gradient
\begin{equation*}
 \nabla F_{\epsilon_k}(\f{X}) = \f{U} \dg \bigg(\frac{\sigma_i(\f{X})}{\max(\sigma_i(\f{X}),\epsilon_k)^{2}}\bigg)_{i=1}^d \f{V}^*,
\end{equation*}
where $\f{X}$ has a singular value decomposition $\f{X} = \f{U} \dg\big(\sigma(\f{X})\big) \f{V}^* = \f{U} \dg\big(\sigma \big) \f{V}^*$. It is clear that the optimization landscape of $F_{\epsilon}$ crucially depends on the smoothing parameter $\epsilon$. Instead of minimizing $F_{\epsilon_k}$ directly, our method minimizes, for $k \in \N$, $\epsilon_k >0$ and $\f{X}\hk$ a \emph{quadratic model} 
\[
\begin{split}
Q_{\epsilon_k}(\f{X}|\f{X}^{(k)}) &= F_{\epsilon_k}(\f{X}^{(k)}) + \langle \nabla F_{\epsilon_k}(\f{X}^{(k)}), \f{X} - \f{X}^{(k)}\rangle \\
&+ \frac{1}{2} \langle \f{X}-\f{X}^{(k)},W^{(k)}(\f{X}-\f{X}^{(k)}) \rangle
\end{split}
\]
under the data constraint $P_{\Omega}(\f{X})=P_{\Omega}(\f{X}^0)$, where $W^{(k)}$ is the following operator, describing the precise shape of the quadratic model.
\begin{definition}[Optimal weight operator] \label{def:optimalweightoperator}
Let $\epsilon_k > 0$ and $\f{X}\hk \in \Rdd$ be a matrix with singular value decomposition $\f{X}\hk =  \f{U}_k \dg(\sigma^{(k)}) \f{V}_k^{*}$, i.e., $\f{U}_k \in \R^{d_1 \times d_1}$ and $\f{V}_k \in \R^{d_2 \times d_2}$ are orthonormal matrices. Then we call the linear operator $W\hk: \Rdd \to \Rdd$ the \emph{optimal weight operator} of the $\epsilon_k$-smoothed log-det objective $F_{\epsilon_k}$ of \eqref{eq:smoothing:Fpeps} at $\f{X}\hk$ if for $\f{Z} \in \Rdd$,
\begin{equation} \label{eq:def:W}
	W^{(k)}(\f{Z}) = \f{U}_k \left[\f{H}_k \circ (\f{U}_k^{*} \f{Z} \f{V}_k)\right] \f{V}_k^{*},
\end{equation}
where  $\f{H}_k \in \R^{d_1 \times d_2}$ is a matrix with positive entries such that
$(\f{H}_k)_{ij} := \Big(\max(\sigma_i^{(k)},\epsilon_k) \max(\sigma_j^{(k)},\epsilon_k)\Big)^{-1}$
and $\f{H}_k \circ (\f{U}_k^{*} \f{Z} \f{V}_k)$ denotes the entrywise product of $\f{H}_k$ and $\f{U}_k^{*} \f{Z} \f{V}_k$.
\end{definition}
The weight operator $W^{(k)}$ is a positive, self-adjoint operator with strictly positive eigenvalues that coincide with the entries of the matrix $\f{H}_k \in \Rdd$, and it is easy to verify that $W^{(k)}(\f{X}\hk) =  \nabla F_{\epsilon_k}(\f{X}\hk)$. Based on this, it follows that the minimization of the quadratic model $Q_{\epsilon_k}(\f{X}|\f{X}^{(k)})$ boils down to a minimization of a quadratic form weighted by $W\hk$. This enables us to design the iterative method \emph{Matrix Iteratively Reweighted Least Squares (MatrixIRLS)}, which we describe in \Cref{algo:MatrixIRLS}.

\begin{algorithm}[tb]
\caption{\texttt{MatrixIRLS} for low-rank matrix completion} \label{algo:MatrixIRLS}
\begin{algorithmic}
\STATE{\bfseries Input:}  Set $\Omega$, observations $\f{y} \in \R^m$, rank estimate $\widetilde{r}$.
\STATE Initialize $k=0$, $\epsilon^{(0)}=\infty$ and $W^{(0)} = \Id$.

\FOR{$k=1$ to $K$}
\STATE \textbf{Solve weighted least squares:} Use a \emph{conjugate gradient method} to solve
\begin{equation} \label{eq:MatrixIRLS:Xdef}
\f{X}^{(k)} =\argmin\limits_{\f{X}:P_{\Omega}(\f{X})=\f{y}} \langle \f{X}, W^{(k-1)}(\f{X}) \rangle.
\end{equation}
\STATE \textbf{Update smoothing:} \label{eq:MatrixIRLS:bestapprox} Compute 	$\widetilde{r}+1$-th singular value of $\f{X}\hk$ to update
\begin{equation} \label{eq:MatrixIRLS:epsdef}
\epsilon_k=\min\left(\epsilon_{k-1},\sigma_{\widetilde{r}+1}(\f{X}\hk)\right).
\end{equation}

\STATE \textbf{Update weight operator:} For $r_k := |\{i \in [d]: \sigma_i(\f{X}\hk) > \epsilon_k\}|$, compute the first $r_k$ singular values $\sigma_i\hk := \sigma_i(\f{X}\hk)$ and matrices $\f{U}\hk \in \R^{d_1 \times r_k}$ and $\f{V}\hk \in \R^{d_2 \times r_k}$ with leading $r_k$ left/ right singular vectors of $\f{X}\hk$ to update $W\hk$ defined in \Cref{eq:def:W}. \label{eqdef:Wk} 

\ENDFOR
\STATE{\bfseries Output:} $\f{X}^{(K)}$.
\end{algorithmic}
\end{algorithm}

Apart from the weighted least squares step \eqref{eq:MatrixIRLS:Xdef}, which minimizes the quadratic model $Q_{\epsilon_{k-1}}(\cdot|\f{X}^{(k-1)})$ of $F_{\epsilon_{k-1}}$ for fixed $\epsilon_{k-1}$, an indispensable ingredient of our scheme is the \emph{update of the smoothing parameter $\epsilon_k$}, which is performed in the spirit of smoothing methods for non-smooth objectives \cite{Chen2012smoothing}. In particular, the update rule \eqref{eq:MatrixIRLS:epsdef}, which is similar to the update rule of \cite{KS18}, makes sure that if the rank estimate $\widetilde{r}$ is chosen such that $\widetilde{r} \geq r$, the smoothing parameter $\epsilon_k$ converges to $0$ as the iterates approach a rank-$r$ solution. 

We note that \emph{Iteratively Reweighted Least Squares (IRLS)} methods with certain similarities to \Cref{algo:MatrixIRLS} had been proposed \cite{Fornasier11,Mohan10,KS18} for the minimization of Schatten-$p$ quasi-norms for $0 < p \leq 1$. Comparing the gradients of smoothed Schatten-$p$ quasi-norms and of \eqref{eq:smoothing:Fpeps}, minimizing a smoothed log-det objective can be considered as a limit case for $p \to 0$. Most importantly, however, our algorithm has two distinct, conceptual differences compared to these methods: Firstly, the weight operator of \Cref{def:optimalweightoperator} is able capture the \emph{second-order information} of $F_{\epsilon_k}$, allowing for an interpretation of \texttt{MatrixIRLS} as a saddle-escaping smoothing Newton method, cf. \Cref{sec:Newton:interpretation}, unlike the methods of \cite{Fornasier11,Mohan10,KS18} due to the different structure of their weight operators. Secondly, the interplay of $F_{\epsilon_k}$ and the weight operator $W\hk$ in \Cref{algo:MatrixIRLS} is designed to allow for efficient numerical implementations, cf. \Cref{sec:computational:complexity}. 

Finally, we note that it is non-trivial to show that the quadratic model $Q_{\epsilon_k}(\cdot|\f{X}^{(k)})$ induced by $W\hk$ from \Cref{def:optimalweightoperator} is actually a \emph{majorant} of $F_{\epsilon_k}(\cdot)$ such that $F_{\epsilon_k}(\f{X}) \leq Q_{\epsilon_k}(\f{X}|\f{X}^{(k)})$ for all $\f{X} \in \Rdd$. We defer a proof of this and a proof of the ``optimality'' of the majorant to an upcoming paper.

\section{Computational Complexity} \label{sec:computational:complexity}
A crucial property of \Cref{algo:MatrixIRLS} is that due to the structure of the weight operator \eqref{eq:def:W} and the smoothing update rule \eqref{eq:MatrixIRLS:epsdef}, in fact, the weighted least squares step \eqref{eq:MatrixIRLS:Xdef} can be computed by solving a positive definite linear system of size $(r_k (d_1 +d_2 -r_k)) \times (r_k (d_1 +d_2 -r_k))$, where $r_k$ is the number of singular values of $\f{X}\hk$ that are larger than $\epsilon_k$, which is typically equal or very close to $\widetilde{r}$ (cf. Appendix A). Conceptually, this corresponds to a linear system in the tangent space $T_k$ of the rank-$r_k$ matrix manifold at the best rank-$r_k$ approximation of $\f{X}\hk$,
$
T_k= \left\{ \begin{bmatrix} \f{U}\hk\!\!\! &\!\! \!\f{U}_{\perp}\hk \end{bmatrix} \!\!  \begin{bmatrix} \R^{r_k \times r_k} \!\!&\!\!\! \R^{r_k(d_2-r_k)} \\ \R^{(d_1-r_k)r_k}\! \!&\!\! \!\f{0} \end{bmatrix} \!\! \begin{bmatrix} \f{V}\hk \!\!\!&\!\!\! \f{V}_{\perp}\hk \end{bmatrix}^* \right\}.
$

We note that in our implementation, it is never necessary to compute more than $r_k$ singular vector pairs and singular values of $\f{X}\hk$, and $\f{X}\hk$ can be represented as a sum of a sparse matrix and a matrix in $T_k$, cf. \Cref{thm:MatrixIRLS:computationalcost:Xkk}.
Thus, when using an iterative solver such as \emph{conjugate gradients} to solve the linear system, we obtain an implementation of \texttt{MatrixIRLS} with a time and space complexity of the same order as for state-of-the-art first-order algorithms based on matrix factorization (i.e., of Burer-Monteiro type) \cite{chen_chi18}. We refer to the supplementary materials (Appendix A) for details and a proof.
\begin{theorem} \label{thm:MatrixIRLS:computationalcost:Xkk}
	Let $\f{X}\hk \in \Rdd$ be the $k$-th iterate of \texttt{MatrixIRLS} for an observation vector $\f{y} \in \R^m$ and $\widetilde{r}=r$. Assume that $\sigma_i\hk \leq \epsilon_k$ for all $i > r$ and $\sigma_r\hk > \epsilon_k$. Then an implicit representation of the new iterate $\f{X}\hkk \in \R^{d_1 \times d_2}$ can be calculated in a \emph{time complexity} of 
	\[
	 O \left( (m r + r^2 D) \cdot N_{\text{CG\_inner}} \right),
	\]
where $N_{\text{CG\_inner}}$ is the number of inner iterations used in the conjugate gradient method and $D=\max(d_1,d_2)$. More precisely, $\f{X}\hkk$ can be represented as
\[
\f{X}\hkk = P_{\Omega}^*(\f{r}_{k+1})  +\f{U}\hk \f{M}_{1}^{(k+1)*} + \f{M}_2^{(k+1)} \f{V}^{(k)*},
\]
where $\f{r}_{k+1} \in \R^m$, $\f{M}_{1}^{(k+1)} \in \R^{d_2 \times r}$ and $\f{M}_2^{(k+1)} \in \R^{d_1 \times r}$,
i.e., with a \emph{space complexity} of $O ( m+ r D)$.
\end{theorem}

\Cref{thm:MatrixIRLS:computationalcost:Xkk} illustrates the computational advantage of \texttt{MatrixIRLS} compared to previous iteratively reweighted least squares algorithms for low-rank matrix recovery problems \cite{Fornasier11,Mohan10,KS18}, which all require the storage and updates of full $(d_1 \times d_2)$-matrices and the calculation of singular value decompositions of these.

According to \Cref{thm:MatrixIRLS:computationalcost:Xkk}, since $P_{\Omega}^*(\f{r}_{k+1}) \in \Rdd$ is $m$-sparse, $\f{X}\hkk$ can be seen a sum of a sparse and two rank-$r$ matrices. Intuitively, this representation is possible as the weight operator $W^{(k)}$ of \Cref{def:optimalweightoperator} can be written as ``identity + diagonal on $T_k$'', and due to the Sherman-Morrison-Woodbury formula applied to the inverse in $\f{X}\hkk =  (W\hk)^{-1}P_{\Omega}^*\left(P_{\Omega}(W\hk)^{-1} P_{\Omega}^*\right)^{-1}(\f{y})$, which is an explicit representation of the solution of \eqref{eq:MatrixIRLS:Xdef}. 

As a result, fast matrix-vector multiplications can be used in methods such as Lanczos bidiagonalization or randomized Block Krylov \cite{MuscoMusco15} to compute $r_{k+1}$ singular values and vectors of $\f{X}\hkk$ in step 3 of \Cref{algo:MatrixIRLS}.

\section{Theoretical Analysis} \label{sec:convergence}
This section sheds light on several theoretical aspects of \Cref{algo:MatrixIRLS}.
\subsection{Local Convergence with Superlinear Rate \& Conditioning of System Matrix}
In order to obtain a theoretical understanding of the generic behavior of \texttt{MatrixIRLS}, we consider the canonical uniform random sampling model \cite{CR09,recht,Chen15} where the sampling set  $\Omega = (i_{\ell}, j_{\ell})_{\ell=1}^m \subset [d_1] \times [d_2]$ consists of $m$ double indices that are drawn uniformly at random without replacement. Not \emph{each} rank-$r$ matrix $\f{X}^0 \in \Rdd$ is expected to be identifiable from a small number of samples $m$ under this sampling model. We quantify the alignment of a matrix with the standard basis of $\Rdd$ by the following notion of incoherence, which is slightly weaker than related conditions of \cite{recht,Chen15}.
 \begin{definition}\label{def:incoherence}
We say that a rank-$r$ matrix $\f{X} \in \Rdd$ with singular value decomposition $\f{X} = \f{U} \dg(\sigma) \f{V}^*$, $\f{U} \in \R^{d_1 \times r}$, $\f{V} \in \R^{d_2 \times r}$, is \emph{$\mu_0$-incoherent} if there exists a constant $\mu_0 \geq 1$ such that
\begin{equation}\label{eq:MatrixIRLS:incoherence}
\max_{1 \leq i \leq d_1, 1 \leq j \leq d_2} \|\mathcal{P}_T (e_i e_j^*)\|_F   \leq  \sqrt{\mu_0 r \frac{d_1 + d_2}{d_1 d_2}}, 
\end{equation}
where $T = T_{\f{X}} =  \{ \f{U} \f{M}^* + \widetilde{\f{M}} \f{V}^* :~ \f{M} \in \R^{d_2 \times r}, ~ \widetilde{\f{M}} \in \R^{d_1 \times r} \}$ is the tangent space onto the rank-$r$ matrix manifold at $\f{X}$ and $\mathcal{P}_T$ is the projection operator onto $T$ .
\end{definition}
With the notation that $\|\f{X}\|_{S_\infty}= \sigma_1(\f{X})$ denotes the \emph{spectral norm} or \emph{Schatten-$\infty$ norm} of a matrix $\f{X}$, we obtain the following local convergence result.
\begin{theorem}[Local convergence of \texttt{MatrixIRLS} with Quadratic Rate] \label{cor:MatrixIRLS:localconvergence:matrixcompletion}
Let $\f{X}^0 \in \Rdd$ be a matrix of rank $r$ that is $\mu_0$-incoherent, and let $P_\Omega: \Rdd  \rightarrow  \R^{m}$ be the subsampling operator corresponding to an index set $\Omega=(i_{\ell},j_{\ell})_{\ell=1}^m \subset [d_1] \times [d_2]$ that is drawn uniformly without replacement. If the sample complexity fulfills
$m \gtrsim \mu_0 r (d_1 + d_2) \log(d_1 + d_2)$,
then with high probability, the following holds: If the output matrix $\f{X}^{(k)}\in \Rdd$ of the $k$-th iteration of \texttt{MatrixIRLS} with inputs $P_{\Omega}$, $\f{y}=P_{\Omega}(\f{X}^0)$ and $\widetilde{r}=r$ updates the smoothing parameter in \eqref{eq:MatrixIRLS:epsdef} such that $\epsilon_k = \sigma_{r+1}(\f{X}\hk)$ and fulfills
\begin{equation} \label{eq:MatrixIRLS:closeness:assumption}
\|\f{X}\hk - \f{X}^0\|_{S_{\infty}} \lesssim \min\left( \sqrt{\frac{\mu_0 r}{d}}, \frac{\mu_0}{d \log(D) \kappa}\right)\sigma_r(\f{X}^0),
\end{equation}
where $\kappa=\sigma_1(\f{X}^0)/\sigma_r(\f{X}^0)$, then the \emph{local convergence rate is quadratic} in the sense that 
 $\|\f{X}^{(k+1)}-\f{X}^0\|_{S_{\infty}}\leq \min(\mu \|\f{X}\hk-\f{X}^0\|_{S_{\infty}}^{2} , \|\f{X}\hk-\f{X}^0\|_{S_{\infty}}) $ with $\mu \leq \frac{ d \log(D) }{\mu_0 \sigma_r(\f{X}^0)} \kappa$, and furthermore $\f{X}^{(k+\ell)} \xrightarrow{\ell \to \infty} \f{X}^0$ if additionally $\|\f{X}\hk - \f{X}^0\|_{S_{\infty}} \lesssim  \min\left( \sqrt{\frac{\mu_0 r}{d}}, \frac{\mu_0^{3/2} r^{1/2}}{d^2 \log(D)^{3/2} \kappa}\right)\sigma_r(\f{X}^0)$.
\end{theorem}
While a comparable local convergence result had been obtained for an IRLS algorithm for (non-convex) Schatten-$p$ minimization \cite{KS18}, that result is \emph{not} applicable for matrix completion, as the proof relied on a \emph{null space property} \cite{Recht11} of the measurement operator, which is not fulfilled by $P_{\Omega}$ since there are always rank-ones matrices in the null space of the entry-wise operator $P_{\Omega}$.

Unlike the theory of other algorithms, the sample complexity assumption of \Cref{cor:MatrixIRLS:localconvergence:matrixcompletion} is \emph{optimal} as it matches a well-known lower bound for this sampling model \cite{CandesTao10} that is necessary for unique identifiability. Among the weakest sufficient conditions for existing algorithms are $m \gtrsim \mu_0 r (d_1+d_2) \log^2(d_1 +d_2)$ for nuclear norm minimization \cite{Chen15}, $m \gtrsim \mu_0 \kappa^{14} r^2 (d_1+d_2) \log^2(d_1 +d_2)$ for gradient descent \cite{ChenLiuLi19} on a variant of \eqref{eq:UV:recommender} and $m \gtrsim \kappa^6 (d_1+d_2) r^2 \log (d_1+d_2)$ required random samples for the Riemannian gradient descent algorithm of \cite{wei_cai_chan_leung}.
On the other hand, in contrast to other results, \Cref{cor:MatrixIRLS:localconvergence:matrixcompletion} only quantifies \emph{local} convergence.

The following theorem implies that iterative solvers are indeed able to efficiently solve the linear system underlying \eqref{eq:MatrixIRLS:Xdef} up to high accuracy in few iterations. It suggests that $N_{\text{CG\_inner}}$ of \Cref{thm:MatrixIRLS:computationalcost:Xkk} can be chosen as an absolute constant.

\begin{theorem}[Well-conditioning of system matrices of \texttt{MatrixIRLS}] \label{thm:wellconditioning}
In the setup and sampling model of \Cref{thm:MatrixIRLS:computationalcost:Xkk}, if $m \gtrsim \mu_0 r (d_1 + d_2) \log(d_1 + d_2)$, the following holds with high probability: If $\epsilon_k=\sigma_{r+1}(\f{X}\hk) < \sigma_{r}(\f{X}\hk)$ and if $\|\f{X}\hk - \f{X}^0\|_{S_{\infty}} \lesssim \min \left(\sqrt{\frac{\mu_0 r}{d}}, \frac{1}{4}\right) \sigma_r(\f{X}^0)$, the spectrum $\lambda(\f{A}_k)$ of the linear system matrix $\f{A}_k \in \R^{r(d_1+d_2-r) \times r(d_1+d_2-r)}$ of the weighted least squares step \eqref{eq:MatrixIRLS:Xdef} of \texttt{MatrixIRLS} satisfies  $\lambda(\f{A}_k) \subset \frac{m}{d_1 d_2}\left[\frac{6}{10}; \frac{24}{10}\right]$, and thus, the condition number of $\f{A}_k$ fulfills
$\kappa(\f{A}_k) \leq 4$.
\end{theorem}

\Cref{thm:wellconditioning} shows that \texttt{MatrixIRLS} is able to overcome a common problem of many IRLS algorithms for related problems: Unlike the methods of \cite{Daubechies10,Fornasier16,Mohan10,Fornasier11,KS18}, does not suffer from ill-conditioned linear systems close to a low-rank (or sparse) solution.

\subsection{MatrixIRLS as saddle-escaping smoothing Newton method} \label{sec:Newton:interpretation}
From a theoretical point of view, the local quadratic convergence rate is an inherently local property that does not explain the numerically observed global convergence behavior (see \Cref{sec:numerics}), which is remarkable due to the non-convexity of the objective function.

A possible avenue to explain this is to interpret \texttt{MatrixIRLS} as a \emph{saddle-escaping smoothing Newton method}. Smoothing Newton methods minimize a non-smooth and possibly non-convex function $F$ by using derivatives of certain smoothings of $F$ \cite{ChenQiSun98,Chen2012smoothing}. Interpreting the optimization problem
$\min_{\f{X}:P_{\Omega}(X)=\f{y}} F_{\epsilon_k}(\f{X})$
as an unconstrained optimization problem over the null space of $P_{\Omega}$, we can write
\begin{equation*}
    \begin{split}
&\f{X}^{(k+1)} = \f{X}^{(k)} - P_{\Omega^c}^* \left(P_{\Omega^c} W\hk P_{\Omega^c}^*\right)^{-1} P_{\Omega^c} W\hk (\f{X}\hk) \\
	&\!\!= \f{X}^{(k)}\!\! -\!P_{\Omega^c}^* \left(P_{\Omega^c}\!  \overline{\nabla^2 F_{\epsilon_k}(\f{X}\hk)}\! P_{\Omega^c}^*\right)^{-1}\!\!\!\!\!\! P_{\Omega^c} \nabla F_{\epsilon}(\f{X}\hk),
\end{split}
\end{equation*}
if $\Omega^c = [d_1] \times [d_2] \setminus \Omega$ corresponds to the unobserved indices, where $\overline{\nabla^2 F_{\epsilon_k}(\f{X}\hk)}: \Rdd \to \Rdd$ is a \emph{modified} Hessian of $F_{\epsilon_k}$ at $\f{X}\hk$ that replaces negative eigenvalues of the Hessian $\nabla^2 F_{\epsilon_k}(\f{X}\hk)$ by positive ones and slightly increases small eigenvalues. We refer to the supplementary material for more details. In \cite{PaternainMokhtariRibeiro19}, it has been proved that for a fixed smooth function $F_{\epsilon_k}$, similar modified Newton-type steps are able to escape the first-order saddle points at a rate that is independent of the problem's condition number.

\subsection{MatrixIRLS as variable metric forward-backward method}
Another instructive angle to understand our method comes from the framework of \emph{variable metric forward-backward methods} \cite{Bonnans95,Chouzenoux14,Frankel15}. 

A forward-backward method can be seen as a combination of a gradient descent method and a proximal point algorithm \cite{Combettes2011} that can be used to minimize the sum of a non-smooth function and a function with Lipschitz continuous gradients. In particular, if $F$ is a proper, lower semi-continuous function, $G$ is differentiable with Lipschitz gradient $\nabla G$ and $(\alpha_k)_{k}$ a sequence of step sizes, the iterations of the forward-backward algorithm \cite{Attouch13} are such that
$\f{X}\hkk \in \prox_{\alpha_k F} \left(\f{X}\hk - \alpha_k \nabla  G(\f{X}\hk) \right)$, where $\prox_{\alpha_k F}(\cdot)$ is the proximity operator of $\alpha_k F$. Typically, in such an algorithm, $F$ would be chosen as the structure-promoting objective (such as the smoothed log-det objective $F_{\epsilon}$ above) and $G$ as a data-fit term such as $G(\f{X}) = \|P_{\Omega}(\f{X})-\f{y}\|_2^2/\lambda$, leading to thresholding-type algorithms.
\Cref{algo:MatrixIRLS}, however, fits into this framework if we choose, for $\epsilon_k >0$, the non-smooth part $F$ as the indicator function $F:=\chi_{P_{\Omega}^{-1}(\f{y})}:\Rdd \to \R$ of the constraint set $P_{\Omega}^{-1}(\f{y}):=\{\f{X} \in \Rdd : P_{\Omega}(\f{X})=\f{y}\}$ and the smooth part $G$ such that $G:=F_{\epsilon_k}:\Rdd \to \R$ as in \eqref{eq:smoothing:Fpeps}, while offsetting the distortion induced by the non-Euclidean nature of the level sets of $F_{\epsilon_k}$ via an appropriate choice of a \emph{variable metric} $d_{A_k}(\f{X},\f{Z})=\sqrt{\langle \f{X}-\f{Z},A_k(\f{X}-\f{Z})\rangle_F}$ for a positive definite linear operator $A_k: \Rdd \to \Rdd$, such that
\[
\f{X}\hkk \in \prox_{\alpha_k F}^{A_k} \left(\f{X}\hk - \alpha_k  A_k^{-1}(\nabla  G(\f{X}\hk)) \right),
\]
where $\prox_{F}^{A_k}(\f{X}):= \argmin_{\f{Z}\in\Rdd} F(\f{Z}) + \frac{1}{2} d_{A_k}(\f{X},\f{Z})^2 $ is the proximity operator of $F$ \emph{scaled in the metric $d_{A_k}$} at $\f{X}$ \cite{Chouzenoux14}. Specifically, if we choose the metric induced by the weight operator of \eqref{eq:def:W} such that $A_k := W\hk$ and unit step sizes $\alpha_k = 1$, we obtain
\[
\begin{split}
&\prox_{\alpha_k F}^{A_k} \left(\f{X}\hk - \alpha_k  A_k^{-1}(\nabla  G(\f{X}\hk)) \right) \\
&= \prox_{\chi_{P_{\Omega}^{-1}}}^{W\hk} \left(\f{X}\hk - W_k^{-1}(\nabla  F_{\epsilon_k}(\f{X}\hk)) \right) \\
&= \prox_{\chi_{P_{\Omega}^{-1}}}^{W\hk} \left(\f{X}\hk - W_k^{-1}W_k(\f{X}\hk)) \right) = \prox_{\chi_{P_{\Omega}^{-1}}}^{W\hk} \left(\f{0} \right)  \\
&= \argmin_{\f{X}: P_{\Omega}(\f{X})=\f{y}} \frac{1}{2} d_{A_k}(\f{X},\f{0})^2 = \argmin_{\f{X}: P_{\Omega}(\f{X})=\f{y}}  \langle \f{X},W\hk(\f{X})\rangle,
\end{split}
\]
where we used that $W_k(\f{X}\hk) = \nabla F_{\epsilon_k}(\f{X}\hk)$ in the third line. This shows that this update rule for $\f{X}\hkk$ coincides with \eqref{eq:MatrixIRLS:Xdef}.

Thus, \texttt{MatrixIRLS} can be considered as a forward-backward method with a variable metric induced by the weight operator $W\hk$, using a unit step size $\alpha_k=1$ for each $k$. One advantage of our method is therefore also that unlike many methods in this family, there is no step size to be tuned. A crucial difference, which makes existing theory (as, e.g., \cite{Frankel15}) for splitting methods not directly applicable for a convergence analysis of \texttt{MatrixIRLS}, is that the smooth function $G=F_{\epsilon_k}$ \emph{is changing} at each iteration due to the smoothing parameter update \eqref{eq:MatrixIRLS:epsdef}. On the other hand, the results of \cite{Frankel15} already imply the finite sequence length of $(\f{X}\hk)_{k}$ in the case that the smoothing parameter $\epsilon_k$ stagnates for $k \geq k_0$, using a Kurdyka-\L ojasiewicz property \cite{BolteKL07} of $F_{\epsilon_k} + \chi_{P_{\Omega}^{-1}(\f{y})}$. We leave a detailed discussion of this for future work. 

Finally, we note that previous IRLS methods \cite{Fornasier11,Mohan10,KS18} would also fit in the presented splitting framework, however, without fully capturing the underlying geometry as their weight operator has no strong connection to the Hessian $\nabla^2 F_{\epsilon_k}(\f{X}\hk)$ of $F_{\epsilon_k}$, as explained in the supplementary material.

\section{Numerical Experiments} \label{sec:numerics}
We explore the performance of \texttt{MatrixIRLS} for the completion of synthetic low-rank matrices in terms of statistical and computational efficiency in comparison to state-of-the-art algorithms in the literature. 
We base our choice on the desire to obtain a representative picture of state-of-the-art algorithms for matrix completion, including in particular those that are scalable to problems with dimensionality in the thousands or more, those that come with the best theoretical guarantees, and those that claim to perform particularly well to complete \emph{ill-conditioned} matrices. All the methods are provided with the true rank $r$ of $\f{X}^0$ as an input parameter. If possible, we use the MATLAB implementation provided by the authors of the respective papers.

The algorithms being tested against \texttt{MatrixIRLS} can be grouped into three main categories: the non-convex matrix factorization ones which includes \texttt{LMaFit} \cite{Wen12},  \texttt{ScaledASD} \cite{TannerWei16} and \texttt{ScaledGD} \cite{tong_ma_chi}, the Riemannian optimization on the manifold of fixed rank matrices ones which includes \texttt{LRGeomCG} \cite{Vandereycken13}, \texttt{RTRMC} \cite{boumal_absil_15} and \texttt{R3MC} \cite{MishraS14}, one alternating projection method on the manifold of fixed rank matrices, \texttt{NIHT} \cite{TannerWei13} (see \cite{wei_cai_chan_leung} for a connection between NIHT and Riemannian methods), and the recent \texttt{R2RILS} \cite{BauchNadler20} which can be seen as a factorization based method but also contains ideas from the Riemannian optimization family of algorithms. In the supplementary material we provide a description of each algorithm as well as the parameters used in the numerical section. 

As for the numerical experiments, it is important to note that we are interested to find low-rank completions from a sampling set $\Omega$ of sample size $|\Omega|=: m = \lfloor \rho r (d_1 +d_2 - r) \rfloor$, where $\rho$ is an oversampling ratio since $r (d_1 +d_2 - r)$ is just the number of degrees of freedom of an $(d_1 \times d_2)$-dimensional rank-$r$ matrix. For a given $\Omega$, the solution of \eqref{rank_equation} might \emph{not} coincide with $\f{X}^0$, or the solution might not be unique, even if the sample set $\Omega$ is chosen uniformly at random. In particular, this will be the case if $\Omega$ is such that there is a row or a column with \emph{fewer than} $r$ revealed entries, which a necessary condition for uniqueness of the \eqref{rank_equation} \cite{Pimentel15}. To mitigate this problem that is rather related to the structure of the sampling set than to the performance of a certain algorithm, we, in fact, adapt the sampling model of uniform sampling without replacement. For a given factor $\rho \geq 1$, we sample a set $\Omega \subset [d_1] \times [d_2]$ of size $m = \lfloor \rho r (d_1 +d_2 - r) \rfloor$ indices randomly without replacement. Then we check whether the condition such that each row and each column in $\Omega$ has at least $r$ observed entries, and resample $\Omega$ if this condition is not fulfilled. This procedure is repeated up to a maximum of $1000$ resamplings.

We consider the following setup: we sample a pair of random matrices $\f{U} \in \R^{d_1 \times r}$ and $\f{V} \in \R^{d_2 \times r}$ with $r$ orthonormal columns, and define the diagonal matrix $\Sigma \in \R^{r \times r}$ such that $\Sigma_{ii} = \kappa \exp(-\log(\kappa)\frac{i-1}{r-1})$ for $i \in [r]$. With this definition, we define a ground truth matrix $\f{X}^0 =\f{U}\Sigma \f{V}^*$ of rank $r$ that has exponentially decaying singular values between $\kappa$ and $1$. 

\subsection{Data-efficient recovery of ill-conditioned matrices}\label{numerics:data_efficiency}
\normalsize
First, we run \texttt{MatrixIRLS} and the algorithms \texttt{R2RILS} , \texttt{RTRMC}, \texttt{LRGeomCG}, \texttt{LMaFit}, \texttt{ScaledASD}, \texttt{ScaledGD}, \texttt{NIHT} and \texttt{R3MC} to complete $\f{X}^0$ from $P_{\Omega}(\f{X}^0)$ where $\Omega$ corresponds to different oversampling factors $\rho$ between $1$ and $4$, and where the condition number of $\f{X}^0$ is $\kappa= \sigma_1(\f{X}^0)/\sigma_r(\f{X}^0) = 10$. In \Cref{fig:sampcomp:1}, we report the median Frobenius errors $\|\f{X}^{(K)}-\f{X}^0\|_F/\|\f{X}^0\|_F$ of the respective algorithmic outputs $\f{X}^{(K)}$ across $100 $ independent realizations. 

\begin{figure}[h]
    \setlength\figureheight{45mm} 
    \setlength\figurewidth{120mm}
%
%
\definecolor{mycolor1}{rgb}{0.00000,0.44700,0.74100}%
\definecolor{mycolor2}{rgb}{0.85000,0.32500,0.09800}%
\definecolor{mycolor3}{rgb}{0.92900,0.69400,0.12500}%
\definecolor{mycolor4}{rgb}{0.49400,0.18400,0.55600}%
\definecolor{mycolor5}{rgb}{0.46600,0.67400,0.18800}%
\definecolor{mycolor6}{rgb}{0.30100,0.74500,0.93300}%
\definecolor{mycolor7}{rgb}{0.63500,0.07800,0.18400}%
\definecolor{mycolor8}{rgb}{0.08000,0.39200,0.25100}%
\begin{tikzpicture}

\begin{axis}[%
width=0.978\figurewidth,
height=\figureheight,
at={(0\figurewidth,0\figureheight)},
scale only axis,
unbounded coords=jump,
xmin=1,
xmax=4,
xlabel style={font=\color{white!15!black}},
xlabel={Oversampling factor $\rho$},
ymode=log,
ymin=1e-15,
ymax=100000,
yminorticks=true,
ylabel style={font=\color{white!15!black}},
ylabel={Median of rel. Frob. errors of $\mathbf{X}^{(K)}$},
axis background/.style={fill=white},
legend style={legend cell align=left, align=left, draw=white!15!black},
xlabel style={font=\tiny},ylabel style={font=\tiny},legend style={font=\fontsize{7}{30}\selectfont, anchor=south, legend columns = 5, at={(0.5,1.03)}}
]
\addplot [color=mycolor1, line width=1.0pt, mark=x, mark options={solid, mycolor1}]
  table[row sep=crcr]{%
1.05	0.854255555628291\\
1.1	0.749423097099185\\
1.2	0.494716881980835\\
1.3	0.307031095563721\\
1.4	0.0710950413951957\\
1.5	5.22910462862914e-13\\
1.6	1.68588918367454e-13\\
1.7	7.0339665885088e-14\\
1.8	5.22646777788465e-14\\
1.9	3.12007321555135e-14\\
2	2.26159575881266e-14\\
2.1	1.84275579519234e-14\\
2.2	1.20010224139451e-14\\
2.3	9.65828095082859e-15\\
2.4	8.33489396563835e-15\\
2.5	5.87405834953726e-15\\
2.6	7.06122879837332e-15\\
2.7	6.140965219143e-15\\
2.8	5.15640770569221e-15\\
2.9	5.24064290362918e-15\\
3	4.58553528727644e-15\\
3.1	4.81251473302107e-15\\
3.2	4.03701051254568e-15\\
3.3	5.28365319386791e-15\\
3.4	4.07526625097761e-15\\
3.5	3.54126224418117e-15\\
3.6	2.79718832356394e-15\\
3.7	2.81718954270358e-15\\
3.8	2.69583211552255e-15\\
3.9	2.82440163639793e-15\\
4	2.46058369466678e-15\\
};
\addlegendentry{MatrixIRLS}

\addplot [color=mycolor2, line width=1.0pt, mark=+, mark options={solid, mycolor2}]
  table[row sep=crcr]{%
1.05	2.26456245685704\\
1.1	2.5024356292252\\
1.2	2.43235945106256\\
1.3	2.03349200850792\\
1.4	1.21611319874659\\
1.5	1.23751335798688e-06\\
1.6	5.20820599960089e-09\\
1.7	2.4080827102596e-11\\
1.8	7.35611767992492e-14\\
1.9	1.14972497705595e-14\\
2	1.14753038353458e-14\\
2.1	1.00726730331893e-14\\
2.2	7.34438613648422e-15\\
2.3	7.96384643655968e-15\\
2.4	7.79328724569884e-15\\
2.5	1.00231206491362e-14\\
2.6	1.3546830583003e-14\\
2.7	1.31005965452693e-14\\
2.8	1.35392160987976e-14\\
2.9	9.22728400705331e-15\\
3	1.77220221973004e-14\\
3.1	1.07494351508938e-14\\
3.2	1.04816978255403e-14\\
3.3	1.55757331921461e-14\\
3.4	1.83155751788827e-14\\
3.5	1.41801812290188e-14\\
3.6	3.49996967411273e-14\\
3.7	2.38941709817596e-14\\
3.8	1.93508288410887e-14\\
3.9	2.6480581760627e-14\\
4	2.35937791312582e-14\\
};
\addlegendentry{R2RILS}

\addplot [color=mycolor3, line width=1.0pt, mark=asterisk, mark options={solid, mycolor3}]
  table[row sep=crcr]{%
1.05	1075.26057752175\\
1.1	1036.88076801983\\
1.2	898.187957850437\\
1.3	676.778169415698\\
1.4	684.417922044455\\
1.5	711.785063703726\\
1.6	725.442000383892\\
1.7	751.035923142478\\
1.8	651.905082795613\\
1.9	629.404587607984\\
2	441.777176221888\\
2.1	429.821955824141\\
2.2	343.56802178822\\
2.3	1.15941912323249e-06\\
2.4	1.04551680418991e-07\\
2.5	1.27275200027277e-07\\
2.6	1.14870558919892e-07\\
2.7	1.05903565530647e-07\\
2.8	8.54610944390847e-09\\
2.9	4.96007271424229e-08\\
3	2.61470122950561e-08\\
3.1	8.86485982905778e-09\\
3.2	1.49412147034074e-08\\
3.3	6.46692958302814e-09\\
3.4	2.4700422542626e-08\\
3.5	1.03709991798648e-08\\
3.6	2.06415898459728e-08\\
3.7	9.18382240490974e-09\\
3.8	3.58229325687377e-09\\
3.9	2.52518262515746e-08\\
4	5.83585108699413e-09\\
};
\addlegendentry{RTRMC}

\addplot [color=mycolor4, line width=1.0pt, mark=o, mark options={solid, mycolor4}]
  table[row sep=crcr]{%
1.05	4.54218788203825\\
1.1	5.12790257479795\\
1.2	5.33019386122026\\
1.3	4.98053994216232\\
1.4	4.39615077981243\\
1.5	4.22488135311128\\
1.6	3.4144198382766\\
1.7	3.12003690973836\\
1.8	3.00097875387249\\
1.9	2.983446174957\\
2	2.60611156994507\\
2.1	2.36681710851521\\
2.2	1.71475814466035\\
2.3	1.71775747587247\\
2.4	1.20879362205861\\
2.5	0.764388428225149\\
2.6	0.91644120517575\\
2.7	2.39041820122551e-07\\
2.8	1.62071295437065e-07\\
2.9	1.50039926066069e-07\\
3	1.27835472461951e-07\\
3.1	1.12400902325457e-07\\
3.2	1.03200233643693e-07\\
3.3	9.61158837664088e-08\\
3.4	8.59690579442177e-08\\
3.5	8.40231752692976e-08\\
3.6	7.59751041476831e-08\\
3.7	7.16773678345403e-08\\
3.8	6.4853253141164e-08\\
3.9	6.01399241139154e-08\\
4	5.84948579668478e-08\\
};
\addlegendentry{LRGeomCG}

\addplot [color=mycolor5, line width=1.0pt, mark=x, mark options={solid, mycolor5}]
  table[row sep=crcr]{%
1.05	1.12746309687931\\
1.1	1.1215115543341\\
1.2	1.05005289952185\\
1.3	0.979491896633853\\
1.4	0.827263202581372\\
1.5	0.738733099621636\\
1.6	0.607355931786455\\
1.7	0.497511329328573\\
1.8	0.405136304159793\\
1.9	0.334586407456963\\
2	0.283213525406684\\
2.1	0.242367443226484\\
2.2	0.196034213853877\\
2.3	0.151317020176626\\
2.4	0.117156714237261\\
2.5	0.115895091936267\\
2.6	0.101974980837135\\
2.7	6.65749401699423e-09\\
2.8	5.01786293007783e-09\\
2.9	4.59006342081414e-09\\
3	4.28262136377922e-09\\
3.1	3.97593246298878e-09\\
3.2	3.84862100638697e-09\\
3.3	3.56396882136211e-09\\
3.4	3.44257715382382e-09\\
3.5	3.28634578928152e-09\\
3.6	3.22147596465297e-09\\
3.7	3.1379469525384e-09\\
3.8	2.97955273363032e-09\\
3.9	2.94591484325744e-09\\
4	2.8687371350171e-09\\
};
\addlegendentry{LMaFit}

\addplot [color=mycolor6, line width=1.0pt, mark=square, mark options={solid, mycolor6}]
  table[row sep=crcr]{%
1.05	1.06466416015086\\
1.1	1.06463314388243\\
1.2	1.03965187979647\\
1.3	0.960668507770738\\
1.4	0.824747589865423\\
1.5	0.745062952365437\\
1.6	0.604212717779343\\
1.7	0.533104871209988\\
1.8	0.434557035092144\\
1.9	0.388395943543961\\
2	0.305635793463275\\
2.1	0.284434146621849\\
2.2	0.235606018890938\\
2.3	0.208483454008394\\
2.4	0.182529631516212\\
2.5	0.165489912583052\\
2.6	7.05991430201768e-09\\
2.7	5.16997427925381e-09\\
2.8	4.59969613371997e-09\\
2.9	4.21095337333196e-09\\
3	4.01020980931519e-09\\
3.1	3.65962492755945e-09\\
3.2	3.51540016368054e-09\\
3.3	3.30205204705014e-09\\
3.4	3.24585144342233e-09\\
3.5	3.12736215079214e-09\\
3.6	3.06860576382731e-09\\
3.7	2.98223020227907e-09\\
3.8	2.81961243454812e-09\\
3.9	2.81855322114525e-09\\
4	2.70286224887889e-09\\
};
\addlegendentry{ScaledASD}

\addplot [color=mycolor7, line width=1.0pt, mark=diamond, mark options={solid, mycolor7}]
  table[row sep=crcr]{%
1.05	0.780342983376787\\
1.1	0.760559629534551\\
1.2	0.716937289667204\\
1.3	0.655527483255922\\
1.4	0.589124333833299\\
1.5	0.542132854431097\\
1.6	0.45418458358205\\
1.7	0.385560512899637\\
1.8	0.341913628226789\\
1.9	0.278746040604984\\
2	0.242730595784693\\
2.1	0.192126001507383\\
2.2	0.145452661065458\\
2.3	0.138749643187959\\
2.4	0.127851513423253\\
2.5	0.0834717206401568\\
2.6	0.0896687668276914\\
2.7	0.0318824695108248\\
2.8	0.0534563334227891\\
2.9	8.57397358047281e-08\\
3	1.1479597956224e-07\\
3.1	4.3746445456053e-09\\
3.2	4.05807191058956e-09\\
3.3	3.86537463377338e-09\\
3.4	3.63303397873685e-09\\
3.5	3.57399045753236e-09\\
3.6	3.43987811253563e-09\\
3.7	3.34572184308354e-09\\
3.8	3.1544837589476e-09\\
3.9	3.12171693602334e-09\\
4	2.97912161408277e-09\\
};
\addlegendentry{ScaledGD}

\addplot [color=mycolor8, line width=1.0pt, mark=triangle, mark options={solid, mycolor8}]
  table[row sep=crcr]{%
1.05	0.741178657713085\\
1.1	0.710345115228534\\
1.2	0.657484763497934\\
1.3	0.583841951688662\\
1.4	0.52969749416985\\
1.5	0.476354734852938\\
1.6	0.404884442048636\\
1.7	0.360820981001756\\
1.8	0.306935251980735\\
1.9	0.303708974319308\\
2	0.267456591481581\\
2.1	0.234331899499627\\
2.2	0.21993357166271\\
2.3	0.206373745234411\\
2.4	0.224257420213961\\
2.5	0.188961454098439\\
2.6	0.21827413665062\\
2.7	0.238458775410865\\
2.8	0.180849453808525\\
2.9	0.211087228946457\\
3	0.245008088121678\\
3.1	0.236380488132301\\
3.2	0.243076273505301\\
3.3	0.189558676300446\\
3.4	0.234927232199093\\
3.5	0.291938223692077\\
3.6	0.274422676296374\\
3.7	0.233688337092644\\
3.8	0.250113008520314\\
3.9	0.235759425673383\\
4	0.312027766249649\\
};
\addlegendentry{NIHT}

\addplot [color=black, line width=1.0pt, mark=triangle, mark options={solid, rotate=180, black}]
  table[row sep=crcr]{%
1.05	4.96665697678023\\
1.1	5.43839837093083\\
1.2	5.74675569745413\\
1.3	5.42493715826557\\
1.4	4.3646339396004\\
1.5	2.54550778470283\\
1.6	0.409610299286151\\
1.7	0.188531284845151\\
1.8	0.154461662572826\\
1.9	0.093072044579349\\
2	0.00171842201573833\\
2.1	1.22116099546181e-05\\
2.2	6.48779352867206e-08\\
2.3	2.72354271769887e-09\\
2.4	2.52573283283794e-09\\
2.5	2.37147759830555e-09\\
2.6	2.24853483384248e-09\\
2.7	2.1459391718041e-09\\
2.8	2.07617540982446e-09\\
2.9	2.01045576186298e-09\\
3	1.93093913884828e-09\\
3.1	1.88115521161772e-09\\
3.2	1.85045558627907e-09\\
3.3	1.78166679408017e-09\\
3.4	1.7295284998751e-09\\
3.5	1.69754239354936e-09\\
3.6	1.65662741520908e-09\\
3.7	1.63299624376917e-09\\
3.8	1.59336292953113e-09\\
3.9	1.54973257293995e-09\\
4	1.53229705822349e-09\\
};
\addlegendentry{R3MC}

\end{axis}
\end{tikzpicture}%
\caption{Performance of matrix completion algorithms for $1000 \times 1000$ matrices of rank $r=5$ with condition number $\kappa=10$, given $m= \lfloor \rho r (d_1 + d_2-r)\rfloor$ random samples. Median of Frobenius errors $\|\f{X}^{(K)}-\f{X}^0\|_F/\|\f{X}^0\|_F$ of $100$ independent realizations.}
\label{fig:sampcomp:1}
\end{figure}
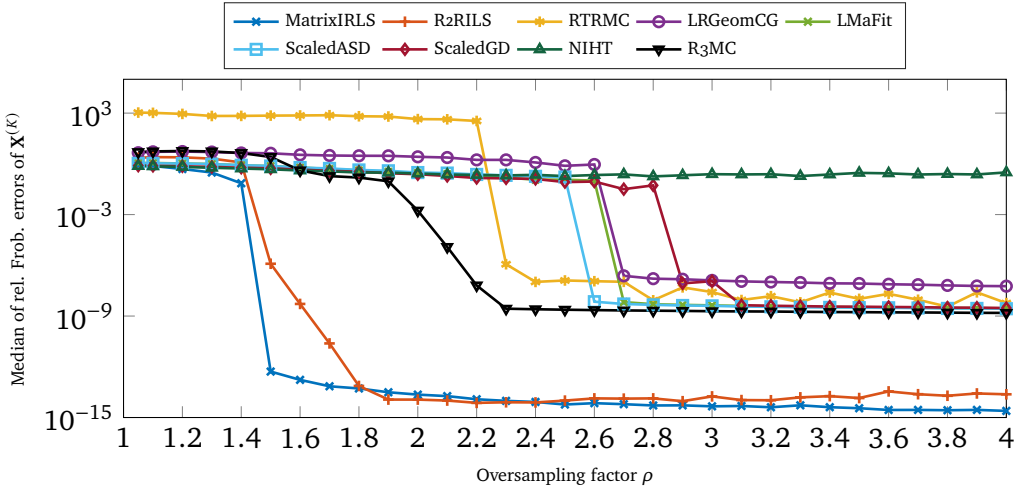
We see that \texttt{MatrixIRLS} and \texttt{R2RILS} are the only algorithms that are able to complete $\f{X}^0$ already for $\rho=1.5$. In our experiment, \texttt{R3MC} completes $\f{X}^0$ in a majority of instances starting from $\rho=2.0$, whereas the other algorithms, except from \texttt{NIHT}, are able to reconstruct the matrix most of the times if $\rho$ is at least between $2.4$ and $3.0$. This confirms the findings of \cite{BauchNadler20} which show that even for quite well-conditioned matrices, fewer samples are required if second-order methods such as \texttt{R2RILS} or \texttt{MatrixIRLS} are used.

We repeat this experiment for ill-conditioned matrices $\f{X}^0$ with $\kappa = 10^5$. In  \Cref{fig:sampcomp:2}, we see that current state-of-the-art methods are \emph{not able} to achieve exact recovery of $\f{X}^0$. This is in particular true as given the exponential decay of the singular values, in order to recover the subspace corresponding to the smallest singular value of $\f{X}^0$, a relative Frobenius error of $10^{-5}$ or even several orders of magnitude smaller needs to be achieved. We observe that \texttt{MatrixIRLS} is the only method that is able to complete $\f{X}^0$ for any of the considered oversampling factors.

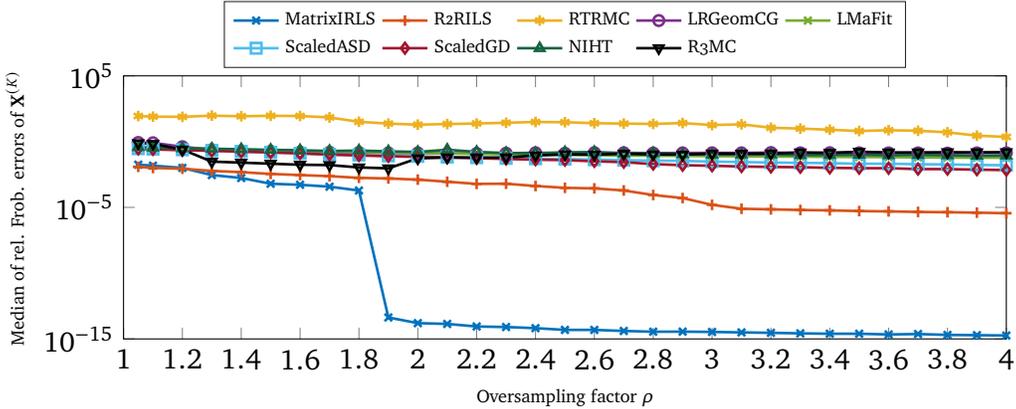
\begin{figure}[ht]
    \setlength\figureheight{35mm} 
    \setlength\figurewidth{120mm}
%
%
\definecolor{mycolor1}{rgb}{0.00000,0.44700,0.74100}%
\definecolor{mycolor2}{rgb}{0.85000,0.32500,0.09800}%
\definecolor{mycolor3}{rgb}{0.92900,0.69400,0.12500}%
\definecolor{mycolor4}{rgb}{0.49400,0.18400,0.55600}%
\definecolor{mycolor5}{rgb}{0.46600,0.67400,0.18800}%
\definecolor{mycolor6}{rgb}{0.30100,0.74500,0.93300}%
\definecolor{mycolor7}{rgb}{0.63500,0.07800,0.18400}%
\definecolor{mycolor8}{rgb}{0.08000,0.39200,0.25100}%
\begin{tikzpicture}

\begin{axis}[%
width=0.978\figurewidth,
height=\figureheight,
at={(0\figurewidth,0\figureheight)},
scale only axis,
xmin=1,
xmax=4,
xlabel style={font=\color{white!15!black}},
xlabel={Oversampling factor $\rho$},
ymode=log,
ymin=1e-15,
ymax=100000,
yminorticks=true,
ylabel style={font=\color{white!15!black}},
ylabel={Median of rel. Frob. errors of $\mathbf{X}^{(K)}$},
axis background/.style={fill=white},
legend style={legend cell align=left, align=left, draw=white!15!black},
xlabel style={font=\tiny},ylabel style={font=\tiny},legend style={font=\fontsize{7}{30}\selectfont, anchor=south, legend columns = 5, at={(0.5,1.03)}}
]
\addplot [color=mycolor1, line width=1.0pt, mark=x, mark options={solid, mycolor1}]
  table[row sep=crcr]{%
1.05	0.017696793293338\\
1.1	0.0149247467929977\\
1.2	0.0100358741421393\\
1.3	0.00287401737080599\\
1.4	0.0017217492213112\\
1.5	0.000627509308504794\\
1.6	0.000523487605715263\\
1.7	0.000374934439857983\\
1.8	0.000192177543449912\\
1.9	4.26148893627606e-14\\
2	1.58623861308894e-14\\
2.1	1.37788927229948e-14\\
2.2	8.82670925448159e-15\\
2.3	8.02507277531667e-15\\
2.4	6.59203757233433e-15\\
2.5	4.85856656140693e-15\\
2.6	4.84646205110599e-15\\
2.7	4.03688126061311e-15\\
2.8	3.52166740720263e-15\\
2.9	3.57778726155854e-15\\
3	3.39853128746517e-15\\
3.1	3.13012670512855e-15\\
3.2	2.90704826967485e-15\\
3.3	2.64447657593206e-15\\
3.4	2.47776651290129e-15\\
3.5	2.47702997773112e-15\\
3.6	2.18614791909683e-15\\
3.7	2.35729374885045e-15\\
3.8	2.0121153940015e-15\\
3.9	1.92108224894144e-15\\
4	1.82116757856728e-15\\
};
\addlegendentry{MatrixIRLS}

\addplot [color=mycolor2, line width=1.0pt, mark=+, mark options={solid, mycolor2}]
  table[row sep=crcr]{%
1.05	0.0121729439043043\\
1.1	0.00997144484447369\\
1.2	0.00872449609984643\\
1.3	0.00594736215558571\\
1.4	0.00476375327613421\\
1.5	0.00350287893314341\\
1.6	0.00286693482596602\\
1.7	0.00235681936745234\\
1.8	0.00176184450881462\\
1.9	0.00157710691467946\\
2	0.0012894886519888\\
2.1	0.000890416387229741\\
2.2	0.000612510871767183\\
2.3	0.000628320352780005\\
2.4	0.000422461566296125\\
2.5	0.000307890100903442\\
2.6	0.000278390558878799\\
2.7	0.000193359962186809\\
2.8	8.73876008786447e-05\\
2.9	5.12954109819375e-05\\
3	1.55842407560365e-05\\
3.1	7.8624565679366e-06\\
3.2	7.09562671875953e-06\\
3.3	6.33123971806548e-06\\
3.4	5.92027738932675e-06\\
3.5	5.24569952214572e-06\\
3.6	4.96587249559889e-06\\
3.7	4.51427820997321e-06\\
3.8	4.28391740187636e-06\\
3.9	3.97622888670062e-06\\
4	3.61094465249887e-06\\
};
\addlegendentry{R2RILS}

\addplot [color=mycolor3, line width=1.0pt, mark=asterisk, mark options={solid, mycolor3}]
  table[row sep=crcr]{%
1.05	88.1773642405158\\
1.1	79.3512035402019\\
1.2	76.7745266497778\\
1.3	94.3732403850792\\
1.4	84.7556749644728\\
1.5	92.0131576237411\\
1.6	88.5243100930591\\
1.7	68.9512284446234\\
1.8	32.1455731454929\\
1.9	23.4058432834348\\
2	19.4557226288538\\
2.1	21.6674555929097\\
2.2	23.8182185283779\\
2.3	27.0991116333954\\
2.4	31.168987747205\\
2.5	29.8500776136293\\
2.6	25.0704730678691\\
2.7	23.4754271963616\\
2.8	22.1912782891398\\
2.9	25.7075375758748\\
3	18.1710204054141\\
3.1	20.7392540885847\\
3.2	11.3573930075962\\
3.3	9.77324309823878\\
3.4	8.03811839649874\\
3.5	6.38279842759175\\
3.6	7.36222030845832\\
3.7	6.69451532674315\\
3.8	5.09370937566492\\
3.9	2.83881370325447\\
4	2.28326868894603\\
};
\addlegendentry{RTRMC}

\addplot [color=mycolor4, line width=1.0pt, mark=o, mark options={solid, mycolor4}]
  table[row sep=crcr]{%
1.05	0.854836884420406\\
1.1	0.814048700836708\\
1.2	0.406959731008971\\
1.3	0.244009238431741\\
1.4	0.199270482150474\\
1.5	0.182880066156403\\
1.6	0.158373902708547\\
1.7	0.15458310852111\\
1.8	0.147855130834231\\
1.9	0.13619099137814\\
2	0.136395239316276\\
2.1	0.129441656898353\\
2.2	0.125141551454935\\
2.3	0.13113615625425\\
2.4	0.141995452528891\\
2.5	0.12983210569411\\
2.6	0.143843931977864\\
2.7	0.141583100719317\\
2.8	0.133851752615748\\
2.9	0.138590807836888\\
3	0.139726800063849\\
3.1	0.13928720817718\\
3.2	0.146900614777377\\
3.3	0.139156782879128\\
3.4	0.142598685370717\\
3.5	0.140343853668292\\
3.6	0.138222212114855\\
3.7	0.13808149128863\\
3.8	0.132007622667505\\
3.9	0.138055046348376\\
4	0.139845416443067\\
};
\addlegendentry{LRGeomCG}

\addplot [color=mycolor5, line width=1.0pt, mark=x, mark options={solid, mycolor5}]
  table[row sep=crcr]{%
1.05	0.370479094334638\\
1.1	0.344738688789208\\
1.2	0.305003685693449\\
1.3	0.274819969474776\\
1.4	0.247740013300637\\
1.5	0.220229797221742\\
1.6	0.197593429507371\\
1.7	0.174974560110361\\
1.8	0.161880928811185\\
1.9	0.149072956986771\\
2	0.135573313964275\\
2.1	0.125617559765276\\
2.2	0.117624754306401\\
2.3	0.107770796023168\\
2.4	0.102869148200105\\
2.5	0.0973699783171384\\
2.6	0.0911976587510287\\
2.7	0.087070467542029\\
2.8	0.0851695140911353\\
2.9	0.0820343098805306\\
3	0.0784133307659842\\
3.1	0.0768336670990977\\
3.2	0.073298345272641\\
3.3	0.071935717650967\\
3.4	0.0688059674895942\\
3.5	0.0675089237185283\\
3.6	0.0660170864969936\\
3.7	0.0630233750193977\\
3.8	0.0608667400397613\\
3.9	0.0599704565136605\\
4	0.0594213595172383\\
};
\addlegendentry{LMaFit}

\addplot [color=mycolor6, line width=1.0pt, mark=square, mark options={solid, mycolor6}]
  table[row sep=crcr]{%
1.05	0.293535194670914\\
1.1	0.273424548894986\\
1.2	0.241863789436971\\
1.3	0.210308160885623\\
1.4	0.179671150516937\\
1.5	0.153704204975937\\
1.6	0.127383037107324\\
1.7	0.104581966212891\\
1.8	0.0967381965419033\\
1.9	0.0829217878720201\\
2	0.0767587856465289\\
2.1	0.0677358836206883\\
2.2	0.0610941507309169\\
2.3	0.0536826770846946\\
2.4	0.0501433886034221\\
2.5	0.047032656931879\\
2.6	0.0424026045167179\\
2.7	0.0391535053499478\\
2.8	0.0355206008825965\\
2.9	0.0342638696279479\\
3	0.0301352181009652\\
3.1	0.0279617109248466\\
3.2	0.0257797090099662\\
3.3	0.0240692916355119\\
3.4	0.023210471331287\\
3.5	0.0206607719407823\\
3.6	0.0208466572877179\\
3.7	0.0190403180574802\\
3.8	0.0184171481393597\\
3.9	0.0167656690800692\\
4	0.0160719495980064\\
};
\addlegendentry{ScaledASD}

\addplot [color=mycolor7, line width=1.0pt, mark=diamond, mark options={solid, mycolor7}]
  table[row sep=crcr]{%
1.05	0.282005628998929\\
1.1	0.259057103039847\\
1.2	0.222535496634528\\
1.3	0.200741125186381\\
1.4	0.17154972703974\\
1.5	0.151177149976177\\
1.6	0.127593953591546\\
1.7	0.103040137529032\\
1.8	0.0888678232426331\\
1.9	0.0775871141116796\\
2	0.0689565811530352\\
2.1	0.0628166001450546\\
2.2	0.0552483658972788\\
2.3	0.0479359822321287\\
2.4	0.0423040634728772\\
2.5	0.0392512972817216\\
2.6	0.0318814987640092\\
2.7	0.0274905957643751\\
2.8	0.0195835655807152\\
2.9	0.0161063827691886\\
3	0.0145834490638091\\
3.1	0.0128517676575462\\
3.2	0.0118565862006724\\
3.3	0.0113175533060402\\
3.4	0.0101385628631579\\
3.5	0.00956005369402189\\
3.6	0.0095480484530076\\
3.7	0.00816374445443083\\
3.8	0.00815332450461049\\
3.9	0.00738917286243342\\
4	0.00685805936868511\\
};
\addlegendentry{ScaledGD}

\addplot [color=mycolor8, line width=1.0pt, mark=triangle, mark options={solid, mycolor8}]
  table[row sep=crcr]{%
1.05	0.381684333532095\\
1.1	0.350455776498883\\
1.2	0.31465709602788\\
1.3	0.289284561868943\\
1.4	0.261387299429177\\
1.5	0.236763350860108\\
1.6	0.219918434274913\\
1.7	0.193810606899661\\
1.8	0.209297075614819\\
1.9	0.181745589813183\\
2	0.163995242516317\\
2.1	0.24155823327221\\
2.2	0.164789720459069\\
2.3	0.136348972794721\\
2.4	0.131894690711273\\
2.5	0.154017219048052\\
2.6	0.159725873265842\\
2.7	0.127157617918866\\
2.8	0.132809280412949\\
2.9	0.110726691809945\\
3	0.103675178074836\\
3.1	0.105090353370288\\
3.2	0.105795441660664\\
3.3	0.0954892884532374\\
3.4	0.0971222604079868\\
3.5	0.100048891056211\\
3.6	0.0925294362854234\\
3.7	0.0961592579970626\\
3.8	0.0913540561972631\\
3.9	0.0848092286768069\\
4	0.0876188156232948\\
};
\addlegendentry{NIHT}

\addplot [color=black, line width=1.0pt, mark=triangle, mark options={solid, rotate=180, black}]
  table[row sep=crcr]{%
1.05	0.759397938691566\\
1.1	0.686619819776935\\
1.2	0.247550670944846\\
1.3	0.0302600651319708\\
1.4	0.0253275850883919\\
1.5	0.0210098160573033\\
1.6	0.0174786621929334\\
1.7	0.016526845397056\\
1.8	0.010903643548303\\
1.9	0.0093197216103729\\
2	0.0532077659509479\\
2.1	0.0620901253738798\\
2.2	0.0610195002873348\\
2.3	0.0625049374381483\\
2.4	0.088277519849199\\
2.5	0.106776504482481\\
2.6	0.103339485074257\\
2.7	0.119815555071001\\
2.8	0.106115009408714\\
2.9	0.104642644167653\\
3	0.123117755038511\\
3.1	0.121361867139353\\
3.2	0.125245424789379\\
3.3	0.149990040632705\\
3.4	0.140117824316866\\
3.5	0.16979281874768\\
3.6	0.149784142357538\\
3.7	0.151104361587104\\
3.8	0.156137119057806\\
3.9	0.153833423213751\\
4	0.15677368919979\\
};
\addlegendentry{R3MC}

\end{axis}
\end{tikzpicture}%
\caption{Performance of matrix completion algorithms as in \Cref{fig:sampcomp:1}, but with $\kappa = 10^5$. Median of $50$ realizations.}
\label{fig:sampcomp:2}
\end{figure}

\subsection{Running time for ill-conditioned problems} \label{sec:runtime:illconditioned}

 \begin{figure}[ht]
\includegraphics{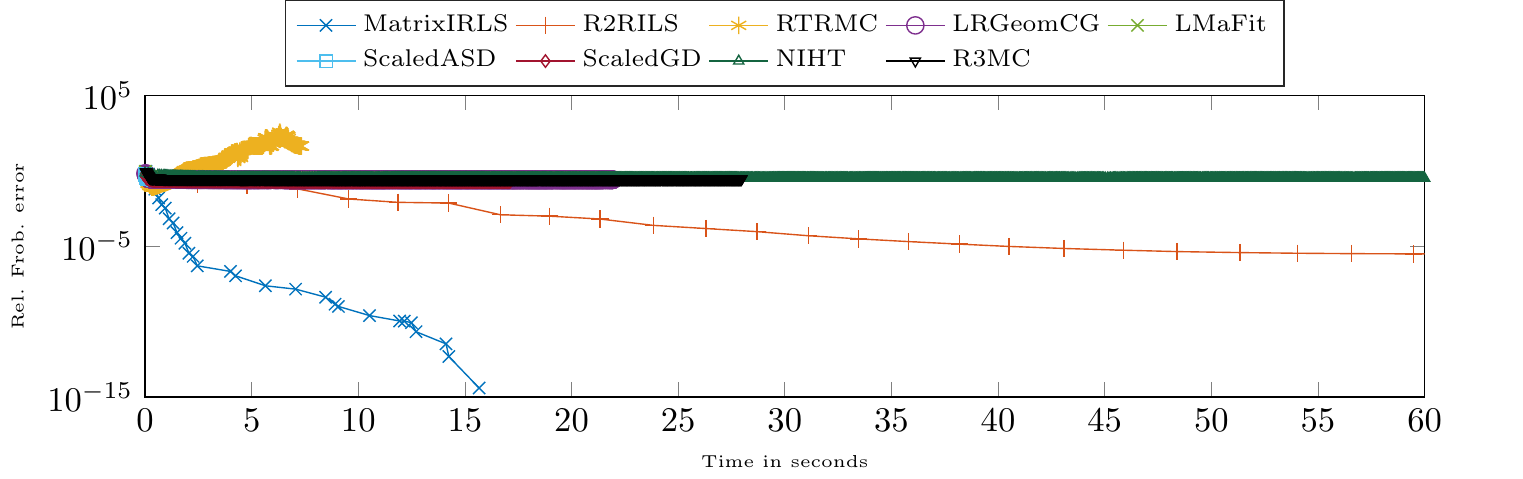}
 \caption{Completion task for a highly ill-conditioned $1000 \times 1000$ matrix of rank $r=10$ with $\kappa=10^{10}$ ($\rho= 4$).}
 \label{running_time_ill_cond}
 \end{figure}

In Figure \ref{running_time_ill_cond}, for an oversampling ratio of $\rho = 4$, we illustrate the completion of one single extremely ill-conditioned $1000 \times 1000$ matrix with $\rank=10 $ and $\kappa=10^{10}$ and exponentially interpolated singular values as described above. We again can see that only second-order methods such as \texttt{R2RILS} or \texttt{MatrixIRLS} are able to achieve a relative Frobenius error $\approx 10^{-5}$ or smaller.  \texttt{MatrixIRLS} goes beyond that and attains a relative Frobenius error of the order of the machine precision and, remarkably, exactly recover all the singular values up to 15 digits. This also shows that the conjugated gradient and the randomized block Krylov method used at the inner core of our implementation can be extremely precise when properly adjusted. \texttt{R2RILS} is also able to obtain relatively low Frobenius error but unlike our method, it is not able to retrieve all the singular values with high accuracy. Other methods were observed to lead to a meaningful error decrease for the ill-conditioned matrix of interest.

In Figure \ref{running_time_MatrixIRLS_vs_R2RILS}, we compare the execution time of \texttt{R2RILS} and \texttt{MatrixIRLS} for a range of ground truth matrices with increasing dimension, for an oversampling ratio of $\rho = 2.5$, whose singular values are linearly interpolated between $\kappa$ and $1$. We observe that the larger the dimensions are, the larger is the discrepancy in the running time of the two algorithms.  Other algorithms are not considered in this experiment because they typically do not reach a relative error below $10^{-4}$ for $\kappa\gg10^2$.

\begin{figure}[h!]
    \setlength\figureheight{35mm} 
    \setlength\figurewidth{120mm}
%
%
\definecolor{mycolor1}{rgb}{0.00000,0.44700,0.74100}
\definecolor{mycolor2}{rgb}{0.85000,0.32500,0.09800}
\definecolor{mycolor3}{rgb}{0.92900,0.69400,0.12500}
\definecolor{mycolor4}{rgb}{0.49400,0.18400,0.55600}
\definecolor{mycolor5}{rgb}{0.46600,0.67400,0.18800}
\definecolor{mycolor6}{rgb}{0.30100,0.74500,0.93300}
\definecolor{mycolor7}{rgb}{0.63500,0.07800,0.18400}
\definecolor{mycolor8}{rgb}{0.08000,0.39200,0.25100}
\begin{tikzpicture}

\begin{axis}[%
width=0.978\figurewidth,
height=\figureheight,
at={(0\figurewidth,0\figureheight)},
scale only axis,
xmin=100,
xmax=1000,
xlabel={m (with n=m+100)},
ymin=2.5,
ymax=7,
ylabel={Execution time (seconds)},
axis background/.style={fill=white},
xmajorgrids,
ymajorgrids,
xlabel style={font=\tiny},ylabel style={font=\tiny},legend style={font=\fontsize{7}{30}\selectfont, anchor=south, legend columns = 2, at={(0.5,1.01)}}
]

\addplot [color=mycolor1, line width=1pt,  mark=asterisk, mark options={solid, mycolor1}]
  table[row sep=crcr]{%
100	2.95304825418182\\
200	2.98230332166667\\
300	3.00724506338462\\
400	3.08807303557143\\
500	3.22394640293334\\
600	3.274575889875\\
700	3.33046376941177\\
800	3.42672688188889\\
900	3.54640258852632\\
1000	3.7073854512\\
};
\addlegendentry{MatrixIRLS r=5}

\addplot [color=mycolor2, line width=1pt,  mark=asterisk, mark options={solid, mycolor2}]
  table[row sep=crcr]{%
100	5.199734626\\
200	4.97052347\\
300	5.06955289830769\\
400	5.03849069399999\\
500	5.082422796\\
600	5.218918925\\
700	5.35722198082353\\
800	5.58250735188888\\
900	5.83916094263158\\
1000	6.4279443572\\
};
\addlegendentry{R2RILS r=5}

\addplot [color=mycolor1, line width=1.5pt]
  table[row sep=crcr]{%
100	3.55308134704762\\
200	3.44113479809091\\
300	3.37757395730435\\
400	3.35926170525\\
500	3.39726942392001\\
600	3.47673567853847\\
700	3.60797207814816\\
800	3.75406394928572\\
900	3.94470405496553\\
1000	4.20021874420001\\
};
\addlegendentry{MatrixIRLS r=10}

\addplot [color=mycolor2, line width=1.5pt]
  table[row sep=crcr]{%
100	6.20542460038095\\
200	6.087665208\\
300	6.03986127608695\\
400	6.08141432908333\\
500	6.20602667224\\
600	6.37619482830769\\
700	6.64026641185185\\
800	6.98286348607142\\
900	7.3768662233793\\
1000	7.76715193293332\\
};
\addlegendentry{R2RILS r=10}

\end{axis}
\end{tikzpicture}%
\caption{Execution time of \texttt{R2RILS} and \texttt{MatrixIRLS} for completion of rank $r \in \{5,10\}$ matrices of size $m \times (m+100)$ and condition number $\kappa=10^2$, averaged across 50 independent realizations.}
\label{running_time_MatrixIRLS_vs_R2RILS}
\end{figure}
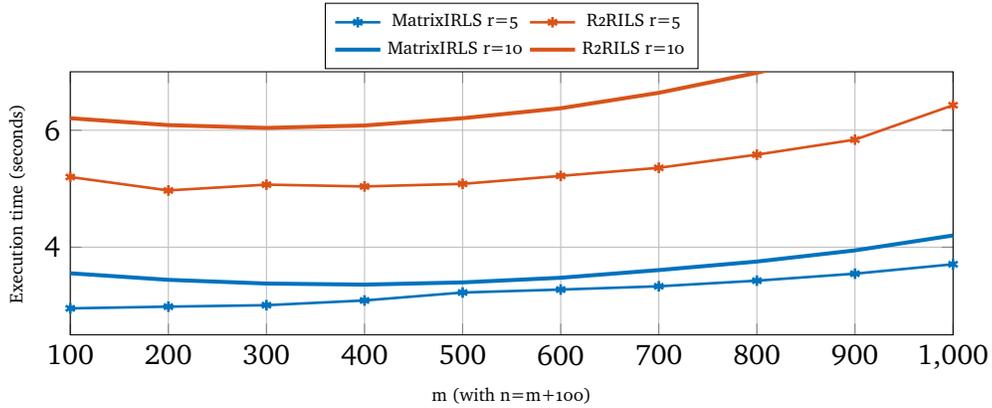

\subsection{MatrixIRLS vs. rank-adaptive strategies} \label{sec:rank_update}

In \Cref{numerics:data_efficiency}, all methods were provided with
the correct rank $r$ of the ground truth, which was used to determine
the size of the matrix factors or the rank of the fixed rank manifold.
Even in this case, we illustrated numerically that most of the methods
are not able to recover highly ill-conditioned matrices. To handle such
ill-conditioned completion problems, \cite{MishraS14,Uschmajew_Vandereycken,Tan2014} proposed rank-adaptive variants of the methods \texttt{R3MC} and \texttt{LRGeomCG}.
These variants, which we call \texttt{LRGeomCG
Pursuit}\footnote{The MATLAB code containing the rank update was
provided by B. Vandereycken in a private communication.} \cite{Uschmajew_Vandereycken,Tan2014} and \texttt{R3MC w/ Rank Update} \cite{MishraS14}, respectively, combine fixed-rank optimization with outer iterations that increase $\widetilde{r}$ from 1 to a target rank $r$, while warm starting each outer iteration with the output of the previous iteration. 
To compare the data efficiency of \texttt{MatrixIRLS} with the one of these three algorithms, we repeat the experiments of \Cref{numerics:data_efficiency} for these methods and report the median Frobenius errors for the completion of $1000 \times 1000$ matrices of rank $r=5$ with condition numbers $\kappa = 10$ and $\kappa = 10^{5}$, respectively, with those of \texttt{MatrixIRLS} in Figures \ref{fig:sampcomp:3} and \ref{fig:sampcomp:4}.

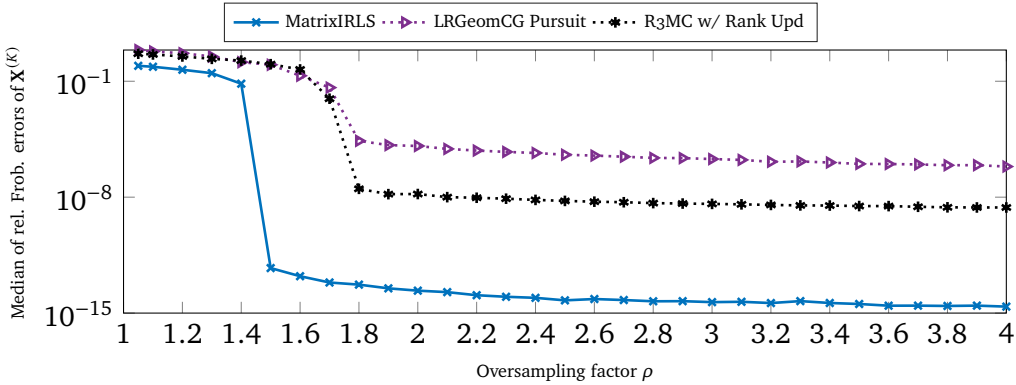
\begin{figure}[h!]
    \setlength\figureheight{35mm} 
    \setlength\figurewidth{120mm}
%
%
\definecolor{mycolor1}{rgb}{0.00000,0.44700,0.74100}%
\definecolor{mycolor2}{rgb}{0.49400,0.18400,0.55600}%
\begin{tikzpicture}

\begin{axis}[%
width=0.978\figurewidth,
height=\figureheight,
at={(0\figurewidth,0\figureheight)},
scale only axis,
xmin=1,
xmax=4,
xlabel style={font=\color{white!15!black}},
xlabel={Oversampling factor $\rho$},
ymode=log,
ymin=1e-15,
ymax=7.68468566470869,
yminorticks=true,
ylabel style={font=\color{white!15!black}},
ylabel={Median of rel. Frob. errors of $\mathbf{X}^{(K)}$},
axis background/.style={fill=white},
legend style={legend cell align=left, align=left, draw=white!15!black},
xlabel style={font=\tiny},ylabel style={font=\tiny},legend style={font=\fontsize{7}{30}\selectfont, anchor=south, legend columns = 3, at={(0.45,1.03)}}
]
\addplot [color=mycolor1, line width=1.0pt, mark=x, mark options={solid, mycolor1}]
  table[row sep=crcr]{%
1.05	0.854255555628291\\
1.1	0.749423097099185\\
1.2	0.494716881980835\\
1.3	0.307031095563721\\
1.4	0.0710950413951957\\
1.5	5.22910462862914e-13\\
1.6	1.68588918367454e-13\\
1.7	7.0339665885088e-14\\
1.8	5.22646777788465e-14\\
1.9	3.12007321555135e-14\\
2	2.26159575881266e-14\\
2.1	1.84275579519234e-14\\
2.2	1.20010224139451e-14\\
2.3	9.65828095082859e-15\\
2.4	8.33489396563835e-15\\
2.5	5.87405834953726e-15\\
2.6	7.06122879837332e-15\\
2.7	6.140965219143e-15\\
2.8	5.15640770569221e-15\\
2.9	5.24064290362918e-15\\
3	4.58553528727644e-15\\
3.1	4.81251473302107e-15\\
3.2	4.03701051254568e-15\\
3.3	5.28365319386791e-15\\
3.4	4.07526625097761e-15\\
3.5	3.54126224418117e-15\\
3.6	2.79718832356394e-15\\
3.7	2.81718954270358e-15\\
3.8	2.69583211552255e-15\\
3.9	2.82440163639793e-15\\
4	2.46058369466678e-15\\
};
\addlegendentry{MatrixIRLS}

\addplot [color=mycolor2, dotted, line width=1.0pt, mark=triangle, mark options={solid, rotate=270, mycolor2}]
  table[row sep=crcr]{%
1.05	7.68468566470869\\
1.1	6.33162290906835\\
1.2	4.92349255222447\\
1.3	3.12162630366101\\
1.4	1.41749403397035\\
1.5	0.98634426581074\\
1.6	0.214297345328806\\
1.7	0.0412216051407876\\
1.8	2.51793451811461e-05\\
1.9	1.38968454137595e-05\\
2	1.23152370619004e-05\\
2.1	8.20433060375076e-06\\
2.2	6.53576698004167e-06\\
2.3	5.40674693300936e-06\\
2.4	4.66020650532623e-06\\
2.5	3.64846077678792e-06\\
2.6	3.24232171440137e-06\\
2.7	2.7723503024717e-06\\
2.8	2.34415869464538e-06\\
2.9	2.28493158805644e-06\\
3	2.00192936817848e-06\\
3.1	1.77801953042947e-06\\
3.2	1.37398137399264e-06\\
3.3	1.43266666214825e-06\\
3.4	1.23575172370206e-06\\
3.5	1.01947795121814e-06\\
3.6	1.00446961483233e-06\\
3.7	9.48841130253936e-07\\
3.8	8.47640618636798e-07\\
3.9	8.56280139705756e-07\\
4	7.0677722786273e-07\\
};
\addlegendentry{LRGeomCG Pursuit}

\addplot [color=black, dotted, line width=1.0pt, mark=asterisk, mark options={solid, black}]
  table[row sep=crcr]{%
1.05	4.60604036755615\\
1.1	4.14801814433686\\
1.2	3.14926448754613\\
1.3	2.24984461047932\\
1.4	1.77007898452388\\
1.5	1.07116419615937\\
1.6	0.515756910915468\\
1.7	0.00880415516599799\\
1.8	3.18527638045956e-08\\
1.9	1.52373312931524e-08\\
2	1.54823306806071e-08\\
2.1	1.00314406720369e-08\\
2.2	9.22369181999578e-09\\
2.3	8.08304684869755e-09\\
2.4	6.83347150933846e-09\\
2.5	5.9417476988411e-09\\
2.6	5.33562674212395e-09\\
2.7	4.94270606318811e-09\\
2.8	4.38103075853045e-09\\
2.9	4.11497079487388e-09\\
3	3.92070499289155e-09\\
3.1	3.68417325414267e-09\\
3.2	3.39925521824471e-09\\
3.3	3.17926085031972e-09\\
3.4	3.11918607579225e-09\\
3.5	2.91043164527757e-09\\
3.6	2.88991847438305e-09\\
3.7	2.62461028536712e-09\\
3.8	2.43543128363482e-09\\
3.9	2.39104263164535e-09\\
4	2.38334404044277e-09\\
};
\addlegendentry{R3MC w/ Rank Upd}

\end{axis}
\end{tikzpicture}%
\caption{Completion of $1000 \times 1000$ matrices of rank $r=5$ with condition number $\kappa=10$, experiment as in Figure \ref{fig:sampcomp:1}.}
\label{fig:sampcomp:3}
\end{figure}

In Figure \ref{fig:sampcomp:3}, we observe that in the presence of a relatively small condition number of $\kappa = 10$, \texttt{MatrixIRLS} is more data efficient than the two rank-adaptaive methods as their phase transition occurs for a larger oversampling factor ($\rho= 1.8$ vs. $\rho = 1.5$). 

On the other hand, it can be seen in Figure \ref{fig:sampcomp:4} that the rank-adaptive strategies \texttt{LRGeomCG
Pursuit} and  \texttt{R3MC w/ Rank Update} shine when completing matrices with large condition number such as $\kappa = 10^5$, as their phase transition occurs at around $\rho= 1.8$ and $\rho = 1.7$, where it occurs at $\rho = 1.9$ for \texttt{MatrixIRLS}. This shows that for large condition number, rank adaptive strategies can outperform the data efficiency of \texttt{MatrixIRLS}, and in both experiments, the phase transitions are considerably better than for their fixed rank versions \texttt{LRGeomCG} and  \texttt{R3MC}, cf. Figures \ref{fig:sampcomp:1} and Figure \ref{fig:sampcomp:2}.
\begin{figure}[h]
    \setlength\figureheight{35mm} 
    \setlength\figurewidth{120mm}
%
%
\definecolor{mycolor1}{rgb}{0.00000,0.44700,0.74100}%
\definecolor{mycolor2}{rgb}{0.49400,0.18400,0.55600}%
\begin{tikzpicture}

\begin{axis}[%
width=0.978\figurewidth,
height=\figureheight,
at={(0\figurewidth,0\figureheight)},
scale only axis,
xmin=1,
xmax=4,
xlabel style={font=\color{white!15!black}},
xlabel={Oversampling factor $\rho$},
ymode=log,
ymin=1e-15,
ymax=2.61139409645811,
yminorticks=true,
ylabel style={font=\color{white!15!black}},
ylabel={Median of rel. Frob. errors of $\mathbf{X}^{(K)}$},
axis background/.style={fill=white},
legend style={legend cell align=left, align=left, draw=white!15!black},
xlabel style={font=\tiny},ylabel style={font=\tiny},legend style={font=\fontsize{7}{30}\selectfont, anchor=south, legend columns = 3, at={(0.45,1.03)}}
]
\addplot [color=mycolor1, line width=1.0pt, mark=x, mark options={solid, mycolor1}]
  table[row sep=crcr]{%
1.05	0.017696793293338\\
1.1	0.0149247467929977\\
1.2	0.0100358741421393\\
1.3	0.00287401737080599\\
1.4	0.0017217492213112\\
1.5	0.000627509308504794\\
1.6	0.000523487605715263\\
1.7	0.000374934439857983\\
1.8	0.000192177543449912\\
1.9	4.26148893627606e-14\\
2	1.58623861308894e-14\\
2.1	1.37788927229948e-14\\
2.2	8.82670925448159e-15\\
2.3	8.02507277531667e-15\\
2.4	6.59203757233433e-15\\
2.5	4.85856656140693e-15\\
2.6	4.84646205110599e-15\\
2.7	4.03688126061311e-15\\
2.8	3.52166740720263e-15\\
2.9	3.57778726155854e-15\\
3	3.39853128746517e-15\\
3.1	3.13012670512855e-15\\
3.2	2.90704826967485e-15\\
3.3	2.64447657593206e-15\\
3.4	2.47776651290129e-15\\
3.5	2.47702997773112e-15\\
3.6	2.18614791909683e-15\\
3.7	2.35729374885045e-15\\
3.8	2.0121153940015e-15\\
3.9	1.92108224894144e-15\\
4	1.82116757856728e-15\\
};
\addlegendentry{MatrixIRLS}

\addplot [color=mycolor2, dotted, line width=1.0pt, mark=triangle, mark options={solid, rotate=270, mycolor2}]
  table[row sep=crcr]{%
1.05	2.61139409645811\\
1.1	1.92937826415313\\
1.2	1.4240050674394\\
1.3	0.858196089636759\\
1.4	0.0571440189086365\\
1.5	0.00219121017064355\\
1.6	5.06080311439453e-06\\
1.7	9.22712201261971e-06\\
1.8	2.75021572328361e-09\\
1.9	1.63560402933676e-09\\
2	1.28404802920269e-09\\
2.1	1.09149156802796e-09\\
2.2	7.56968670143147e-10\\
2.3	6.44017078920023e-10\\
2.4	6.03265562182345e-10\\
2.5	4.64605360276084e-10\\
2.6	4.01978294909424e-10\\
2.7	3.77870237521132e-10\\
2.8	3.41349263644955e-10\\
2.9	3.05195705766597e-10\\
3	2.99387494160118e-10\\
3.1	2.16127034539723e-10\\
3.2	1.5992057100716e-10\\
3.3	1.73904037269403e-10\\
3.4	1.49654699097085e-10\\
3.5	1.20455995360223e-10\\
3.6	1.29854370869532e-10\\
3.7	1.14117236528731e-10\\
3.8	1.09935477094216e-10\\
3.9	9.7770034762092e-11\\
4	9.38533117015799e-11\\
};
\addlegendentry{LRGeomCG Pursuit}

\addplot [color=black, dotted, line width=1.0pt, mark=asterisk, mark options={solid, black}]
  table[row sep=crcr]{%
1.05	0.293936113014958\\
1.1	0.245327074408468\\
1.2	0.153372963092698\\
1.3	0.0161058422117617\\
1.4	0.00486326356421978\\
1.5	0.000321101084437894\\
1.6	5.22005424611248e-05\\
1.7	6.03805928604739e-12\\
1.8	4.17691953016486e-12\\
1.9	2.30820575678731e-12\\
2	1.73409069085845e-12\\
2.1	1.33240780577722e-12\\
2.2	1.17984926698607e-12\\
2.3	1.05727574442111e-12\\
2.4	7.69367858550375e-13\\
2.5	6.86676051910896e-13\\
2.6	6.33340334301605e-13\\
2.7	5.83712837735635e-13\\
2.8	5.29067137574081e-13\\
2.9	5.21280377177225e-13\\
3	4.73863365582093e-13\\
3.1	4.32000148731268e-13\\
3.2	4.0977648031905e-13\\
3.3	3.80006930907078e-13\\
3.4	3.80325545195378e-13\\
3.5	3.56500663616837e-13\\
3.6	3.23585428634483e-13\\
3.7	2.97451323558135e-13\\
3.8	3.05701116505058e-13\\
3.9	2.95470573555244e-13\\
4	2.82766902066654e-13\\
};
\addlegendentry{R3MC w/ Rank Upd}

\end{axis}
\end{tikzpicture}%
\caption{Completion of $1000 \times 1000$ matrices of rank $r=5$ with condition number $\kappa=10^5$, experiment as in Figure \ref{fig:sampcomp:2}.}
\label{fig:sampcomp:4}
\end{figure}
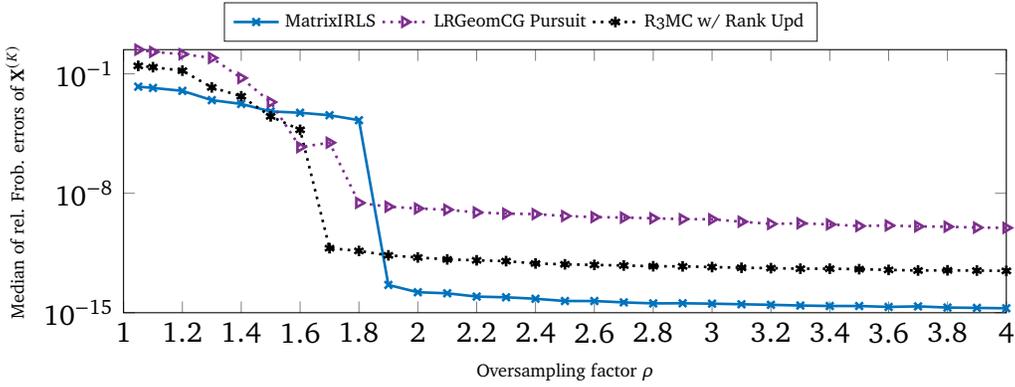

In all experiments so far, we have considered low-rank matrices with $r$ singular values that exponentially decrease from $\kappa$ to $1$, as described in the beginning of this section. This might be a setting that is particularly suitable for rank-adaptive strategies that increase the rank parameter $\widetilde{r}$ one-by-one, as the singular subspaces are all one-dimensional and well-separated. For this reason, in a last experiment, we change this setup and consider ground truth matrices $\f{X}^0$ that have a \emph{plateau} in the set of singular values, potentially presenting a larger challenge for completion methods due to a higher dimensional subspace spanned by a set of multiple singular vectors. In particular, we consider the completion of a $1000 \times 1000$ matrix $\f{X}^0$ with $10$ singular values equal to $10^{10}\cdot \exp(-10\cdot\log(10)\frac{14}{29})$, and with $10$ singular values linearly interpolated on a logarithmic scale between this value and $10^{10}$ and, and another $10$ between this value and $1$ (see also Appendix \ref{appendix:figure7} for an illustration). For a random instance of such a matrix, we report the relative Frobenius error vs. execution time for the methods \texttt{MatrixIRLS} against the rank-adaptive variants of
\texttt{LRGeomCG} and \texttt{R3MC}, here denoted by \texttt{LRGeomCG
Pursuit} and \texttt{R3MC w/ Rank Update} in Figure \ref{fig:plateau}, from random samples with a small oversampling factor of $\rho = 1.5$.

\begin{figure}[h]
\includegraphics{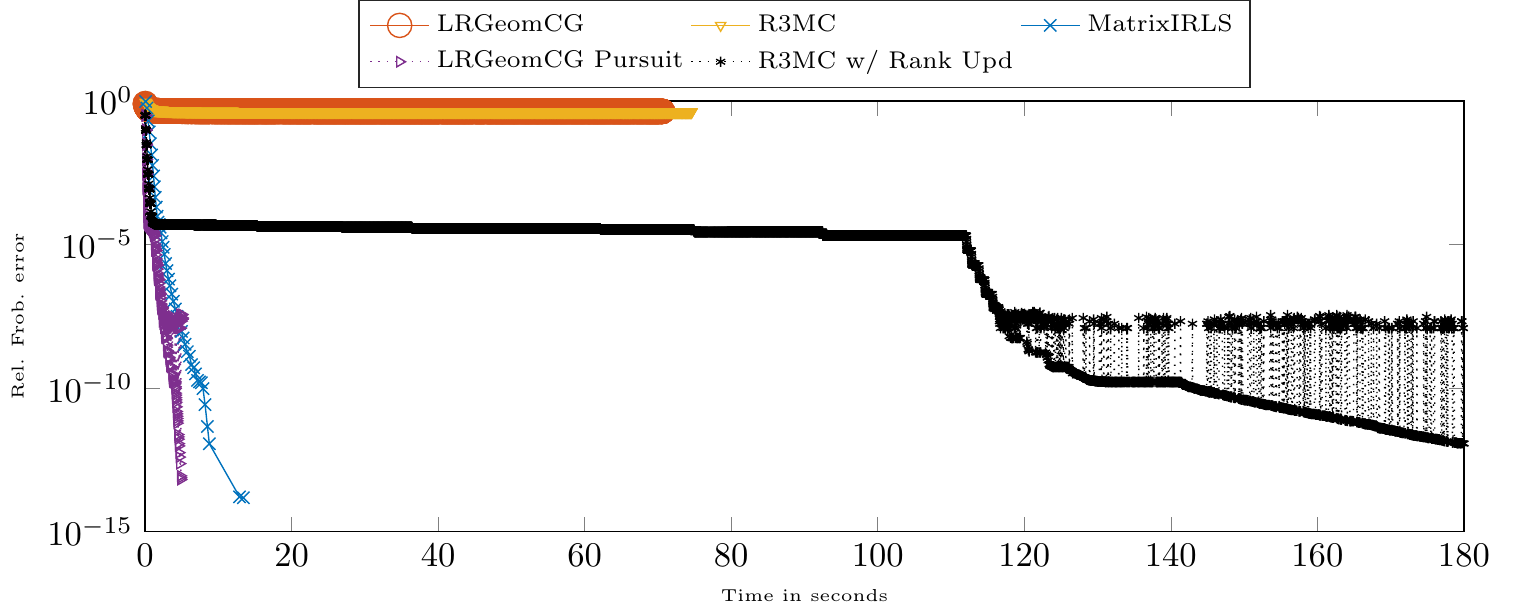}
\caption{Comparison of matrix completion algorithms for $1000 \times
1000$ matrices of rank $r=30$ with condition number $\kappa=10^{10}$ and $10$ equal singular values, oversampling factor of $\rho = 1.5$.}
\label{fig:plateau}
\end{figure}

We observe that the fixed-rank variants \texttt{LRGeomCG} and \texttt{R3MC} are not able to complete the matrix, which is in line with the experiment of Section \ref{sec:runtime:illconditioned}. \texttt{R3MC w/ Rank Update} exhibits a quick error decrease to a range around $6 \cdot 10^{-5}$, after which it just decreases very slowly for around 110 seconds, before converging to $\f{X}^0$ up to an error of around $10^{-12}$ within another $70$ seconds. The stagnation phase presumably corresponds to the learning of the $10$-dimensional singular space of $\f{X}^0$ in the central part of its spectrum. \texttt{LRGeomCG Pursuit}, on the other hand, reaches an error of around $10^{-12}$ already after $5$ seconds, albeit without monotonicity with a fluctuation phase between errors of $10^{-8}$ and $10^{-12}$ from seconds $3$ to $5$. 
For \texttt{MatrixIRLS}, we use a tolerance parameter for the relative residual in the conjugate gradient method of $\text{tol}_{\text{inner}} = 10^{-3}$ and a maximal number of $3$ iterations for the randomized Block Krylov method (cf. \ref{sec:experimental:details} for the default parameters), and observe that the method successfully converges to $\f{X}^0$ slightly slower with a convergence within $13$ seconds, but, remarkably, unlike \texttt{LRGeomCG Pursuit}, with a monotonous error decrease.\footnote{For the default choice of algorithmic parameters as described in Appendix \ref{sec:experimental:details}, we obtain a qualitatively similar behavior for \texttt{MatrixIRLS}, but with a small runtime multiple due to the higher required precision at each iteration.}

\section{Conclusion and Outlook}\label{conclusion}

We formulated \texttt{MatrixIRLS}, a second order method that is able to efficiently complete large, highly ill-conditioned matrices from few samples, a problem for which most state-of-the-art methods fail. It improves on previous approaches for the optimization of non-convex rank objectives by applying a suitable smoothing strategy combined with saddle-escaping Newton-type steps.

As one goal of our investigation has been also to provide an efficient implementation, we focused on the matrix completion problem, leaving the extension of the ideas to other low-rank matrix estimation problems to future work including the case of inexact data or measurement errors. Furthermore, while we establish a local convergence guarantee for the algorithm, a precise analysis of its global convergence behavior might be of interest.

\section*{Software}
An implementation of \texttt{MatrixIRLS} including scripts to reproduce the presented experiments can be found at \url{https://github.com/ckuemmerle/MatrixIRLS}.

\section*{Acknowledgements}
We thank Bart Vandereycken and Boaz Nadler for fruitful discussions and for providing very valuable feedback. 
C.K. is grateful for support by the NSF under the grant NSF-IIS-1837991. C.M.V. is supported by the German Science Foundation (DFG) within the Gottfried Wilhelm Leibniz Prize under Grant BO 1734/20-1, under contract number PO-1347/3-2, under Emmy Noether junior research group KR 4512/1-1 and within Germany’s Excellence Strategy EXC-2111 390814868.

\appendix

\section*{Appendix: Supplementary material for \emph{A Scalable Second Order Method for Ill-Conditioned Matrix Completion from Few Samples}}

This supplementary material is divided into three parts: In \Cref{sec:proof:computational}, we provide details for an efficient implementation of \texttt{MatrixIRLS} and furthermore, show \Cref{thm:MatrixIRLS:computationalcost:Xkk}. Then, we show in \Cref{sec:proof:local:guarantee} and in \Cref{sec:proof:wellconditioning} the theoretical results of \Cref{sec:convergence}. In \Cref{sec:remarks:smoothing:newton}, we provide some details about the interpretation of \texttt{MatrixIRLS} as a saddle-escaping smoothing Newton method, cf. \Cref{sec:Newton:interpretation}. Finally, we present a detailed description of the algorithmic parameters of the experiments of \Cref{sec:numerics} in \Cref{sec:experimental:details}, as well as a remark on an experiment of \Cref{sec:rank_update}.

\section{Computational Aspects of MatrixIRLS} \label{sec:proof:computational}
We first introduce some notation that is suitable to describe an efficient implementation of \Cref{algo:MatrixIRLS}. Whenever we write $\|\f{X}\|_{S_\infty}$, we refer to the Schatten-$\infty$ or spectral norm $\|\f{X}\|_{S_\infty} = \sigma_1 (\f{X})$ of a matrix $\f{X} \in \Rdd$. For $d_1,d_2 \in \N$, $d=\min(d_1,d_2)$ and a vector $\sigma \in \R^d$, we denote the $(d_1 \times d_2)$-matrix with the entries of $\sigma$ on its diagonal (an $0$ otherwise) as $\dg(\sigma) \in \R^{d_1 \times d_2}$, i.e., $\dg(\sigma)_{ij} = \sigma_i \delta_{ij}$ for all $i \in [d_1]$ and $j \in [d_2]$. If $\f{X}\hk \in \Rdd$ is a matrix with singular values $\sigma_i\hk:=\sigma_i(\f{X}\hk)$, for $\epsilon_k > 0$, let $r_k := |\{i \in [d]: \sigma_i(\f{X}\hk) > \epsilon_k\}| = |\{i \in [d]: \sigma_i\hk > \epsilon_k\}|$. Then we write a singular value decomposition of $\f{X}\hk$ such that
\begin{equation} \label{eq:Xk:bothsvds}
\f{X}\hk= \f{U}_k \dg(\sigma^{(k)}) \f{V}_k^* = 
\begin{bmatrix} 
\f{U}\hk & \f{U}_{\perp}\hk
\end{bmatrix}\begin{bmatrix} 
\f{\Sigma}\hk & 0 \\
0 & \f{\Sigma}_{\perp}\hk 
\end{bmatrix}
\begin{bmatrix} 
\f{V}^{(k)*} \\
\f{V}_{\perp}^{(k)*} 
\end{bmatrix},
\end{equation}
where $\f{U}_k \in \R^{d_1 \times d_1}$ and $\f{V}_k \in \R^{d_2 \times d_2}$, and corresponding submatrices $\f{U}\hk \in \R^{d_1 \times r_k}$, $\f{U}_{\perp}\hk \in \R^{d_1 \times (d_1 -r_k)}$, $\f{V}\hk \in \R^{d_2 \times r_k}$, $\f{V}_{\perp}\hk \in \R^{d_2 \times (d_2 -r_k)}$, $\f{\Sigma}\hk := \diag(\sigma_1\hk,\ldots \sigma_{r_k}\hk)$ and \\
$\f{\Sigma}_{\perp}\hk := \dg(\sigma_{r_k+1}\hk,\ldots \sigma_{d}\hk)$. Furthermore, we denote by  $\mathcal{T}_{r_k}(\f{X}\hk)$ the \emph{best rank-$r_k$ approximation} of $\f{X}\hk$, i.e., 
\begin{equation} \label{eq:best:rankr:approximation}
\mathcal{T}_{r_k}(\f{X}\hk) :=  \argmin_{\f{Z}: \rank(\f{Z}) \leq r_k} \|\f{Z} - \f{X}\hk\| = \f{U}\hk \f{\Sigma}\hk \f{V}^{(k)*},
\end{equation}
where $\| \cdot \|$ can be any unitarily invariant norm, due to the Eckardt-Young-Mirsky theorem \cite{Mirsky60}. Let $\mathcal{M}_{r}:= \{ \f{X}\in \Rdd: \rank(\f{X}) = r\}$ the manifold of rank-$r$ matrices. Given these definitions, let
\begin{equation} \label{eq:def:tangentspace}
\begin{split} 
T_k &:= T_{\mathcal{T}_{r_k}(\f{X}\hk)}\mathcal{M}_{r_k}: =  \left\{ \begin{bmatrix} \f{U}\hk\!\!\! &\!\! \!\f{U}_{\perp}\hk \end{bmatrix} \!\!  \begin{bmatrix} \R^{r_k \times r_k} \!\!&\!\!\! \R^{r_k(d_2-r_k)} \\ \R^{(d_1-r_k)r_k}\! \!&\!\! \!\f{0} \end{bmatrix} \!\! \begin{bmatrix} \f{V}\hk \!\!\!&\!\!\! \f{V}_{\perp}\hk \end{bmatrix}^* \right\} \\
&=  \left\{ \begin{bmatrix} \f{U}\hk\!\!\! &\!\! \!\f{U}_{\perp}\hk \end{bmatrix} \!\!  \begin{bmatrix} \f{M}_1 \!\!&\!\!\! \f{M}_2 \\ \f{M}_3\! \!&\!\! \!\f{0} \end{bmatrix} \!\! \begin{bmatrix} \f{V}\hk \!\!\!&\!\!\! \f{V}_{\perp}\hk \end{bmatrix}^*: \f{M}_1\in \R^{r_k \times r_k}, \f{M}_2 \in \R^{r_k \times (d_2-r_k)}, \f{M}_3 \in \R^{(d_1 -r_k) \times r_k} \text{ arbitrary} \right\}  \\
&= \Big\{\f{U}\hk \Gamma_1 \f{V}^{(k)*} + \f{U}\hk \Gamma_2  \big(\f{I}-\f{V}\hk \f{V}^{(k)*}\big) + \big(\f{I}-\f{U}^{(k)}\f{U}^{(k)*}\big) \Gamma_3 \f{V}^{(k)*}:  \\
& \Gamma_1 \in \R^{r_k \times r_k}, \Gamma_2 \in \R^{r_k \times d_2}, \Gamma_3 \in \R^{d_1 \times r_k}  \Big\}
\end{split}
\end{equation}
be tangent space of the rank-$r_k$ matrix manifold at $\mathcal{T}_{r_k}(\f{X}\hk)$, see also \cite{Vandereycken13} and Chapter 7.5 of \cite{Boumal20}. While it is often more advantageous to computationally represent elements of $T_k$ as in the last representation, it becomes clear from the first equality that $T_k$ is an $r_k(d_1+d_2-r_k) = r_k^2 +r_k(d_1-r_k) +r_k(d_2-r_k)$-dimensional subspace of $\Rdd$. If $\{B_i\}_{i =1}^{r_k(d_1+d_2-r_k)}$ is the standard orthonormal basis of the subspace $T_k$ and $S_k := \R^{r_k(d_1+d_2-r_k)}$, let $P_{T_k}: S_k \to T_k$ be the parametrization operator such that for $\gamma \in S_k$,
\[
P_{T_k}(\gamma) := \sum_{i=1}^{r_k(d_1+d_2-r_k)} \gamma_i B_i = \f{U}\hk \Gamma_1 \f{V}^{(k)*} + \f{U}\hk \Gamma_2  \left(\f{I}-\f{V}\hk \f{V}^{(k)*}\right) + (\f{I}-\f{U}^{(k)}\f{U}^{(k)*}) \Gamma_3 \f{V}^{(k)*}
\]
where $\Gamma_1 \in \R^{r_k \times r_k}, \Gamma_2 \in \R^{r_k \times d_2}, \Gamma_2\f{V}\hk= 0, \Gamma_3 \in \R^{d_1 \times r_k},  \f{U}^{(k)*}\Gamma_3 = 0$. The representation $\gamma \in S_k$ of an element of $T_k$ is called  an \emph{intrinsic representation} \cite{Huang17intrinsic}, and while this is not trivial observation, we note that it is possible to calculate the matrices $\Gamma_1$, $\Gamma_2$ and $\Gamma_3$ of $P_{T_k}(\gamma)$ from just from the knowledge of $\gamma$, $\f{U}\hk \in \R^{d_1 \times r_k}$ and $\f{V}\hk \in \R^{d_2 \times r_k}$ in just $8 (d_1+d_2) r^2 = O(D r^2)$ flops, see Algorithm 7 of \cite{Huang17intrinsic}. Furthermore, we denote by $P_{T_k}^*$ the adjoint operator of $P_{T_k}: S_k \to \Rdd$, and this operator  $P_{T_k}^*:\Rdd \to S_k$ maps a matrix $\f{X}\in \Rdd$ first to $\{\Gamma_1,\Gamma_2,\Gamma_3\}: = \big\{\f{U}^{(k)*} \f{X} \f{V}\hk, \f{U}^{(k)*} \f{X}(\f{I}-\f{V}\hk\f{V}^{(k)*}),(\f{I}-\f{U}\hk\f{U}^{(k)*})\f{X} \f{V}\hk\big\}$ and then via Algorithm 6 of \cite{Huang17intrinsic} to $\gamma = P_{T_k}^*(\f{X}) \in S_k$. The computational complexity of this operator is $O(r \|\f{X}\|_0 + D r^2)$.

Together with $P_{T_k}$ and $P_{T_k}^*$, the definition of the following diagonal operator $\f{D}_{S_k}:S_k \to S_k$ allows us to reformulate and simplify the optimal weight operator $W\hk: \Rdd \to \Rdd$ of \Cref{def:optimalweightoperator} as follows. Recalling that $\f{H}_k \in \Rdd$ is such that $(\f{H}_k)_{ij} := \Big(\max(\sigma_i^{(k)},\epsilon_k) \max(\sigma_j^{(k)},\epsilon_k)\Big)^{-1}$ for all $i$ and $j$, we write for each $\f{Z} \in \Rdd$
\begin{equation} \label{eq:Wk:simplification:1}
\begin{split}
& W\hk(\f{Z}) = \f{U}_k \left[\f{H}_k \circ (\f{U}_k^{*} \f{Z} \f{V}_k)\right] \f{V}_k^{*} \\
&=\begin{bmatrix} 
	\f{U}\hk & \f{U}_{\perp}\hk
\end{bmatrix}
\left(\f{H}_k
\circ
\begin{bmatrix}
\f{U}^{(k)*} \f{Z} \f{V}^{(k)} &  \f{U}^{(k)*} \f{Z} \f{V}_{\perp}^{(k)} \\
\f{U}_{\perp}^{(k)*} \f{Z} \f{V}^{(k)} & \f{U}_{\perp}^{(k)*} \f{Z} \f{V}_{\perp}^{(k)} 
\end{bmatrix}
\right)
\begin{bmatrix} 
	\f{V}^{(k)*} \\ \f{V}_{\perp}^{(k)*}
\end{bmatrix} \\
&= \begin{bmatrix} 
	\f{U}\hk & \f{U}_{\perp}\hk
\end{bmatrix}
\left(
\begin{bmatrix}
\f{H}\hk &  \f{H}_{1,2}\hk \\
\f{H}_{2,1}\hk & \epsilon_k^{-2} \mathbf{1}
\end{bmatrix} 
\circ 
\begin{bmatrix}
\f{U}^{(k)*} \f{Z} \f{V}^{(k)} &  \f{U}^{(k)*} \f{Z} \f{V}_{\perp}^{(k)} \\
\f{U}_{\perp}^{(k)*} \f{Z} \f{V}^{(k)} & \f{U}_{\perp}^{(k)*} \f{Z} \f{V}_{\perp}^{(k)} 
\end{bmatrix}
\right)
\begin{bmatrix} 
	\f{V}^{(k)*} \\ \f{V}_{\perp}^{(k)*}
\end{bmatrix} \\
&= \left(P_{T_k} \f{D}_{S_k} P_{T_k}^* + \epsilon_k^{-2} \left(\f{I} - P_{T_k}P_{T_k}^*\right)\right) \f{Z},
\end{split}
\end{equation}
where the matrices $\f{H}\hk \in \R^{r_k \times r_k}$, $\f{H}_{1,2}\hk \in \R^{r_k \times (d_2 - r_k)}$ and $\f{H}_{2,1}\hk \in \R^{ (d_1 - r_k) \times r_k}$ are defined such that
\begin{equation} \label{eq:H:def:small}
\f{H}_{ij}\hk = \left(\sigma_i^{(k)} \sigma_j^{(k)}\right)^{-1} \text{ for all }i,j \in [r_k],
\end{equation}
$\left(\f{H}_{1,2}\hk\right)_{ij} =\left(\sigma_i^{(k)} \epsilon_k\right)^{-1}$ for all $i \in [r_k]$ and $j \in [d_2-r_k]$ and
$ \left(\f{H}_{2,1}\hk\right)_{ij} = \left(\epsilon_k \sigma_j^{(k)} \right)^{-1}$ 
for all $i \in [d_1-r_k]$ and $j \in [r_k]$. Furthermore, $\mathbf{1}$ in the third line is the $((d_1 - r_k) \times (d_2 -r_k))$-matrix of ones $\mathbf{1}$, and $\f{I}$ is the identity operator. Defining $\f{D}_{S_k}:S_k \to S_k$ implicitly through the last equality, we observe that $\f{D}_{S_k}$ is a diagonal matrix with the entries of $\f{H}\hk$, $\f{H}_{1,2}\hk$ and $\f{H}_{2,1}\hk$ enumerated on its diagonal. Recall that the set of indices corresponding to provided entries is defined as $\Omega = \{(i_{\ell},j_{\ell})\} \subset [d_1] \times [d_2]$,
and $P_{\Omega}: \Rdd \to \R^m$ is the subsampling operator 
\[
P_{\Omega}(\f{Z}) = \sum_{(i_{\ell},j_{\ell}) \in \Omega} \langle e_{i_{\ell}}, \f{Z} e_{j_{\ell}}\rangle,
\]
where $e_{i_{\ell}}$ and $e_{j_{\ell}}$ are the $i_{\ell}$-th and $j_{\ell}$-th standard basis vectors of $\R^{d_1}$ and $\R^{d_2}$, respectively.

\subsection{Implementation of MatrixIRLS}\label{sec:implementation}
The definitions above now enable us to implement the weighted least squares step of \eqref{eq:MatrixIRLS:Xdef} efficiently.
\begin{algorithm}[h ]
\caption{Practical implementation of weighted least squares step of \texttt{MatrixIRLS}} \label{algo:MatrixIRLS:mainstep:implementation}
\begin{algorithmic}
\STATE{\bfseries Input:}  Set $\Omega$, observations $\f{y} \in \R^m$, left and right singular vectors $\f{U}\hk \in \R^{d_1 \times r_k}$, $\f{V}\hk \in \R^{d_2 \times r_k}$ and singular values $\sigma_1\hk,\ldots,\sigma_{r_k}\hk$, smoothing parameter $\epsilon_k$, projection $\gamma_k^{(0)} = P_{T_k}^*P_{T_{k-1}}(\gamma_{k-1}) \in S_k$ of solution $\gamma_{k-1} \in S_{k-1}$ of linear system \cref{eq:gamma:system} for previous iteration $k-1$.
\end{algorithmic}
\begin{algorithmic}[1]
\STATE  Compute $\f{h}_k^0:= P_{T_k}^*P_{\Omega}^*\left(\f{y}\right) - \left(\epsilon_k^{2} \left(\f{D}_{S_k}^{-1}- \epsilon_k^{2} \f{I}_{S_k}\right)^{-1} + P_{T_k}^* P_{\Omega}^* P_{\Omega}  P_{T_k}\right) \gamma_k^{(0)}  \in S_k$.
\STATE Solve 	
\begin{equation} \label{eq:gamma:system}
	\left(\epsilon_k^{2} \left(\f{D}_{S_k}^{-1}- \epsilon_k^{2} \f{I}_{S_k}\right)^{-1} + P_{T_k}^* P_{\Omega}^* P_{\Omega}  P_{T_k}\right) \Delta\f{\gamma}_{k} =  \f{h}_k^0	
\end{equation}
	 for $\Delta\gamma_k \in S_k$ by the \emph{conjugate gradient} method \cite{HestenesStiefel52,Meurant}.
\STATE Compute $\gamma_k = \gamma_k^{(0)} + \Delta\gamma_k$.
\STATE Compute residual $\f{r}_{k+1}:= \f{y}-P_{\Omega} P_{T_k} (\f{\gamma}_{k}) \in \R^m$.
\end{algorithmic}
\begin{algorithmic}
\STATE{\bfseries Output:} $\f{r}_{k+1} \in \R^m$ and $\gamma_k \in S_k$.
\end{algorithmic}
\end{algorithm}

In the following lemma, we show that \Cref{algo:MatrixIRLS:mainstep:implementation} indeed computes the solution of the weighted least squares step  \eqref{eq:MatrixIRLS:Xdef} of \Cref{algo:MatrixIRLS}.
\begin{lemma} \label{lemma:weightedleastsquares:implement}
If $\f{r}_{k+1} \in \R^m$ and $\gamma_k \in S_k$ is the output of \Cref{algo:MatrixIRLS:mainstep:implementation}, then $\f{X}\hkk$ as in step  \eqref{eq:MatrixIRLS:Xdef} of \Cref{algo:MatrixIRLS} fulfills
\[
\f{X}\hkk = P_{\Omega}^*(\f{r}_{k+1}) + P_{T_k}(\gamma_k).
\]
\end{lemma}
\begin{proof}[{Proof of \Cref{lemma:weightedleastsquares:implement}}]
Let $W\hk:\Rdd \to \Rdd$ be the weight operator of \Cref{def:optimalweightoperator}. As $W\hk$ is a positive definite operator, it holds that 
 the minimizer 
\[
\f{X}\hkk = \argmin\limits_{P_{\Omega}(\f{X})=\f{y}} \langle \f{X}, W^{(k)}(\f{X}) \rangle
\] 
of \eqref{eq:MatrixIRLS:Xdef} is unique, and it is well-known \cite{Bjoerck96} that the solution of this linearly constrained weighted least squares problem can be written such that
\begin{equation} \label{eq:weightedleastsquares:formula}
\f{X}\hkk =  (W\hk)^{-1}P_{\Omega}^*\left(P_{\Omega}(W\hk)^{-1} P_{\Omega}^*\right)^{-1}(\f{y}),
\end{equation}
where $(W\hk)^{-1}: \Rdd \to \Rdd$ is the \emph{inverse} of the weight operator $W\hk$: This inverse exists as $W\hk$ is self-adjoint and positive definite, which can be shown by realizing that its $(d_1 d_2 \times d_1 d_2)$-matrix representation has eigenvectors $v_i\hk \otimes u_j\hk$, where $v_i\hk \in \R^{d_1}$ and $u_{j}\hk \in \R^{d_2}$ are columns of $\f{V}_k$ and $\f{U}_k$, respectively, and that the eigenvalues of $W\hk$ are just the entries of $\f{H}\hk$.

From this eigendecomposition of $W\hk$ if further follows that the action of $(W\hk)^{-1}:\Rdd \to \Rdd$ is such that for each $\f{Z}\in \Rdd$,
\[
(W\hk)^{-1}(\f{Z}) = \f{U}_k \left[\f{H}_k^{-1} \circ (\f{U}_k^{*} \f{Z} \f{V}_k)\right] \f{V}_k^{*},
\]
where $\f{H}_k^{-1} \in \R^{d_1 \times d_2}$ is the matrix whose entries are the \emph{entrywise} inverse of the entries of $\f{H}_k$. With the same argument as in \eqref{eq:Wk:simplification:1}, we can rewrite $(W\hk)^{-1}$ such that
\begin{equation*}
(W\hk)^{-1} = P_{T_k}\f{D}_{S_k}^{-1} P_{T_k}^* + \epsilon_k^{2} \left(\f{I} - P_{T_k}P_{T_k} ^*  \right)  =P_{T_k} \left(\f{D}_{S_k}^{-1}- \epsilon_k^{2} \f{I}_{S_k}\right)P_{T_k}^* + \epsilon_k^{2} \f{I},
\end{equation*}
where $\f{I}_{S_k}$ is the identity on $S_k$ and $\f{D}_{S_k}^{-1}$ is the diagonal matrix that is the inverse of $\f{D}_{S_k}$ in \eqref{eq:Wk:simplification:1}.

Since $P_{\Omega} P_{\Omega}^* = \f{I}_m$, we can use this representation of $(W\hk)^{-1}$ to write
\[
\left(P_{\Omega}(W\hk)^{-1} P_{\Omega}^*\right)^{-1} = \left( P_{\Omega} P_{T_k} \left(\f{D}_{S_k}^{-1}- \epsilon_k^{2} \f{I}_{S_k}\right)P_{T_k}^* P_{\Omega}^* + \epsilon_k^{2} \f{I}_m \right)^{-1}.
\]
Using the \emph{Sherman-Morrison-Woodbury} formula \cite{Woodbury50} 
\begin{equation*}
(\f{E} \f{C} \f{F}^* + \f{B})^{-1} = \f{B}^{-1} - \f{B}^{-1} \f{E} ( \f{C}^{-1} + \f{F}^* \f{B}^{-1} \f{E} )^{-1} \f{F}^* \f{B}^{-1}
\end{equation*}
for $\f{B}:= \epsilon_k^{2} \f{I}_m$, $\f{C} := \left(\f{D}_{S_k}^{-1}- \epsilon_k^{2} \f{I}_{S_k}\right)$ and $\f{E}= \f{F} := P_{\Omega} P_{T_k}$, we obtain that
\begin{equation} \label{eq:MatrixIRLS:Woodbury:application}
\left(P_{\Omega}(W\hk)^{-1} P_{\Omega}^*\right)^{-1}= \epsilon_k^{-2} \f{I} - \epsilon_k^{-2}P_{\Omega} P_{T_k}\left(\epsilon_k^{2} \f{C}^{-1} + P_{T_k}^* P_{\Omega}^* P_{\Omega} P_{T_k}\right)^{-1} P_{T_k}^* P_{\Omega}^* = \epsilon_k^{-2} \f{I} - \epsilon_k^{-2}P_{\Omega} P_{T_k}\f{M}^{-1} P_{T_k}^* P_{\Omega}^*
\end{equation}
with linear system matrix $\f{M}:=\epsilon_k^{2} \f{C}^{-1} + P_{T_k}^* P_{\Omega}^* P_{\Omega} P_{T_k} $, noting that $\f{C} = \left(\f{D}_{S_k}^{-1}- \epsilon_k^{2} \f{I}_{S_k}\right)$ is invertible since $(\f{H}_{ij}\hk)^{-1} = \sigma_i^{(k)} \sigma_j^{(k)} > \epsilon_k^2$ for all $i,j \in [r_k]$ and since $(\f{D}_{ii}\hk)^{-1} = \sigma_i^{(k)} \epsilon_k > \epsilon_k^2$ for all $i \in [r_k]$.

Next, we note that the definition $\gamma_k = \gamma_k^{(0)} + \Delta\gamma_k$ and \eqref{eq:gamma:system} implies that
\begin{equation} \label{eq:gammak:Mmin1y}
\gamma_k = \gamma_k^{(0)} + \Delta\gamma_k =  \gamma_k^{(0)} + \f{M}^{-1}\f{h}_k^0 = \gamma_k^{(0)} + \f{M}^{-1} \left( P_{T_k}^*P_{\Omega}^*\left(\f{y}\right) - \f{M} \gamma_k^{(0)}\right) =  \f{M}^{-1}P_{T_k}^*P_{\Omega}^*\left(\f{y}\right).
\end{equation}
Inserting this into \eqref{eq:MatrixIRLS:Woodbury:application}, we see that the residual $\f{r}_{k+1}$ of \Cref{algo:MatrixIRLS:mainstep:implementation} satisfies
\[
\f{z} := \left(P_{\Omega}(W\hk)^{-1} P_{\Omega}^*\right)^{-1} (\f{y}) = \epsilon_k^{-2} \f{y} - \epsilon_k^{-2}P_{\Omega} P_{T_k}\f{M}^{-1} P_{T_k}^* P_{\Omega}^*(\f{y}) = \epsilon_k^{-2}\left(\f{y} -  P_{\Omega} P_{T_k} (\gamma_k)\right) =  \epsilon_k^{-2} \f{r}_{k+1}.
\]
Thus, we obtain the representation 
\begin{equation} \label{eq:Z:calculation:computational}
\begin{split}
\f{X}\hkk &= (W\hk)^{-1}P_{\Omega}^* \left( P_{\Omega} (W\hk)^{-1}P_{\Omega}^*\right)^{-1}(\f{y}) = (W\hk)^{-1}P_{\Omega}^* (\f{z}) \\
&= \epsilon_k^{2} \left[\f{I} + \epsilon_k^{-2} P_{T_k} \f{C} P_{T_k}^*  \right]P_{\Omega}^* (\f{z}) = \epsilon_k^2 P_{\Omega}^* (\f{z}) +   P_{T_k} \f{C}  P_{T_k}^* P_{\Omega}^*(\f{z}) \\
&= \epsilon_k^{2}P_{\Omega}^* (\f{z}) + P_{T_k} \f{C} \f{C}^{-1} \left(\epsilon_k^{2} \f{C}^{-1} + P_{T_k}^* P_{\Omega}^* P_{\Omega} P_{T_k}\right)^{-1} P_{T_k}^* P_{\Omega}^*(\f{y}) \\
&= P_{\Omega}^* (\f{r}_{k+1}) + P_{T_k}(\gamma_k),
\end{split}
\end{equation}
for $\f{X}\hkk$, using \eqref{eq:PTkPOmegaz} in the fifth equality, where \eqref{eq:PTkPOmegaz} represents
\begin{equation} \label{eq:PTkPOmegaz}
\begin{split}
P_{T_k}^* P_{\Omega}^*(\f{z}) &= \epsilon_k^{-2} P_{T_k}^* P_{\Omega}^*(\f{r}_{k+1}) = \epsilon_k^{-2}  \left(P_{T_k}^* P_{\Omega}^*(\f{y})-  P_{T_k}^* P_{\Omega}^*P_{\Omega} P_{T_k} (\gamma_k)  \right)\\
&= \epsilon_k^{-2}  \left( P_{T_k}^* P_{\Omega}^*(\f{y})-  P_{T_k}^* P_{\Omega}^*P_{\Omega} P_{T_k} \f{M}^{-1} P_{T_k}^* P_{\Omega}^*(\f{y})\right) \\
&= \epsilon_k^{-2}  \left( P_{T_k}^* P_{\Omega}^*(\f{y})-  P_{T_k}^* P_{\Omega}^*P_{\Omega} P_{T_k}  \left(\epsilon_k^{2} \f{C}^{-1} + P_{T_k}^* P_{\Omega}^* P_{\Omega} P_{T_k}\right)^{-1} P_{T_k}^* P_{\Omega}^*(\f{y})\right)    \\
&= \epsilon_k^{-2} \left(P_{T_k}^* P_{\Omega}^*(\f{y})-  \left(P_{T_k}^* P_{\Omega}^*P_{\Omega} P_{T_k}  \pm \epsilon_k^{2} \f{C}^{-1}  \right) \left(\epsilon_k^{2} \f{C}^{-1} + P_{T_k}^* P_{\Omega}^* P_{\Omega} P_{T_k}\right)^{-1} P_{T_k}^* P_{\Omega}^*(\f{y})\right)  \\
&= \f{C}^{-1}  \left(\epsilon_k^{2} \f{C}^{-1} + P_{T_k}^* P_{\Omega}^* P_{\Omega} P_{T_k}\right)^{-1} P_{T_k}^* P_{\Omega}^*(\f{y}),
\end{split}
\end{equation}
which uses \eqref{eq:gammak:Mmin1y} in the third equality and the definition of $\f{M}$ in the fourth equality.

This finishes the proof.
\end{proof}

With \Cref{lemma:weightedleastsquares:implement}, we are now able to prove \Cref{thm:MatrixIRLS:computationalcost:Xkk}.

\begin{proof}[{Proof of \Cref{thm:MatrixIRLS:computationalcost:Xkk}}]
Based on \Cref{lemma:weightedleastsquares:implement}, we can compute the representation $\f{X}\hkk = P_{\Omega}^*(\f{r}_{k+1}) + P_{T_k}(\gamma_k)$ for $\f{X}\hkk$ using \Cref{algo:MatrixIRLS:mainstep:implementation}. By the assumption of \Cref{thm:MatrixIRLS:computationalcost:Xkk}, we know that $r =\widetilde{r}=r_k$.

The main computational cost in \Cref{algo:MatrixIRLS:mainstep:implementation} lies in the application of the operators $(P_{\Omega} P_{T_k})^* = P_{T_k}^* P_{\Omega}^*:\R^m \to S_k$, $P_{\Omega} P_{T_k}: S_k \to \R^m$ and $\left(\f{D}_{S_k}^{-1}- \epsilon_k^{2} \f{I}_{S_k}\right)^{-1} : S_k \to S_k$. The application of $\left(\f{D}_{S_k}^{-1}- \epsilon_k^{2} \f{I}_{S_k}\right)^{-1}$ has a time complexity of $r(d_1+d_2-r) = O(r D)$ as the operator is diagonal. The action of $P_{T_k}^* P_{\Omega}^*$ and $P_{\Omega} P_{T_k}$ can be computed as in \Cref{algo:implement:PomegaPTkstar} and \Cref{algo:implement:PomegaPTk}, respectively.

\begin{algorithm}[h]
\caption{Implementation of $P_{T_k}^* P_{\Omega}^*: \R^m \to S_k$} \label{algo:implement:PomegaPTkstar}
\begin{algorithmic}
\STATE{\bfseries Input:}  Argument vector $\f{y} \in \R^m$, index set $\Omega$, left and right singular vectors $\f{U}\hk \in \R^{d_1 \times r_k}$, $\f{V}\hk \in \R^{d_2 \times r_k}$. 
\end{algorithmic}
\begin{algorithmic}[1]
\STATE $\f{A}_1 = \f{U}^{(k)*}P_{\Omega}^*(\f{y}) \in \R^{r _k\times d_2}$. \hfill $\triangleright$ $m r_k$ flops
\STATE $\f{A}_2 = P_{\Omega}^*(\f{y}) \f{V}^{(k)} \in \R^{d_1\times r_k}$. \hfill $\triangleright$ $m r_k$ flops  
\STATE $\Gamma_1 =  \f{A}_1  \f{V}^{(k)}  \in \R^{r_k \times r_k}$. \hfill $\triangleright$ $d_2 r_k^2$ flops
\STATE $\Gamma_2 = \f{A}_1 - \Gamma_1  \f{V}^{(k)*} \in \R^{r _k\times d_2}$. \hfill   $\triangleright$ $2 d_2 r_k^2$ flops
\STATE $\Gamma_3 = \f{A}_2 - \f{U}\hk \Gamma_1 \in \R^{d_1\times r_k}$. \hfill   $\triangleright$ $d_1 r_k^2$ flops
\STATE Apply Algorithm 6 of \cite{Huang17intrinsic} to compute\footnotemark \;$\gamma \in S_k$ from $\{\Gamma_1,\Gamma_2,\Gamma_3\}$. \hfill   $\triangleright$ $4 r_k^2(d_1+d_2-r_k)+2 r_k^2$ flops
\end{algorithmic}
\begin{algorithmic}
\STATE{\bfseries Output:} $\gamma \in S_k$.
\end{algorithmic}
\end{algorithm}
\footnotetext{In the notation of  \cite{Huang17intrinsic}, we have the correspondence that $U = \f{U}\hk$, $V=\f{V}\hk$, $\dot S = \Gamma_1$, $S \dot V^T = \Gamma_2$ and $\dot U S = \Gamma_3$.}

\begin{algorithm}[h]
\caption{Implementation of $P_{\Omega} P_{T_k}: S_k \to  \R^m$} \label{algo:implement:PomegaPTk}
\begin{algorithmic}
\STATE{\bfseries Input:}  Argument vector $\f{\gamma} \in S_k$, index set $\Omega$, left and right singular vectors $\f{U}\hk \in \R^{d_1 \times r_k}$, $\f{V}\hk \in \R^{d_2 \times r_k}$. 
\end{algorithmic}
\begin{algorithmic}[1]
\STATE Apply Algorithm 7 of \cite{Huang17intrinsic} to compute $\{\Gamma_1,\Gamma_2,\Gamma_3\}$ from \;$\gamma \in S_k$.  \hfill   $\triangleright$ $8 (d_1+d_2) r_k^2$ flops
\STATE $\f{N}_1 =  \f{U}\hk \Gamma_1 \in \R^{d_1 \times r_k}$. \hfill $\triangleright$ $d_1 r_k^2$ flops
\STATE $\f{N}_2 = \Gamma_2 \f{V}\hk \in \R^{r_k \times r_k}$. \hfill $\triangleright$ $d_2 r_k^2$ flops
\STATE $\f{N}_3 = \f{U}^{(k)*} \Gamma_3 \in \R^{r_k \times r_k}$. \hfill $\triangleright$ $d_1 r_k^2$ flops
\STATE $\f{N}_4 = \f{N}_1 - \f{U}\hk \f{N}_2 + \Gamma_3 - \f{U}\hk \f{N}_3 \in \R^{d_1 \times r_k}$. \hfill $\triangleright$ $d_1 r_k^2$ flops
\STATE Set $\f{y}$ such that $\f{y}_{\ell} =  P_{\Omega}( \f{N}_4 \f{V}^{(k)*})_{\ell} = \sum_{k=1}^r (\f{N}_4)_{i_\ell,k} (\f{V}^{(k)})_{j_{\ell},k}$ for each $\ell \in [m]$. \hfill $\triangleright$ $2 m r_k$ flops
\STATE Set $\f{y}$ such that $\f{y}_{\ell} = \f{y}_{\ell} +  P_{\Omega}( \f{U}\hk \Gamma_2)_{\ell} = \sum_{k=1}^r ( \f{U}\hk)_{i_\ell,k} (\Gamma_2^*)_{j_{\ell},k}$ for each $\ell \in [m]$. \hfill $\triangleright$ $2 m r_k$ flops
\end{algorithmic}
\begin{algorithmic}
\STATE{\bfseries Output:} $\f{y} \in \R^m$.
\end{algorithmic}
\end{algorithm}

Using \Cref{algo:implement:PomegaPTkstar} and \Cref{algo:implement:PomegaPTk}, we see that the first step of \Cref{algo:MatrixIRLS:mainstep:implementation} has a time complexity of $O(m r + r^2 D)$ flops, and each inner iteration of the conjugate gradient method in step 2 of \Cref{algo:MatrixIRLS:mainstep:implementation} has likewise a time complexity of $O(m r +r^2 D)$ flops. Finally, step 3 takes also $O(m r + r^2 D)$ flops.

We observe that the linear system \eqref{eq:gamma:system} is positive definite, since $\epsilon_k^{2} \left(\f{D}_{S_k}^{-1}- \epsilon_k^{2} \f{I}_{S_k}\right)^{-1}$ is diagonal with positive entries and since $P_{T_k}^* P_{\Omega}^* P_{\Omega}  P_{T_k}$ is a symmetric, positive definite operator. Thus, it is possible to use the conjugate gradient (CG) method, whose main step applies the three operators above at each iteration. It is known that in general, the CG method terminates with the exact solution $\gamma_k$ after at most $N_{\text{CG\_inner}} = \dim(S_k)=  r (d_1 +d_2-r)$ iterations. However, if the system matrix 
\[
\frac{\epsilon_k^{2} \f{I}}{\f{D}_{S_k}^{-1}- \epsilon_k^{2} \f{I}_{S_k}} + P_{T_k}^* P_{\Omega}^* P_{\Omega}  P_{T_k} 
\]
is well-conditioned (for example, with a condition number bounded by a small constant), the CG method can be used as an inexact solver of \eqref{eq:gamma:system},
returning very high precision approximate solutions \emph{after a constant number} of iterations $N_{\text{CG\_inner}}$, thus, amounting to a time complexity of $O \left( (m r + r^2 D) \cdot N_{\text{CG\_inner}} \right)$. We refer to \Cref{thm:wellconditioning} for a result that ensures this well-conditioning of the system matrix under certain conditions.

Since we have obtained the representation  $\f{X}\hkk = P_{\Omega}^*(\f{r}_{k+1}) + P_{T_k}(\gamma_k)$ of $\f{X}\hkk$ (which is approximate by nature if an iterative solver such as CG is used), we can now apply Algorithm 7 of \cite{Huang17intrinsic} to $\gamma_k \in S_k$ to compute $\f{M}_1\hkk \in \R^{d_2 \times r}$ and $\f{M}_2\hkk \in \R{d_1 \times r}$ such that

\[
\f{X}\hkk = P_{\Omega}^*(\f{r}_{k+1})  +\f{U}\hk \f{M}_{1}^{(k+1)*} + \f{M}_2^{(k+1)} \f{V}^{(k)*},
\]
by setting $\f{M}_{1}^{(k+1)}= \f{V}^{(k)*} \Gamma_1^* + \Gamma_2^{*} \in \R^{d_2 \times r}$ and $\f{M}_2^{(k+1)} = \Gamma_3 \in \R^{d_1 \times r}$ if $\{\Gamma_1,\Gamma_2,\Gamma_3\}$ is the output of Algorithm 7 of \cite{Huang17intrinsic}. This last step has a time complexity of $8 (d_1+d_2) r^2+ r^2 d_2+ r d_2 = O(r^2 D)$ flops, so that we obtain a total time complexity of $O \left( (m r + r^2 D) \cdot N_{\text{CG\_inner}} \right)$ for computing $\f{X}\hkk$.

Including $\f{r}_{k+1} \in \R^m$, $\f{U}\hk \in \R^{d_1 \times r}$ and $\f{V}\hk \in \R^{d_2 \times r}$ this amounts to a representation of $\f{X}\hkk$ with a space complexity of $m + 2r(d_1 + d_2) = O(m + r D)$. Since also the space requirement of the intermediate variables does not exceed $O(m + r D)$, this finishes the proof of \Cref{thm:MatrixIRLS:computationalcost:Xkk}.
\end{proof}

Based on \Cref{thm:MatrixIRLS:computationalcost:Xkk}, we see that in \texttt{MatrixIRLS}, it is \emph{never} necessary to work with full $(d_1 \times d_2)$-matrices. In order to update the smoothing parameter $\epsilon_{k+1}$ as in \eqref{eq:MatrixIRLS:epsdef}, we need the $\widetilde{r}+1$-th singular value of $\f{X}\hkk$. Furthermore, to update the information to define the weight operator $W\hkk$, we need to find the number $r_{k+1}$ of singular values of $\f{X}\hkk$ that are larger than $\epsilon_{k+1}$, and their corresponding left and right singular vectors. Due to the definition of \eqref{eq:MatrixIRLS:epsdef}, it is clear that $r_{k+1} \geq \widetilde{r}$, and $r_{k+1} = \widetilde{r}$ for each iteration $k$ when the smoothing parameter decreases, i.e., for each $k$ with $\epsilon_{k+1} < \epsilon_{k}$.  

In our experiments on exact completion of rank-$r$ matrices with oracle knowledge of the rank $r$, we chose $\widetilde{r}=r$, and we observe that in most iterations of most of our experiments with \texttt{MatixIRLS}, it holds that $r_k = r$. In rare cases, we observe that $r_k = c r$ with a small constant $c >1$ for a small number of the iterations $k$. On the other hand, we do not have a theoretical statement that bounds $r_k$ in general.

The $r_k$ singular values and vector pairs of $\f{X}\hkk$ can be computed efficiently using any suitable method that uses matrix-matrix or matrix-vector products, as, for example, matrix-vector products with $\f{X}\hkk$ can be calculated in $m+ 2 r_k (d_2+d_1) = O(m +r_k D)$ due to its ``sparse + low-rank'' structure. Such a suitable method can be a (randomized) block Lanczos method \cite{Golub77,MuscoMusco15,Yuan18}. In our implementation we used a version of the method described in \cite{MuscoMusco15}, which allows us to compute a good approximation of the needed singular triplets in a time complexity of $O(m r_k +r_k^2 D)$ \cite{Yuan18}.

\begin{remark}
The authors of the recent preprint \cite{Luo20} propose an iterative method for rank-constrained least squares that has certain similarities to ours from a computational point of view. In particular, it can be shown that Algorithm 1 (called \texttt{RISRO}) of \cite{Luo20} is equivalent to solving the equation
\[
\left(P_{T_k}^* P_{\Omega}^* P_{\Omega}  P_{T_k}\right) \f{z}\hk =  P_{T_k}^* P_{\Omega}^* (\f{y})
\]
for $\f{z\hk} \in S_k$ in our notation, if specialized to matrix completion (see Theorem 2 and equation (58) of \cite{Luo20}), if the parameter $r_k$ coincides with the rank parameter of the rank-constraint least squares. Comparing this to \eqref{eq:gamma:system}, it can be observed that the Riemannian Gauss-Newton step of \cite{Luo20} corresponds to choosing the weights such that the entries of $\f{H}\hk$, $\f{H}_{1,2}\hk$ and $\f{H}_{2,1}\hk$ are all chosen equal and the smoothing parameter is chosen such that $\epsilon_k = 0$, rendering $\epsilon_k^{2} \left(\f{D}_{S_k}^{-1}- \epsilon_k^{2} \f{I}_{S_k}\right)^{-1} = 0$ in \eqref{eq:gamma:system}. 

On the other hand, the interpretation of \texttt{MatrixIRLS} and \texttt{RISRO} is quite different, as \texttt{MatrixIRLS} can be interpreted as a majorize-minimize method for smoothed log-det objectives \cref{eq:smoothing:Fpeps} with updated smoothing, whereas \texttt{RISRO} is harder to be interpreted with respect to a rank surrogate objective function, but rather follows a rank-constrained least squares framework.
\end{remark}

\section{Proof of \Cref{cor:MatrixIRLS:localconvergence:matrixcompletion}} \label{sec:proof:local:guarantee}
In this section, we prove under a random sampling model on the location of the provided entries, \texttt{MatrixIRLS} converges locally to a low-rank completion of the data with high probability, and that the convergence rate is quadratic, as described in \Cref{cor:MatrixIRLS:localconvergence:matrixcompletion}.

First, we shortly elaborate on our notion of \emph{incoherence} (see \Cref{def:incoherence}), which quantifies the alignment of the standard basis $(e_i e_j^*)_{i=1,j=1}^{d_1,d_2}$ of $\Rdd$ with the \emph{tangent space} onto the rank-$r$ manifold at rank-$r$ matrix at a specific rank-$r$ matrix, that we use in \Cref{cor:MatrixIRLS:localconvergence:matrixcompletion}.

\begin{remark} \label{remark:incoherence:1}
We note that the assumption that a rank-$r$ matrix $\f{X} \in \Rdd$ is $\mu_0$-incoherent according to \Cref{def:incoherence} is \emph{weaker} than similar assumptions described in Definition 1.2, A0 and A1 of \cite{CR09} and Definition 1 and Theorem 2 of \cite{recht}, and even than the assumption (2) of \cite{Chen15}, which is the weakest available incoherence condition in the literature that is used for showing successful completion by nuclear norm minimization. More precisely, \cite{Chen15} calls a matrix $\f{X}$ $\mu_0$-incoherent if
\begin{equation} \label{eq:Chen:incoherence}
\max_{1 \leq i\leq d_1} \| \f{U}^* e_i\|_2 \leq \sqrt{\frac{\mu_0 r}{d_1}} \quad \text{ and } \quad \max_{1 \leq j \leq d_2} \| \f{V}^* e_j\|_F \leq \sqrt{\frac{\mu_0 r}{d_2}}.
\end{equation}
In fact, condition \eqref{eq:Chen:incoherence} is \emph{stronger} than \eqref{eq:MatrixIRLS:incoherence}. If $\f{U}\in \R^{d_1 \times r}$ and $\f{V} \in \R^{d_2 \times r}$ are the left and right singular matrices corresponding to the $r$ non-zero singular values of $\f{X}$, we can write the projection operator $\y{P}_T: \Rdd \to \Rdd$ that projects onto the tangent space $T$ such $\y{P}_T(\f{Z}) = \f{U} \f{U}^{*} \f{Z} + \f{Z} \f{V} \f{V}^* -  \f{U} \f{U}^{*} \f{Z} \f{V} \f{V}^*$. Therefore, it can be seen that
\[
\begin{split}
\|\y{P}_T(e_i e_j^*)\|_F^2 &= \|\f{U}\f{U}^* e_i e_j^* + e_i e_j^*\f{V}\f{V}^*-\f{U}\f{U}^*e_i e_j^* \f{V}\f{V}^*\|_F^2 = \|\f{U}\f{U}^*e_i e_j^*(\f{I}-\f{V}\f{V}^*) + e_i e_j^*\f{V}\f{V}^* \|_F^2 \\
&= \|\f{U}\f{U}^*e_i e_j^*(\f{I}-\f{V}\f{V}^*)\|_F^2 + \|e_i e_j^*\f{V}\f{V}^* \|_F^2 \leq \|\f{U}\f{U}^* e_i e_j^*\|_F^2 \|\f{I}-\f{V}\f{V}^*\|^2 + \|e_i e_j^*\f{V}\f{V}^* \|_F^2 \\
&\leq \|\f{U}^*e_i e_j^*\|_F^2 +  \|e_i e_j^*\f{V}\|_F^2 = \|\f{U}^*e_i\|_2^2 + \|\f{V}^* e_j\|_2^2 \leq \frac{\mu_0 r}{d_1} + \frac{\mu_0 r}{d_2} \leq  \frac{\mu_0 r(d_1 + d_2)}{d_1 d_2}
\end{split} 
\]
for any $i \in [d_1]$, $j \in [d_2]$, if \eqref{eq:Chen:incoherence} is fulfilled, which holds since
\[
\|\f{U}^*e_i e_j^*\|_F^2 = \trace( e_j e_i^* \f{U} \f{U}^* e_i e_j^*) = \trace(e_i^* \f{U} \f{U}^* e_i) = e_i^* \f{U} \f{U}^* e_i = \|\f{U}^* e_i\|_2^2
\]
and similarly $\|e_i e_j^*\f{V}\|_F^2 = \|\f{V}^* e_j\|_2^2$.
\end{remark}

\subsection{Interplay between sampling operator and tangent space}

In the statement of \Cref{cor:MatrixIRLS:localconvergence:matrixcompletion}, we assume that the index set $\Omega$ is \emph{drawn uniformly at random without replacement}. In our proof below, however, we use a sampling model on the locations $\Omega = (i_{\ell},j_{\ell})_{\ell=1}^m$ corresponding to \emph{independent sampling with replacement}. It is well-known (see, e.g., Proposition 3 of \cite{recht}) that the statement then carries over to the above model of sampling without replacement.

As a preparation for our proof, we recall well-known result from \cite{recht} that bounds the number of repetitions of each location in $\Omega$ under the random sampling model with replacement.
\begin{lemma}[{Proposition 5 of \cite{recht}}]\label{lem:repetition}
Let $D = \max(d_1,d_2)$ and $\beta > 1$, let $\Omega=(i_{\ell},j_{\ell})_{\ell=1}^m$ be a multiset of double indices from $[d_1] \times [d_2]$ fulfilling $m < d_1 d_2$ that are sampled independently with replacement. 
Then with probability at least $1-D^{2-2\beta}$, the maximal number of repetitions of any entry in $\Omega$ is less than 
$\frac{8}{3}\beta\log(D)$ for $D\geq 9$ and $\beta > 1$. Consequently, we have that with probability of at least $1-D^{2-2\beta}$, the operator $\y{R}_{\Omega}: \Rdd \to \Rdd$ defined such that
\begin{equation} \label{eq:ROmega:def}
\y{R}_{\Omega}(\f{X}) := P_{\Omega}^*(P_{\Omega}(\f{X})) = \sum_{\ell=1}^m \langle e_{i_{\ell}} e_{j_{\ell}}^*, \f{X} \rangle e_{i_{\ell}} e_{j_{\ell}}^*
\end{equation}
fulfills
\[
	\|\y{R}_\Omega\|_{S_\infty} \leq \frac{8}{3}\beta\log(D).
\]
\end{lemma}

Next, we use a lemma of \cite{recht} that can be seen as a result of a \emph{local restricted isometry property}.
While the proof is fairly standard, we provide it for completeness since we use the weaker incoherence definition of \Cref{def:incoherence} instead of the incoherence notions of \cite{recht,Chen15}.

\begin{lemma}[{Theorem 6 of \cite{recht}}] \label{lem:RIPPS}
Let $0 < \epsilon \leq \frac{1}{2}$, let $\f{X}^0 \in \Rdd$ be a $\mu_0$-incoherent matrix whose tangent space $T_0 = T_{\f{X}^0}$ onto the rank-$r$ manifold $T_0 = T_{\f{X}^0}\mathcal{M}_{r}$ (see \eqref{eq:def:tangentspace}) fulfills \eqref{eq:MatrixIRLS:incoherence} and $\mathcal{R}_{\Omega}: \Rdd \to \Rdd$ be defined as in \eqref{eq:ROmega:def} from $m$ independent uniformly sampled locations. Let $\mathcal{P}_{T_0}: \Rdd \to \Rdd$ be the projection operator associated to $T_0$. Then 
\begin{equation}\label{lem:RIP_tangent}
\left\|\frac{d_1 d_2}{m}\mathcal{P}_{T_0}\mathcal{R}_{\Omega}\mathcal{P}_{T_0}-\mathcal{P}_{T_0}\right\|_{S_{\infty}}
\leq \varepsilon 
\end{equation}
holds with probability at least $1-(d_1 + d_2)^{-2}$ provided that 
\begin{equation} \label{eq:localRIP:samplecomplexity:matrixcompletion}
m \geq \frac{7}{\varepsilon^2} \mu_0 r (d_1 + d_2) \log(d_1 + d_2).
\end{equation}
\end{lemma}

\begin{proof}[{Proof of \Cref{lem:RIPPS}}]
First we define the family of operators $\y{Z}_{\ell}, \widetilde{\y{Z}}_{\ell}: \Rdd \to \Rdd$ such that for $\f{X} \in \Rdd$,
\[
\mathcal{Z}_{\ell}(\f{X}) := \frac{d_1 d_2}{m} \langle e_{i_{\ell}} e_{j_{\ell}}^*, \y{P}_{T_0}(\f{X}) \rangle \y{P}_{T_0}(e_{i_{\ell}} e_{j_{\ell}}^*) - \frac{1}{m} \y{P}_{T_0}(\f{X}) := \frac{d_1 d_2}{m} \widetilde{\y{Z}}_{\ell}(\f{X}) - \frac{1}{m} \y{P}_{T_0}(\f{X})
\]
for any $\ell \in [m]$. Then 
\begin{equation} \label{eq:recht:localRIP:1}
\mathbb{E}[\mathcal{Z}_{\ell}] = \frac{1}{d_1 d_2} \sum_{i=1}^{d_1} \sum_{j=1}^{d_2} \frac{d_1 d_2}{m} \langle e_{i} e_{j}^*, \y{P}_{T_0}(\cdot) \rangle \y{P}_{T_0}(e_{i} e_{j}^*)  - \frac{1}{m} \mathcal{P}_{T_0} = \frac{1}{d_1 d_2} \frac{d_1 d_2}{m} \mathcal{P}_{T_0} \f{I} \mathcal{P}_{T_0} - \frac{1}{m} \mathcal{P}_{T_0} = 0.
\end{equation}
Since for $\f{X} \in \R^{d_1 \times d_2}$
\[
\langle  e_{i_{\ell}}e_{j_{\ell}}^* ,\mathcal{P}_{T_0}(\f{X}) \rangle \mathcal{P}_{T_0} (e_{i_{\ell}}e_{j_{\ell}}^*) = \langle \mathcal{P}_{T_0} (e_{i_{\ell}}e_{j_{\ell}}^*) ,\f{X} \rangle \mathcal{P}_{T_0} (e_{i_{\ell}}e_{j_{\ell}}^*),
\]
we obtain
\[
\| \langle e_{i_{\ell}}e_{j_{\ell}}^* ,\mathcal{P}_{T_0}(\f{X}) \rangle \mathcal{P}_{T_0} (e_{i_{\ell}}e_{j_{\ell}}^*)\|_F \leq \left|\langle \mathcal{P}_{T_0} (e_{i_{\ell}}e_{j_{\ell}}^*) ,\f{X} \rangle \right|  \|\mathcal{P}_{T_0} (e_{i_{\ell}}e_{j_{\ell}}^*)\|_F \leq  \|\mathcal{P}_{T_0}(e_{i_{\ell}}e_{j_{\ell}}^*)\|_F^2 \|X\|_F
\] 
by Cauchy-Schwartz, and thus the norm bound 
\begin{equation} \label{eq:PTRomegaPT:bound}
\begin{split}
\frac{d_1 d_2}{m} \left\|\widetilde{\y{Z}}_{\ell} \right\|_{S_\infty} &\leq \frac{d_1 d_2}{m}  \|\mathcal{P}_{T_0}(e_{i_{\ell}}e_{j_{\ell}}^*)\|_F^2  \leq  \frac{d_1 d_2}{m} \max_{i \in [d_1], j \in [d_2]}  \| \mathcal{P}_{T_0}(e_{i}e_{j}^*) \|_F^2 \\
&\leq \frac{d_1 d_2}{m} \frac{\mu_0 r(d_1 + d_2)}{d_1 d_2} = \frac{\mu_0 r (d_1+d_2)}{m}
\end{split}
\end{equation}
using the incoherence assumption \eqref{eq:MatrixIRLS:incoherence} in the last inequality. Similarly,
\begin{equation} \label{eq:PT:bound}
\left\|\frac{1}{m} \mathcal{P}_{T_0}\right\|_{S_\infty} = \left\|\frac{1}{m} \mathcal{P}_{T_0} \f{I}  \mathcal{P}_{T_0}\right\|_{S_\infty} \leq \frac{1}{m} \sum_{i=1}^{d_1} \sum_{j=1}^{d_2} \left\| \langle\mathcal{P}_{T_0}(e_i e_{j}^*),(\cdot)\rangle  \mathcal{P}_{T_0}(e_i e_{j}^*) \right\|_{S_\infty} \leq \frac{\mu_0 r (d_1 + d_2)}{m}.
\end{equation}
We note that if operators $\y{A}$ and $\y{B}$ are positive semidefinite, then $\|\y{A}-\y{B}\|_{S_{\infty}} \leq \max(\|\y{A}\|_{S_\infty},\|\y{B}\|_{S_\infty})$, and as both $\widetilde{\y{Z}}_{\ell}$ and $\mathcal{P}_{T_0}$ are positive semidefinite,
\[
\|\y{Z}_{\ell}\|_{S_\infty} \leq \max \left(\frac{d_1 d_2}{m} \left\|\widetilde{\y{Z}}_{\ell} \right\|_{S_\infty}, \frac{1}{m} \left\|\mathcal{P}_{T_0}\right\|_{S_\infty}  \right) = \frac{\mu_0 r (d_1 + d_2)}{m}
\]
for all $\ell \in [m]$. For the expectation of the squares of $\y{Z}_{\ell}$, we obtain
\[
\begin{split}
\Ex \y{Z}_{\ell}\y{Z}_{\ell}^* &= \frac{(d_1 d_2)^2}{m^2} \Ex \left[(\widetilde{\y{Z}}_{\ell})^* \widetilde{\y{Z}}_{\ell}  \right] - \frac{d_1 d_2}{m^2} \Ex\left[\widetilde{\y{Z}}_{\ell}\right] \mathcal{P}_{T_0} - \frac{d_1 d_2}{m^2} \mathcal{P}_{T_0} \Ex\left[\widetilde{\y{Z}}_{\ell}\right] +  \frac{1}{m^2} \mathcal{P}_{T_0} \\
&= \frac{(d_1 d_2)^2}{m^2} \Ex \left[(\widetilde{\y{Z}}_{\ell})^* \widetilde{\y{Z}}_{\ell}  \right] + (1-2) \frac{1}{m^2} \mathcal{P}_{T_0},
\end{split}
\]
as $\mathcal{P}_{T_0}^2 = \mathcal{P}_{T_0}$ and $\Ex[\widetilde{\y{Z}}_{\ell}]=\frac{1}{d_1 d_2} \mathcal{P}_{T_0}$. Thus,
\[
\begin{split}
\left\| \sum_{\ell=1}^m \Ex \y{Z}_{\ell}\y{Z}_{\ell}^* \right\|_{S_\infty} &\leq  \sum_{\ell=1}^m \left\| \Ex \y{Z}_{\ell}\y{Z}_{\ell}^* \right\|_{S_\infty}
= \sum_{\ell=1}^m  \left\|\frac{(d_1 d_2)^2}{m^2} \Ex\left[(\widetilde{\y{Z}}_{\ell})^2\right] - \frac{1}{m^2}\mathcal{P}_{T_0} \right\|_{S_\infty}  \\
&\leq \sum_{\ell=1}^m  \max\left(\frac{(d_1 d_2)^2}{m^2}\left\|\Ex\left[(\widetilde{\y{Z}}_{\ell})^2\right]\right\|_{S_\infty} , \frac{1}{m^2}\left\|\mathcal{P}_{T_0} \right\|_{S_\infty} \right) \\ 
&\leq \sum_{\ell=1}^m  \max\left(\frac{(d_1 d_2)^2}{m^2}\left\|\Ex\left[\|\y{P}_{T_0}(e_{i_{\ell}}e_{j_{\ell}}^*)\|_F^2 \widetilde{\y{Z}}_{\ell}\right]\right\|_{S_\infty} , \frac{1}{m^2} \right) \\
&\leq \sum_{\ell=1}^m \max\left(\frac{(d_1 d_2) (d_1+d_2)\mu_0 r}{m^2}\left\|\Ex  \widetilde{\y{Z}}_{\ell}\right\|_{S_\infty}  , \frac{1}{m^2}\right) \\
&\leq \sum_{\ell=1}^m \max\left(\frac{(d_1+d_2)\mu_0 r}{m^2}, \frac{1}{m^2}\right)  = \frac{\mu_0 r (d_1 + d_2)}{m},
\end{split}
\]
where we used that $\| \mathcal{P}_{T_0} \|_2 \leq 1$ since $\mathcal{P}_{T_0}$ is a projection in the third inequality, the definition of $\mu_0$ in the fourth and the fact that $\Ex \widetilde{\y{Z}}_{\ell} =  \frac{1}{d_1 d_2} \y{P}_{T_0}$ (see \eqref{eq:recht:localRIP:1}) in the fifth.
As the $\y{Z}_{\ell}$ are Hermitian, it follows by the matrix Bernstein inequality (see, e.g., Theorem 5.4.1 of \cite{Ver18}) that
\begin{equation}
\begin{split}
&\bb{P}\left( \left\| \frac{d_1 d_2}{m} \mathcal{P}_{T_0}\mathcal{R}_{\Omega}\mathcal{P}_{T_0}- \mathcal{P}_{T_0} \right\|_{S_\infty} \geq \varepsilon \right)	\leq (d_1 + d_2) \exp\left(- \frac{m \varepsilon^2/2}{\mu_0 r (d_1 + d_2) + \mu_0 r (d_1 + d_2) \epsilon/3}\right) \\
&\leq  (d_1 + d_2) \exp\left(- \frac{m \varepsilon^2}{2 \mu_0 r (d_1 + d_2) + \mu_0 r (d_1 + d_2)/3}\right),
\end{split}
\end{equation}
using that $\varepsilon \leq \frac{1}{2}$ in the last inequality.

Furthermore, if \eqref{eq:localRIP:samplecomplexity:matrixcompletion} is fulfilled, then 
\[
(d_1 + d_2) \exp\left( - \frac{m \varepsilon^2}{\frac{7}{3} \mu_0 r (d_1 + d_2)}\right) \leq (d_1 + d_2)^{-2},
\] 
which shows that \eqref{lem:RIP_tangent} holds with a probability of at least $1-(d_1 + d_2)^{-2}$.
\end{proof}

To prove our theorem, we will use the local restricted isometry statement of \eqref{lem:RIP_tangent} for tangent spaces $T_{\f{X}}$ corresponding to matrices $\f{X} \in \Rdd$ \emph{that are close} to $\f{X}^0$. We show the following auxiliary result, which is a refinement of Lemma 4.2 \cite{wei_cai_chan_leung} as we obtain a bound in the $S_\infty$-norm in (c) instead of in the Frobenius norm.
 \begin{lemma} \label{eq:MatrixIRLS:tangentspace:localRIP:perturbation}
 Let $\f{X}^0, \f{X} \in \Rdd$ be matrices and assume that $0 < \varepsilon < 1$ and that the following three conditions hold: 
 \begin{itemize}
     \item[(a)] For $\y{R}_{\Omega}: \Rdd \to \Rdd$ as in \eqref{eq:ROmega:def},
     \[ \left\|\y{R}_{\Omega}\right\|_{S_\infty} \leq \frac{16}{3}\log(D). \]
      \label{lemma:wei:conditiona}
     \item[(b)]\label{lemma:wei:conditionb} The tangent space $T_0=T_{\f{X}^0}$ onto the rank-$r$ manifold $\mathcal{M}_{r}$ at $\f{X}^0$ fulfills 
     \[
	\left\|\frac{d_1 d_2}{m}\mathcal{P}_{T_0}\y{R}_{\Omega}\mathcal{P}_{T_0}-\mathcal{P}_{T_0}\right\|_{S_{\infty}}
	\leq \varepsilon.
     \]
     \item[(c)] The spectral norm distance between $\f{X}$ and $\f{X}^0$ fulfills
     \[
     \|\f{X}-\f{X}^0\|_{S_\infty} \leq \frac{\sqrt{3}}{32 \sqrt{\log(D)}\sqrt{(1+\varepsilon)}} \varepsilon \sqrt{\frac{m}{d_1 d_2}} \sigma_r(\f{X}^0).
     \]
 \end{itemize} Then the tangent space $T=T_{\f{X}}$ onto the rank-$r$ manifold at $\f{X}$ fulfills 
\begin{equation} \label{eq:lemma:wei:statement}
\left\| \frac{d_1 d_2}{m} \y{P}_{T}\y{R}_{\Omega}\y{P}_{T}- \y{P}_{T}\right\|_{S_\infty}
\leq 4\varepsilon.
\end{equation}
\end{lemma}
\begin{proof}
	For any $\f{Z} \in \Rdd$, we have 
\begin{align*}
\left\|\y{R}_{\Omega}\y{P}_{T_0}(\f{Z})\right\|_F^2 &= \left\langle \y{R}_{\Omega}\y{P}_T(\f{Z}), \y{R}_{\Omega}\y{P}_T(\f{Z})\right\rangle \leq{\frac{16}{3}\log(D)\left\langle\y{P}_{T_0}(\f{Z}),\y{R}_{\Omega}\y{P}_{T_0}(\f{Z})\right\rangle}\\
&=\frac{16}{3}\log(D)\left\langle\y{P}_{T_0}(\f{Z}),\y{P}_{T_0}\y{R}_{\Omega}\y{P}_{T_0}(\f{Z})\right\rangle  \\
&= \frac{16}{3}\log(D)\left(\left\langle\y{P}_{T_0}(\f{Z}),\frac{m}{d_1 d_2}\y{P}_{T_0}(\f{Z})\right\rangle + \left\langle \y{P}_{T_0}(\f{Z}),\left( \y{P}_{T_0}\y{R}_{\Omega}\y{P}_{T_0}(\f{Z}) - \frac{m}{d_1 d_2}\y{P}_{T_0}(\f{Z})  \right)   \right\rangle \right) \\
&\leq \frac{16}{3}\log(D) \left(\frac{m}{d_1 d_2}+ \varepsilon \frac{m}{d_1 d_2}\right) \left\| \y{P}_{T_0}(\f{Z})\right\|_F^2 \leq \frac{16}{3}\log(D) (1+\varepsilon)\frac{m}{d_1 d_2}\left\|\f{Z}\right\|_F^2,
\end{align*}
where the first inequality follows from condition (a) and the second one from condition (b). It follows that 
\begin{equation} \label{eq:lemma:wei:1}
\left\|\y{R}_{\Omega}\y{P}_{T_0}\right\|\leq \sqrt{\frac{16}{3}\log(D) (1+\varepsilon)\frac{m}{d_1 d_2}}.
\end{equation}
Furthermore, if $\f{U},\f{U}_0 \in \R^{d_1 \times r}$ and $\f{V},\f{V}_0 \in \R^{d_2 \times r}$ are the matrices of first $r$ left and right singular vectors of $\f{X}$ and $\f{X}^0$, respectively, it holds that for any $\f{Z} \in \Rdd$,
\[
\begin{split}
(\y{P}_{T} - \y{P}_{T_0})(\f{Z}) &= \f{U}\f{U}^* \f{Z} + \f{Z}\f{V}\f{V}^* - \f{U}\f{U}^* \f{Z} \f{V}\f{V}^* - \f{U}_0\f{U}_0^* \f{Z} - \f{Z}\f{V}_0\f{V}_0^* + \f{U}_0\f{U}_0^* \f{Z} \f{V}_0\f{V}_0^*  \\
&= \left( \f{U}\f{U}^* -  \f{U}_0\f{U}_0^*\right) \f{Z} (\f{I} - \f{V}_0\f{V}_0^*) + ( \f{I}-\f{U}\f{U}^*) \f{Z} ( \f{V}\f{V}^*- \f{V}_0\f{V}_0^*),
\end{split}
\]
which we use to estimate
\begin{equation*}
\begin{split}
	\|(\y{P}_{T} - \y{P}_{T_0})(\f{Z})\|_{F} &\leq \|\f{U}\f{U}^* -  \f{U}_0\f{U}_0^*\|_{S_\infty} \|\f{Z}\|_F \|\f{I} - \f{V}_0\f{V}_0^*\|_{S_\infty} + \|\f{I}-\f{U}\f{U}^*\|_{S_\infty} \|\f{Z}\|_{F} \|\f{V}\f{V}^*- \f{V}_0\f{V}_0^*\|_{S_\infty} \\
	&\leq \frac{\| \y{T}_r (\f{X}) - \f{X}^0\|_{S_\infty}}{\sigma_r(\f{X}^0)} \|\f{Z}\|_F \cdot 1 + 1 \cdot \|\f{Z}\|_F  \frac{\| \y{T}_r (\f{X}) - \f{X}^0\|_{S_\infty}}{\sigma_r(\f{X}^0)} \\
	&\leq 2 \frac{\|\y{T}_r (\f{X})- \f{X}\|_{S_\infty} + \| \f{X} - \f{X}^0\|_{S_\infty}}{\sigma_r(\f{X}^0)} \|\f{Z}\|_F,
\end{split}
\end{equation*}
where $ \y{T}_r (\f{X})$ is the best rank-$r$ approximation \eqref{eq:best:rankr:approximation}. Here, we used the results
\[
\|\f{U}\f{U}^* -  \f{U}_0\f{U}_0^*\|_{S_\infty} \leq \frac{\| \y{T}_r (\f{X}) - \f{X}^0\|_{S_\infty}}{\sigma_r(\f{X}^0)}
\]
and
\[
\|\f{V}\f{V}^* -  \f{V}_0\f{V}_0^*\|_{S_\infty} \leq \frac{\| \y{T}_r (\f{X}) - \f{X}^0\|_{S_\infty}}{\sigma_r(\f{X}^0)}
\]
of Lemma 4.2, ineq. (4.3) of \cite{WeiCCL16MatrixRecovery}, which bound the distance between the projections onto the left and right singular subspaces of $\f{X}$ and $\f{X}^0$.

From the Eckardt-Young-Mirsky theorem \eqref{eq:best:rankr:approximation}, it then follows that
\begin{equation} \label{eq:lemma:wei:3}
\|(\y{P}_{T} - \y{P}_{T_0})\|_{S_\infty} \leq \frac{4 \| \f{X} - \f{X}^0\|_{S_\infty}}{\sigma_r(\f{X}^0)}.
\end{equation}

With this, we further bound
\begin{equation} \label{eq:lemma:wei:4}
\begin{split}
\left\|\y{R}_{\Omega}\y{P}_{T}\right\|_{S_\infty}&\leq \left\| \y{R}_\Omega(\y{P}_{T}-\y{P}_{T_0})\right\|_{S_\infty}+\left\|\y{R}_{\Omega}\y{P}_{T_0}\right\|_{S_\infty}\\
&\leq \frac{16}{3} \log(D) \frac{4\left\| \f{X}-\f{X}^0\right\|_{S_\infty}}{\sigma_r(\f{X}^0)}+\left\|\y{R}_{\Omega}\y{P}_{T_0}\right\|_{S_\infty}\\
&\leq \frac{16}{3} \log(D) \frac{\sqrt{3}}{8 \sqrt{\log(D)}\sqrt{(1+\varepsilon)}} \varepsilon \sqrt{\frac{m}{d_1 d_2}} +\sqrt{\frac{16}{3}\log(D) (1+\varepsilon)\frac{m}{d_1 d_2}}\\
&= \frac{2}{\sqrt{3}} \sqrt{\log(D)} \frac{1}{\sqrt{(1+\varepsilon)}} \varepsilon \sqrt{\frac{m}{d_1 d_2}} +\sqrt{\frac{16}{3}\log(D) (1+\varepsilon)\frac{m}{d_1 d_2}}\\
&\leq 2\sqrt{3}  \sqrt{\log(D)}\sqrt{1+\varepsilon} \sqrt{\frac{m}{d_1 d_2}},
\end{split}
\end{equation}
where the second inequality follows from \eqref{eq:lemma:wei:3} and the third from condition (c). To prove the statement \eqref{eq:lemma:wei:statement}, we calculate
\begin{equation*}
\begin{split}
\left\|\frac{d_1 d_2}{m}\y{P}_{T}\y{R}_\Omega\y{P}_{T} - \y{P}_{T}\right\|_{S_\infty} &\leq\left\|\y{P}_{T}-\y{P}_{T_0}\right\|_{S_\infty}+\frac{d_1 d_2}{m}\left\|\y{P}_{T}\y{R}_{\Omega}\y{P}_{T}-\y{P}_{T}\y{R}_{\Omega}\y{P}_{T_0}\right\|_{S_\infty}\\
&\quad+\frac{d_1 d_2}{m}\left\| \y{P}_{T}\y{R}_{\Omega}\y{P}_{T_0}-\y{P}_{T_0}\y{R}_{\Omega}\y{P}_{T_0}\right\|_{S_\infty}+\left\| \y{P}_{T_0}-\frac{d_1 d_2}{m}\y{P}_{T_0}\y{R}_{\Omega}\y{P}_{T_0}\right\|_{S_\infty}\\
&\leq\left\|\y{P}_{T}-\y{P}_{T_0}\right\|_{S_\infty}+\frac{d_1 d_2}{m}\left\|\y{R}_{\Omega}\y{P}_{T}\right\|_{S_\infty}\left\|\y{P}_{T}-\y{P}_{T_0}\right\|_{S_\infty}\\
&\quad+\frac{d_1 d_2}{m}\left\|\y{R}_{\Omega}\y{P}_{T_0}\right\|_{S_\infty}\left\|\y{P}_{T}-\y{P}_{T_0}\right\|_{S_\infty}+\left\| \y{P}_{T_0}-\frac{d_1 d_2}{m}\y{P}_{T_0}\y{R}_{\Omega}\y{P}_{T_0}\right\|_{S_\infty}\\
&\leq \frac{4\left\| \f{X}-\f{X}^0\right\|_{S_\infty}}{\sigma_r(\f{X}^0)}+\frac{d_1 d_2}{m}\left\|\y{R}_{\Omega}\y{P}_{T}\right\|_{S_\infty}\frac{4\left\| \f{X}-\f{X}^0\right\|_{S_\infty}}{\sigma_r(\f{X}^0)}\\
&\quad +\frac{d_1 d_2}{m}\left\|\y{R}_{\Omega}\y{P}_{T_0}\right\|_{S_\infty}\frac{4\left\| \f{X}-\f{X}^0\right\|_{S_\infty}}{\sigma_r(\f{X}^0)}+\left\| \y{P}_{T_0}-\frac{d_1 d_2}{m}\y{P}_{T_0}\y{R}_{\Omega}\y{P}_{T_0}\right\|_{S_\infty}\\
&\leq 4\varepsilon
\end{split}
\end{equation*}
where in the second inequality, we utilized the fact ${\y{R}_{\Omega}^*=\y{R}_\Omega}$ so that $\left\|\y{P}_{T}\y{R}_\Omega\right\|_{S_\infty}=\left\|\y{R}_\Omega\y{P}_{T}\right\|_{S_\infty}$. The very last estimate follows from conditions (b) and (c) and the bounds \eqref{eq:lemma:wei:1} and \eqref{eq:lemma:wei:4} for $\left\|\y{R}_\Omega\y{P}_{T}\right\|_{S_\infty}$ and $\left\|\y{R}_\Omega\y{P}_{T_0}\right\|_{S_\infty}$.
\end{proof}

In the following lemma, we combine the previous results to show that under our sampling model, with high probability, a local restricted isometry property holds with respect to tangent spaces $T_k$ that are in some sense close to $\f{X}^0$.

\begin{lemma} \label{lemma:MatrixIRLS:localRIP}
Let $\f{X}^0 \in \Rdd$ be a matrix of rank $r$ that is $\mu_0$-incoherent, and let $\Omega = (i_{\ell},j_{\ell})_{\ell=1}^m$ be a random index set of cardinality $|\Omega|=m$ that is sampled uniformly without replacement, or, alternatively, sampled independently with replacement. There exists constants $C, \widetilde{C},C_1$ such that if 
\begin{equation} \label{eq:MatrixIRLS:matrixcompletion:samplecomplexity}
m \geq C \mu_0 r (d_1 + d_2) \log(d_1 + d_2),
\end{equation}
then, with probability at least $1 - 2D^{-2}$, the following holds: For each matrix $\f{X}^{(k)}\in \Rdd$ fulfilling
\begin{equation} \label{eq:MatrixIRLS:closeness:assumption}
\|\f{X}\hk - \f{X}^0\|_{S_{\infty}} \leq C_1 \sqrt{\frac{\mu_0 r}{d}}\sigma_r(\f{X}^0),
\end{equation}
it follows that the projection $\y{P}_{T_k}:\Rdd \to \Rdd$ onto the tangent space $T_k := T_{\mathcal{T}_{r}(\f{X}\hk)}\mathcal{M}_{r}$ satisfies
\begin{equation*}
\left\| \frac{d_1 d_2}{m} \y{P}_{T_k}P_{\Omega}^* P_{\Omega}\y{P}_{T_k}- \y{P}_{T_k}\right\|_{S_\infty}
	\leq \frac{2}{5},
\end{equation*}
and furthermore,
\begin{equation*}
	\|\f{\eta}\|_{F} \leq  \sqrt{\frac{\widetilde{C} d \log(D)}{\mu_0 r}} \|\y{P}_{T_k^{\perp}}(\f{\eta})\|_{F}
\end{equation*}
for each matrix $\eta \in \ker P_{\Omega}$ in the null space of the subsampling operator $P_{\Omega}:\Rdd \to \R^m$.
\end{lemma}
\begin{proof}[{Proof of \Cref{lemma:MatrixIRLS:localRIP}}]
	Assume that there are $m$ locations $\Omega = (i_{\ell},j_{\ell})_{\ell=1}^m$ in $[d_1] \times [d_2]$ sampled independently uniformly \emph{with replacement}, where $m$ fulfills \eqref{eq:MatrixIRLS:matrixcompletion:samplecomplexity} with $C:= 7/\varepsilon^2$ and $\varepsilon= 0.1$. By \Cref{lem:repetition}, it follows that the corresponding operator $\y{R}_{\Omega}: \Rdd \to \Rdd$ from \eqref{eq:ROmega:def} fulfills 
	\begin{equation} \label{eq:MatrixIRLS:localconvergence:matrixcompletion:05}
	\|\y{R}_{\Omega}\|_{S_\infty} \leq \frac{16}{3} \log(D)
	\end{equation}
	on an event called $E_{\Omega}$, which occurs with a probability of at least $1- D^{-2}$, and by \Cref{lem:RIPPS}, the tangent space $T_0=T_{\f{X}^0} \mathcal{M}_{r}$ corresponding to the $\mu_0$-incoherent rank-$r$ matrix $\f{X}^0$ fulfills 
	\[
	\left\|\frac{d_1 d_2}{m}\mathcal{P}_{T_0} P_{\Omega}^* P_{\Omega}  \mathcal{P}_{T_0}-\mathcal{P}_{T_0}\right\|_{S_{\infty}}
	\leq \varepsilon 
	\]
	on an event called $E_{\Omega,T_0}$, which occurs with a probability of at least $1- D^{-2}$.
	Let $\tilde{\epsilon} = \frac{1}{10}$. If $\f{X}\hk \in \Rdd$ is such that $\|\f{X}\hk - \f{X}^0\|_{S_\infty} \leq \widetilde{\xi} \sigma_r(\f{X}^0)$ with
	\begin{equation} \label{eq:MatrixIRLS:localconvergence:matrixcompletion:066}
	\widetilde{\xi} = \frac{\sqrt{3}}{32} \frac{\epsilon}{\sqrt{\log(D)(1+\epsilon)}} \sqrt{\frac{m}{d_1 d_2}} = \frac{\sqrt{3}}{32} \frac{1}{10 \sqrt{\log(D)(11/10)}} \sqrt{\frac{m}{d_1 d_2}},
	\end{equation}
	it follows by \Cref{eq:MatrixIRLS:tangentspace:localRIP:perturbation} that on the event $E_{\Omega} \cap E_{\Omega,T_0}$, the tangent space $T_k:=\f{X}\hk$ onto the rank-$r$ manifold at $\f{X}\hk$ fulfills
	\begin{equation} \label{eq:MatrixIRLS:localconvergence:matrixcompletion:075}
	\left\| \frac{d_1 d_2}{m} \y{P}_{T_k}\y{R}_{\Omega}\y{P}_{T_k}- \y{P}_{T_k}\right\|_{S_\infty}
	\leq 4 \tilde{\epsilon} = \frac{2}{5}.
	\end{equation}
	Next, we claim that on the event $E_{\Omega} \cap E_{\Omega,T_0}$,
	\begin{equation} \label{eq:MatrixIRLS:localconvergence:matrixcompletion:1}
		\|\f{\eta}\|_{F} \leq \sqrt{\frac{\widetilde{C} d \log(D)}{\mu_0 r}} \|\y{P}_{T_k^{\perp}}(\f{\eta})\|_{F}.
	\end{equation}
	for any for each matrix $\eta \in \ker P_{\Omega}$ in the null space of the subsampling operator $P_{\Omega}:\Rdd \to \R^m$.
	
	Indeed, to show this claim, we first note that $\f{\eta} \in \ker P_{\Omega}$ if and only if $\f{\eta} \in \ker \y{R}_{\Omega}: P_{\Omega}^* P_{\Omega}$.
	Let $\f{\eta} \in \ker \y{R}_{\Omega}$. Then 
	\[
	\begin{split}
	\|\y{P}_{T_k}(\f{\eta})\|_F^2 &= \langle \y{P}_{T_k}(\f{\eta}), \y{P}_{T_k}(\f{\eta}) \rangle \\
	&= \left\langle \y{P}_{T_k}(\f{\eta}),\frac{d_1 d_2}{m}  \y{P}_{T_k}\y{R}_{\Omega}\y{P}_{T_k}(\f{\eta}) \right\rangle + 
	\left\langle \y{P}_{T_k}(\f{\eta}),\y{P}_{T_k}(\f{\eta}) - \frac{d_1 d_2}{m} \y{P}_{T_k}\y{R}_{\Omega}\y{P}_{T_k}(\f{\eta})\right\rangle \\
	&\leq \left\langle \y{P}_{T_k}(\f{\eta}),\frac{d_1 d_2}{m}  \y{P}_{T_k}\y{R}_{\Omega}\y{P}_{T_k}(\f{\eta}) \right\rangle + \|  \y{P}_{T_k}(\f{\eta})\|_F \left\| \y{P}_{T_k}- \frac{d_1 d_2}{m} \y{P}_{T_k}\y{R}_{\Omega}\y{P}_{T_k}\right\|_{S_\infty} \|\y{P}_{T_k}(\f{\eta})\|_F \\
	&\leq \left\langle \y{P}_{T_k}(\f{\eta}),\frac{d_1 d_2}{m}  \y{P}_{T_k}\y{R}_{\Omega}\y{P}_{T_k}(\f{\eta}) \right\rangle + 4 \epsilon \|\y{P}_{T_k}(\f{\eta})\|_F^2,
	\end{split}
	\]
	using \eqref{eq:MatrixIRLS:localconvergence:matrixcompletion:075} in the last inequality, which implies that
	\[
	\begin{split}
	\|\y{P}_{T_k}(\f{\eta})\|_F^2 &\leq \frac{1}{1-4\epsilon} \frac{d_1 d_2}{m} \langle \y{P}_{T_k}(\f{\eta}), \y{P}_{T_k}\y{R}_{\Omega}^2\y{P}_{T_k}(\f{\eta}) \rangle =  \frac{1}{1-4\epsilon} \frac{d_1 d_2}{m} \|\y{R}_{\Omega} \y{P}_{T_k}(\f{\eta})\|_F^2 \\
	&\leq \frac{2 d_1 d_2}{m} \|\y{R}_{\Omega} \y{P}_{T_k}(\f{\eta})\|_F^2
	\end{split}
	\]
	using the fact that $\y{R}_{\Omega}: \Rdd \to \Rdd$ is positive semidefinite and has eigenvalues that are $0$ or larger or equal than $1$ only. Furthermore, we used that $\epsilon \leq \frac{1}{10}$ in the last inequality.
	
	Since  $\f{\eta} \in \ker \y{R}_{\Omega}$, it holds that 
	\[
	0 = \|\y{R}_{\Omega}(\f{\eta})\|_F = \left\|\y{R}_{\Omega}\left(\y{P}_{T_k}(\f{\eta}) +  \y{P}_{T_k^{\perp}}(\f{\eta})\right)\right\|_F \geq \|\y{R}_{\Omega}\y{P}_{T_k}(\f{\eta})\|_F - \|\y{R}_{\Omega}\y{P}_{T_k^{\perp}}(\f{\eta})\|_F
	\]
	so that
	\[
	\|\y{R}_{\Omega}\y{P}_{T_k}(\f{\eta})\|_F \leq \|\y{R}_{\Omega}\y{P}_{T_k^{\perp}}(\f{\eta})\|_F \leq \frac{16}{3} \log(D) \|\y{P}_{T_k^{\perp}}(\f{\eta})\|_F,
	\]
	where we used \eqref{eq:MatrixIRLS:localconvergence:matrixcompletion:05} in the last inequality. Inserting this above, we obtain
	\[
	\begin{split}
	\|\f{\eta}\|_F^2 &= \|\y{P}_{T_k}(\f{\eta})\|_F^2 + \|\y{P}_{T_k^{\perp}}(\f{\eta})\|_F^2 \leq \left(\frac{2 d_1 d_2}{m} \frac{16^2}{3^2} \log(D)^2  + 1 \right) \|\y{P}_{T_k^{\perp}}(\f{\eta})\|_F^2 \\
	&\leq \left( \frac{2 d_1 d_2}{C \mu_0 r (d_1 + d_2) \log(d_1+ d_2)} \frac{16^2}{3^2} \log(D)^2  + 1 \right) \|\y{P}_{T_k^{\perp}}(\f{\eta})\|_F^2 \\
	&\leq  \frac{\widetilde{C} d \log(D)}{\mu_0 r}  \|\y{P}_{T_k^{\perp}}(\f{\eta})\|_F^2,
	\end{split}
	\]
	where we used the sample complexity condition \eqref{eq:MatrixIRLS:matrixcompletion:samplecomplexity} in the second inequality and the definition 
	\[
	\widetilde{C}:= \frac{4 \cdot 16^2}{C \cdot 3^2}
	\]
	for the constant $\widetilde{C}$.
	
	Moreover, we observe that for $C_1: = \frac{\sqrt{C}}{320}\sqrt{\frac{30}{11}}$ where $C$ is the constant of \eqref{eq:MatrixIRLS:matrixcompletion:samplecomplexity}, it holds that 
	\[
	C_1 \sqrt{\frac{\mu_0 r}{d}} 	\leq \frac{\sqrt{3}}{32} \frac{1}{10 \sqrt{\log(D)(11/10)}} \sqrt{\frac{C \mu_0 r (d_1 + d_2)\log(d_1 + d_2)}{d_1 d_2}} \leq \widetilde{\xi},
	\]
	implying that the two statements of \Cref{lemma:MatrixIRLS:localRIP} are satisfied on the event $E_{\Omega} \cap E_{\Omega,T_0}$ if \eqref{eq:MatrixIRLS:closeness:assumption} holds. By the above mentioned probability bounds and a union bound, $E_{\Omega} \cap E_{\Omega,T_0}$ occurs with a probability of at least $1 - 2D^{-2}$, finishing the proof for the sampling with replacement model. By the argument of Proposition 3 of \cite{recht}, the result extends to the model of sampling locations drawn uniformly at random without replacement, with the same probability bound. This concludes the proof of \Cref{lemma:MatrixIRLS:localRIP}.
\end{proof}

The following lemma will also play a role in the proof of \Cref{cor:MatrixIRLS:localconvergence:matrixcompletion}.

\begin{lemma}\label{lemma:etaksigmarp1Xk}
Let $C, \widetilde{C},C_1$ be the constants of \Cref{lemma:MatrixIRLS:localRIP} and $\mu_0$ be the incoherence factor of a rank-$r$ matrix $\f{X}^0$. If 
\[
m \geq C \mu_0 r (d_1 + d_2) \log(d_1 + d_2)
\]
and if $\eta\hk = \f{X}\hk - \f{X}^0$ fulfills
\[
\|\eta\hk\|_{S_\infty} \leq \xi  \sigma_r(\f{X}^0),
\]
with
\[
\xi:= \min\left( C_1 \sqrt{\frac{\mu_0 r}{d}}, \frac{\mu_0}{4(1+6 \kappa) d \log(D) \widetilde{C}} \right)
\]
then, on the event of  \Cref{lemma:MatrixIRLS:localRIP}, it holds that
\begin{equation}
\|\f{\eta}\hk\|_{S_\infty} < \sqrt{\frac{4 \widetilde{C} d (d-r) \log(D)}{\mu_0 r}}  \sigma_{r+1}(\f{X}\hk).
\end{equation}
\end{lemma}
\begin{proof}
First, we compute that 
\[
\begin{split}
\|\y{P}_{T_k^{\perp}}(\f{\eta}\hk)\|_{F} &\leq  \|\y{P}_{T_k^{\perp}}(\f{X}\hk)\|_{F} +   \|\y{P}_{T_k^{\perp}}(\f{X}^{0})\|_{F} \leq \sqrt{\sum_{i=r+1}^d \sigma_{i}^2(\f{X}\hk)} +  \left\| \f{U}_{\perp}^{(k)}\f{U}_{\perp}^{(k)*} \f{X}^0 \f{V}_{\perp}^{(k)}\f{V}_{\perp}^{(k)*}\right\|_{F} \\
&\leq  \sqrt{d-r} \sigma_{r+1}(\f{X}\hk) + \|\f{U}_{\perp}^{(k)*} \f{U}_0\|_{S_\infty} \|\f{\Sigma}_0\|_F \|\f{V}_0^{*} \f{V}_{\perp}^{(k)}\|_{S_\infty} \\
&\leq  \sqrt{d-r} \sigma_{r+1}(\f{X}\hk) +  \frac{2 \|\f{\eta}\hk\|_{S_\infty}^2}{(1-\zeta)^2\sigma_r^2(\f{X}^0)} \sqrt{r} \sigma_1(\f{X}^0) \\
&=  \sqrt{d-r} \sigma_{r+1}(\f{X}\hk) +  \frac{2 \|\f{\eta}\hk\|_{S_\infty}^2}{(1-\zeta)^2\sigma_r(\f{X}^0)} \sqrt{r} \kappa,
\end{split}
\]
where $0 < \zeta < 1$ such that $\|\f{X}\hk - \f{X}^0\|_{S_{\infty}} \leq \zeta \sigma_r(\f{X}^0)$, using \Cref{lemma:Wedinbound} twice in the fourth inequality and $\|\f{A} \f{B}\|_F \leq \|\f{A}\|_{S_\infty} \|\f{B}\|_F$ all matrices $\f{A}$ and $\f{B}$, referring to the notations of \Cref{lemma:MatrixIRLS:localconvergence:2} (see below) for $\f{U}_0,\f{\Sigma}_0,\f{V}_0,\f{U}_{\perp}\hk$ and $\f{V}_{\perp}\hk$.

Using \Cref{lemma:MatrixIRLS:localRIP} for $\f{\eta}\hk = \f{X}\hk - \f{X}^{0}$, we obtain on the event on which the statement of \Cref{lemma:MatrixIRLS:localRIP} holds that
\[
\begin{split}
\|\eta\hk\|_{S_\infty} \leq \|\eta\hk\|_{F} &\leq  \sqrt{\frac{\widetilde{C} d \log(D)}{\mu_0 r}} \|\y{P}_{T_k^{\perp}}(\f{\eta}\hk)\|_{F} \\
&\leq  \sqrt{\frac{\widetilde{C} d \log(D)}{\mu_0 r}}     \left(   \sqrt{d-r} \sigma_{r+1}(\f{X}\hk) +  \frac{8  \sqrt{r} \kappa  \|\f{\eta}\hk\|_{S_\infty}^2}{\sigma_r(\f{X}^0)} \right) \\
&\leq  \sqrt{\frac{\widetilde{C} d \log(D)}{\mu_0 r}}     \left(   \sqrt{d-r} \sigma_{r+1}(\f{X}\hk) +   \frac{8  \sqrt{r} \kappa \mu_0 \sigma_{r}(\f{X}^0)}{4 (1+6 \kappa) d \log(D) \widetilde{C} \sigma_r(\f{X}^0)}   \|\f{\eta}\hk\|_{S_\infty} \right) \\
&= \sqrt{\frac{\widetilde{C} d (d-r) \log(D)}{\mu_0 r}} \sigma_{r+1}(\f{X}\hk) + \frac{1}{3} \sqrt{\frac{\mu_0}{\widetilde{C} d \log(D)}} \|\eta\hk\|_{S_\infty}.
\end{split}
\]
Since $\mu_0 \leq \frac{d}{r}$, we have that $\frac{1}{3} \sqrt{\frac{\mu_0}{\widetilde{C} d \log(D)}} < \frac{1}{2}$, and therefore we obtain, after rearranging,
\[
\left(1-\frac{1}{2}\right) \|\eta\hk\|_{S_\infty} < \left(1- \frac{1}{3} \sqrt{\frac{\mu_0}{\widetilde{C} d \log(D)}}  \right) \|\eta\hk\|_{S_\infty} \leq \sqrt{\frac{\widetilde{C} d (d-r) \log(D)}{\mu_0 r}} \sigma_{r+1}(\f{X}\hk),
\]
which implies the statement of this lemma.
\end{proof}

\subsection{Weight operator and matrix perturbation}
In the following, we use a well-known bound on perturbations of the singular value decomposition, which is originally due to \cite{Wedin72}. The result bounds the alignment of the subspaces spanned by the singular vectors of two matrices by their norm distance, given a gap between the first singular values of one matrix and the last singular values of the other matrix that is sufficiently pronounced. 
\begin{lemma}[Wedin's bound \cite{Stewart06}] \label{lemma:Wedinbound} 
Let $\f{X}$ and $\widehat{\f{X}}$ be two matrices of the same size and their singular value decompositions 
\begin{align*}
\f{X}=\begin{pmatrix}\f{U}&\f{U}_{\perp}\end{pmatrix}\begin{pmatrix}\f{\Sigma} &0\\0& \f{\Sigma}_{\perp} \end{pmatrix}\begin{pmatrix}\f{V}^*\\ \f{V}_{\perp}^*\end{pmatrix}\quad\text{ and } \quad \widehat{\f{X}}=\begin{pmatrix}\widehat{\f{U}}&\widehat{\f{U}}_{\perp}\end{pmatrix}\begin{pmatrix}\widehat{\f{\Sigma}} &0\\0& \widehat{\f{\Sigma}}_{\perp} \end{pmatrix}\begin{pmatrix}\widehat{\f{V}}^*\\\widehat{\f{V}}_{\perp}^*\end{pmatrix},
\end{align*}
where the submatrices have the sizes of corresponding dimensions. Suppose that $\delta,\alpha$ satisfying $0< \delta\leq \alpha$ are such that $\alpha\leq \sigma_{\min}(\Sigma)$ and $\sigma_{\max}(\widehat{\Sigma}_{\perp})<\alpha-\delta$. Then 
\begin{equation}\label{Wedin}
\|\widehat{\f{U}}_{\perp}^*\f{U}\|_{S_\infty} \leq \sqrt{2} \frac{\|\f{X}-\widehat{\f{X}}\|_{S_\infty}}{\delta} \text{  and  }\|\widehat{\f{V}}_{\perp}^*\f{V}\|_{S_\infty} \leq \sqrt{2}\frac{\|\f{X}-\widehat{\f{X}}\|_{S_\infty}}{\delta}.
\end{equation}
\end{lemma}

We also use a lemma which provides an explicit formula for the calculation of the new iterate $\f{X}\hk$ of \texttt{MatrixIRLS} and its characterization by optimality conditions. It is well-known in the IRLS literature, see, e.g., Eq. (1.9) and Lemma 5.2 of \cite{Daubechies10} or Lemma 5.1 \cite{Fornasier11}, and is very general as it holds for any positive definite weight operator. 
\begin{lemma}\label{lemma:MatrixIRLS:Xk:optimality} 
Let $P_{\Omega}: \Rdd \to \R^m$ be the sampling operator, let $\f{y} \in \R^m$. Let $W^{(k)}:\Rdd \to \Rdd$ be the weight operator of \Cref{def:optimalweightoperator} defined based on $\f{X}\hk \in \Rdd$. Then the solution of the weighted least squares step \eqref{eq:MatrixIRLS:Xdef} of \Cref{algo:MatrixIRLS} is unique and
\begin{equation} \label{eq:MatrixIRLS:Xk:optimality:explicitformula}
\f{X}^{(k+1)} =\argmin\limits_{P_{\Omega}(\f{X})=\f{y}} \langle \f{X}, W^{(k)}(\f{X}) \rangle =  (W\hk)^{-1}P_{\Omega}^*\left(P_{\Omega}(W\hk)^{-1} P_{\Omega}^*\right)^{-1}(\f{y}),
\end{equation}
where $(W\hk)^{-1}: \Rdd \to \Rdd$ is the inverse matrix operator of $W^{(k)}$.  

Moreover,
a matrix $\f{X}^{(k+1)} \in \Rdd$ coincides with the one of \eqref{eq:MatrixIRLS:Xk:optimality:explicitformula} if and only if
\begin{equation} \label{eq_proof_JminX_1}
\langle W^{(k)}(\f{X}\hkk),\f{\eta} \rangle = 0 \; \text{ for all } \; \f{\eta} \in \ker P_{\Omega}\;\;\text{ and } \;\; P_{\Omega}(\f{X}\hkk) = \f{y}. 
\end{equation}
\end{lemma}

We show the following lemma. Wherever it appears, $\|\f{X}\|_{S_1}$ denotes the nuclear norm $\|\f{X}\|_{S_1} = \sum_{i=1}^d \sigma_i (\f{X})$ of a matrix $\f{X} \in \Rdd$.
\begin{lemma} \label{lemma:MatrixIRLS:localconvergence:1}
Let $\f{X}^0 \in \Rdd$ be a matrix of rank $r$, let $\f{X}\hk$ be the $k$-th iterate of \Cref{algo:MatrixIRLS} for input parameters $\Omega$, $\f{y}=P_{\Omega}(\f{X}^0)$ and $\widetilde{r} =r $. Assume that $\epsilon_k=\sigma_{r+1}(\f{X}\hk)$ and that
\begin{equation} \label{eq:localcvg:ass:1}
\|\eta \|_F \leq c(\mu_0,r,d_1,d_2) \|\y{P}_{T_k^{\perp}} \eta \|_F \quad \quad \text{ for all } \eta \in \ker P_{\Omega}
\end{equation}
for some constant $c(\mu_0,r,d_1,d_2)$ that may depend on $\mu_0,r, d_1,d_2$, where $T_k = T_{\mathcal{T}_{r}(\f{X}\hk)}\mathcal{M}_{r}$  is tangent space onto the manifold of rank-$r$ matrices at $\mathcal{T}_{r}(\f{X}\hk)$. Then
\begin{equation} \label{eq:MatrixIRLS:localconvergence:1}
	\|\f{X}\hkk - \f{X}^0\|_{S_\infty} \leq c(\mu_0,r,d_1,d_2)^2 \epsilon_k^{2} \|W\hk(\f{X}^0)\|_{S_1},
\end{equation}
if $W\hk : \Rdd \to \Rdd$ is the optimal weight operator of \Cref{def:optimalweightoperator} corresponding to $\f{X}\hk$.
\end{lemma}
\begin{proof}[Proof of \Cref{lemma:MatrixIRLS:localconvergence:1}]
Let $\f{\eta}\hkk:= \f{X}\hkk -\f{X}^0$.
Since $\f{\eta}\hkk$ is in the nullspace $\ker P_{\Omega}$, it follows from \eqref{eq:localcvg:ass:1} that
\begin{equation} \label{eq:MatrixIRLS:localconvergence:1p5}
\|\f{\eta}\hkk\|_{S_\infty}^2 \leq \|\f{\eta}\hkk\|_{F}^2 \leq c(\mu_0,r,d_1,d_2)^2  \|\y{P}_{T_k^\perp}(\f{\eta}\hkk)\|_{F}^2.
\end{equation}

Recalling the definition of the weight operator $W^{(k)}:\Rdd \to \Rdd$ from \Cref{def:optimalweightoperator} we see that, if 
\begin{equation} \label{eq:MatrixIRLS:localconvergence:XkSVD}
\f{X}\hk= \f{U}_k \f{\Sigma}_k \f{V}_k^* = 
\begin{bmatrix} 
\f{U}\hk & \f{U}_{\perp}\hk
\end{bmatrix}\begin{bmatrix} 
\f{\Sigma}\hk & 0 \\
0 & \f{\Sigma}_{\perp}\hk 
\end{bmatrix}
\begin{bmatrix} 
\f{V}^{(k)*} \\
\f{V}_{\perp}^{(k)*} 
\end{bmatrix}
\end{equation}
is a singular value decomposition with $\f{U}\hk \in \R^{d_1 \times r}$, $\f{U}_{\perp}\hk \in \R^{d_1 \times (d_1 -r)}$, $\f{V}\hk \in \R^{d_2 \times r}$, $\f{V}_{\perp}\hk \in \R^{d_2 \times (d_2 -r)}$, we have that
\begin{equation} \label{eq:MatrixIRLS:localconvergence:2}
\langle \f{Z},W\hk(\f{Z}) \rangle = \langle \f{U}_k^{*} \f{Z} \f{V}_k, \f{H}_k \circ (\f{U}_k^{*} \f{Z} \f{V}_k)\rangle
\end{equation}
where $\f{H}_k\ \in \Rdd$ is as in \Cref{def:optimalweightoperator}.

If $\f{Z} = \y{P}_{T_k^\perp}(\f{\eta}\hkk) \in T_k^\perp$, we know that $\f{U}^{(k)*} \f{Z} = 0$ and $\f{Z} \f{V}^{(k)} = 0$, and therefore
\[
\f{U}_k^{*} \f{Z} \f{V}_k = \begin{bmatrix} 
\f{U}^{(k)*} \\ \f{U}_{\perp}^{(k)*}
\end{bmatrix}
\f{Z}
\begin{bmatrix} 
\f{V}^{(k)} &
\f{V}_{\perp}^{(k)} 
\end{bmatrix}
= 
\begin{pmatrix}
0 & 0 \\
0 & \f{U}_{\perp}^{(k)*} \f{Z} \f{V}_{\perp}^{(k)}	
\end{pmatrix}
\]
with $\f{U}_{\perp}^{(k)*} \f{Z} \f{V}_{\perp}^{(k)} \in \R^{(d_1 -r) \times (d_2 -r)}$.

By assumption of \Cref{lemma:MatrixIRLS:localconvergence:1}, we know that $\epsilon_k = \sigma_{r+1}(\f{X}\hk)$, which means that $r_k := |\{i \in [d]: \sigma_i(\f{X}\hk) > \epsilon_k\}| = r$, and therefore $(\f{H}_k)_{ij} = \epsilon_k^{-2}$ for all $i,j > r$. This entails with \eqref{eq:MatrixIRLS:localconvergence:2} that
\[
\begin{split}
\langle \y{P}_{T_k^\perp}(\f{\eta}\hkk),W\hk(\y{P}_{T_k^\perp}(\f{\eta}\hkk)) \rangle  &= \epsilon_k^{-2} \langle\f{U}_k^{*} \y{P}_{T_k^\perp}(\f{\eta}\hkk) \f{V}_k, \f{U}_k^{*} \y{P}_{T_k^\perp}(\f{\eta}\hkk) \f{V}_k \rangle  \\
&= \epsilon_k^{-2} \langle \y{P}_{T_k^\perp}(\f{\eta}\hkk), \y{P}_{T_k^\perp}(\f{\eta}\hkk)\rangle = \epsilon_k^{-2} \|\y{P}_{T_k^\perp}(\f{\eta}\hkk)\|_F^2,
\end{split}
\]
using the cyclicity of the trace and the fact that $\f{U}_k$ and $\f{V}_k$ are orthonormal matrices.

Inserting this into \eqref{eq:MatrixIRLS:localconvergence:1p5}, we obtain
\begin{equation} \label{eq:MatrixIRLS:localconvergence:3}
\begin{split}
\|\f{\eta}\hkk\|_{S_\infty}^2 &\leq c(\mu_0,r,d_1,d_2)^2 \epsilon_k^{2} \left\langle \y{P}_{T_k^\perp}(\f{\eta}\hkk),W\hk(\y{P}_{T_k^\perp}(\f{\eta}\hkk)) \right\rangle \\
&\leq c(\mu_0,r,d_1,d_2)^2 \epsilon_k^{2} \left\langle \f{\eta}\hkk , W\hk(\f{\eta}\hkk)\right\rangle,
\end{split}
\end{equation}
where the last inequality holds since $W^{(k)}$ is positive definite and since 
\[
\left\langle \y{P}_{T_k^{\perp}}(\f{\eta}\hkk) , W\hk(\y{P}_{T_k}(\f{\eta}\hkk))\right\rangle = 0
\]
due to the orthogonality of $T_k$ and $T_k^{\perp}$. Due to \Cref{lemma:MatrixIRLS:Xk:optimality}, we know that the new iterate $\f{X}\hkk$ fulfills
\begin{equation*}
\begin{split}
    0&=\langle W\hk(\f{X}^{(k+1)}),\f{\eta}\hkk\rangle=\langle W^{(k)}(\f{\eta}\hkk+\f{X}^0),\f{\eta}\hkk\rangle,
\end{split}
\end{equation*}
and therefore 
\[
\left\langle \f{\eta}\hkk , W\hk(\f{\eta}\hkk)\right\rangle = - \left\langle W^{(k)}(\f{X}^0), \f{\eta}\hkk\right \rangle \leq \|W\hk(\f{X}^0)\|_{S_1} \|\f{\eta}\hkk\|_{S_\infty},
\]
using H{\"o}lder's inequality for Schatten-$p$ (quasi-)norms (cf. Theorem 11.2 of \cite{Gohberg2000}). Dividing \eqref{eq:MatrixIRLS:localconvergence:3} by $\|\f{\eta}\hkk\|_{S_\infty}$ concludes the proof of \Cref{lemma:MatrixIRLS:localconvergence:1}.
\end{proof}

In order to obtain a fast local convergence rate, it is crucial to bound $\|W\hk(\f{X}^0)\|_{S_1}$. For this, we split $\|W\hk(\f{X}^0)\|_{S_1}$ into three parts and estimate the parts separately by using the classical singular subspace perturbation result of \Cref{lemma:Wedinbound}.
\begin{lemma} \label{lemma:MatrixIRLS:localconvergence:2}
	Let $W\hk: \Rdd \to \Rdd$ be the weight operator \eqref{eq:def:W} of \Cref{def:optimalweightoperator} corresponding to $\f{X}\hk$, let $\epsilon_k=\sigma_{r+1}(\f{X}\hk) = \sigma_r\hk$ and $\f{X}^0 \in \Rdd$ be a rank-$r$ matrix. Assume that there exists $0 < \zeta < 1$ such that
\begin{equation} \label{eq:MatrixIRLS:closeness:assumption:B8}
\|\f{X}\hk - \f{X}^0\|_{S_{\infty}} \leq \zeta \sigma_r(\f{X}^0).
\end{equation}	
	Then
	\[
	\big\|W\hk(\f{X}^0)\big\|_{S_1}  \leq r(1-\zeta)^{-2} \sigma_r(\f{X}^0)^{-1} \left(1 + 4 \frac{\|\f{\eta}\hk\|_{S_\infty}}{\epsilon_k} \frac{\sigma_1(\f{X}^0)}{\sigma_r(\f{X}^0)} + 2 \frac{\|\f{\eta}\hk\|_{S_\infty}^2}{\epsilon_k^{2}} \frac{\sigma_1(\f{X}^0)}{\sigma_r(\f{X}^0)} \right).
	\]
\end{lemma}
\begin{proof}
Recalling the notation $\sigma_{\ell}\hk = \sigma_{\ell}(\f{X}\hk)$ for the $\ell$-th singular value of $\f{X}\hk$ and the decomposition 
\begin{equation} \label{eq:Hk:blockstructure}
\f{H}_k = \begin{bmatrix}
\f{H}\hk &  \f{H}_{1,2}\hk \\
\f{H}_{2,1}\hk & \epsilon_k^{-2} \mathbf{1}
\end{bmatrix} 
\end{equation}
of \eqref{eq:Wk:simplification:1}, we bound the entries of the different blocks $\f{H}\hk$, $\f{H}_{1,2}\hk$ and $\f{H}_{2,1}\hk$ separately. 

Since $(\f{H}_k)_{ij} = \Big(\max(\sigma_i^{(k)},\epsilon_k) \max(\sigma_j^{(k)},\epsilon_k)\Big)^{-1}$ for each $i \in [d_1]$ and $j \in [d_2]$ due to definition of $\f{H}_k$, we observe that

\begin{equation} \label{eq:MatrixIRLS:localconvergence:Hbound:1}
\max_{i \in [r], j \in [r]}(\f{H}\hk)_{ij} \leq (\sigma_r\hk)^{-2},
\end{equation}
and
\begin{equation} \label{eq:MatrixIRLS:localconvergence:Hbound:2}
\max\left(\max_{i,j}((\f{H}_{1,2}\hk)_{ij}),\max_{i,j}((\f{H}_{2,1}\hk)_{ij}) \right) = \max_{i \in [r], r+1 \leq j \leq d_2} (\f{H}_{1,2}\hk)_{ij}  \leq (\sigma_r\hk)^{-1} \epsilon_k^{-1},
\end{equation}

In view of these entrywise bounds on the submatrices of $\f{H}_k$ and $\f{H}_2\hk$, we compute, using \eqref{eq:Wk:simplification:1}, that

\begin{equation*}
\begin{split}
&\big\|W\hk(\f{X}^0)\big\|_{S_1} = \left\|  \begin{bmatrix} 
	\f{U}\hk & \f{U}_{\perp}\hk
\end{bmatrix}
\left(
\begin{bmatrix}
\f{H}\hk &  \f{H}_{1,2}\hk \\
\f{H}_{2,1}\hk & \epsilon_k^{-2} \mathbf{1}
\end{bmatrix} 
\circ 
\begin{bmatrix}
\f{U}^{(k)*} \f{X}^0 \f{V}^{(k)} &  \f{U}^{(k)*} \f{X}^0 \f{V}_{\perp}^{(k)} \\
\f{U}_{\perp}^{(k)*} \f{X}^0 \f{V}^{(k)} & \f{U}_{\perp}^{(k)*} \f{X}^0 \f{V}_{\perp}^{(k)} 
\end{bmatrix}
\right)
\begin{bmatrix} 
	\f{V}^{(k)*} \\ \f{V}_{\perp}^{(k)*}
\end{bmatrix}  \right\|_{S_1} \\
 &\leq \left\| \f{U}\hk[\f{H}\hk \circ  (\f{U}^{(k)*} \f{X}^0 \f{V}^{(k)})]\f{V}^{(k)*}\right\|_{S_1} + \left\|\f{U}_k
\begin{bmatrix}
0 &  \f{H}_{1,2}\hk \circ (\f{U}^{(k)*} \f{X}^0 \f{V}_{\perp}^{(k)}) \\
\f{H}_{2,1}\hk \circ (\f{U}_{\perp}^{(k)*} \f{X}^0 \f{V}^{(k)}) & 0
\end{bmatrix}
\f{V}_k^{*}
\right\|_{S_1} \\
 &+ \epsilon_k^{-2} \left\| \f{U}_{\perp}^{(k)}\f{U}_{\perp}^{(k)*} \f{X}^0 \f{V}_{\perp}^{(k)}\f{V}_{\perp}^{(k)*}\right\|_{S_1}  =: \textsc{(I)} + \textsc{(II)} + \textsc{(III)}.
\end{split}
\end{equation*}
We now bound the terms \textsc{(I)}, \textsc{(II)} and \textsc{(III)} separately.

First, we see that
\[
\begin{split}
\textsc{(I)} &= \left\|\f{H}\hk \circ  (\f{U}^{(k)*} \f{X}^0 \f{V}^{(k)})\right\|_{S_1} \leq \sqrt{r} \left\|\f{H}\hk \circ  (\f{U}^{(k)*} \f{X}^0 \f{V}^{(k)})\right\|_{F} \\
&\leq \sqrt{r} \left\|\f{H}\hk \circ  (\f{U}^{(k)*} \f{X}\hk \f{V}^{(k)})\right\|_{F} + \sqrt{r}  \left\|\f{H}\hk \circ  (\f{U}^{(k)*} \eta\hk \f{V}^{(k)})\right\|_{F} 
 \\
 &\leq \sqrt{r} \left\| \f{H}\hk \circ \f{\Sigma}\hk \right\|_{F} + \sqrt{r} (\sigma_r\hk)^{-2} \|\f{U}^{(k)*} \f{\eta}\hk \f{V}^{(k)}\|_F,
 \end{split}
 \]
 where we used the Cauchy-Schwarz inequality in the first inequality, the notation $\f{\eta}\hk= \f{X}\hk - \f{X}^0$ and the triangle inequality in the second inequality, and finally, \eqref{eq:MatrixIRLS:localconvergence:Hbound:1} in the third inequality. $\Sigma\hk \R^{r \times r}$ is here as in \eqref{eq:Xk:bothsvds}.
 
 Since
 \[
\left\| \f{H}\hk \circ \f{\Sigma}\hk \right\|_{F} = \left(\sum_{i=1}^r (\sigma_i\hk)^{-2} \right)^{1/2} \leq  \sqrt{r}(\sigma_r\hk)^{-1}
 \]
 and 
 \[
 \|\f{U}^{(k)*} \f{\eta}\hk \f{V}^{(k)}\|_F\leq \sqrt{r} \|\f{U}^{(k)*} \f{\eta}\hk \f{V}^{(k)}\|_{S_{\infty}} \leq \sqrt{r} \|\f{\eta}\hk\|_{S_{\infty}} \leq \sqrt{r} \zeta \sigma_r(\f{X}^0)
 \]
from assumption \eqref{eq:MatrixIRLS:closeness:assumption:B8}, it follows then that
\[
\textsc{(I)} \leq r (\sigma_r\hk)^{-2} \left( \sigma_r\hk+ \zeta \sigma_r(\f{X}^0)\right). 
\] 
We can use the proximity assumption \eqref{eq:MatrixIRLS:closeness:assumption:B8} further to get rid of the dependence on $k$ in the bound, as
\[
\sigma_r(\f{X}^0)= \sigma_r(\f{X}\hk-\f{\eta}\hk ) \leq \sigma_r\hk+ \sigma_{1}(\f{\eta}\hk) = \sigma_r\hk + \|\f{\eta}\hk\|_{S_{\infty}} \leq    \sigma_r\hk+ \zeta \sigma_r(\f{X}^0),
\]
using $\sigma_{i+j-1}(\f{A})\leq \sigma_i(\f{A}+\f{B})+\sigma_j(\f{B})$ for any $i,j$ (cf. Theorem 3.3.16 of \cite{horn_johnson}) with $\f{A}+\f{B} = X^{(k)}-\eta^{(k)}$ and $\f{B}= \eta^k$ so that
\begin{equation} \label{eq:MatrixIRLS:localconvergence:sigmarXk:bound}
\sigma_r\hk \geq (1-\zeta) \sigma_r(\f{X}^0),
\end{equation}
and hence
\begin{equation} \label{eq:MatrixIRLS:localconvergence:summandI:bound}
\textsc{(I)}  \leq  r \sigma_r(\f{X}^0)^{-2} (1-\zeta)^{-2} \left( \sigma_r(\f{X}^0) (1 -\zeta)+ \zeta \sigma_r(\f{X}^0)\right) = r (1-\zeta)^{-2} \sigma_r(\f{X}^0)^{-1}. 
\end{equation}

For the term \textsc{(II)}, we compute that
\begin{equation*}
\begin{split}
\textsc{(II)} &\leq \sqrt{2r} 
\left\|
\begin{bmatrix}
0 &  \f{H}_{1,2}\hk \circ (\f{U}^{(k)*} \f{X}^0 \f{V}_{\perp}^{(k)}) \\
\f{H}_{2,1}\hk \circ (\f{U}_{\perp}^{(k)*} \f{X}^0 \f{V}^{(k)}) & 0
\end{bmatrix}
\right\|_{F}
\\
&\leq \sqrt{2r}  (\sigma_r\hk)^{-1} \epsilon_k^{-1} \left(\left\| \f{U}^{(k)*} \f{X}^0 \f{V}_{\perp}^{(k)}\right\|_F + \left\| \f{U}_{\perp}^{(k)*} \f{X}^0 \f{V}^{(k)} \right\|_{F}\right) \\
&\leq \sqrt{2r} (\sigma_r\hk)^{-1} \epsilon_k^{-1} \left( \|\f{U}^{(k)*} \f{U}_0 \f{\Sigma}_0\|_F \|\f{V}_0^{*} \f{V}_{\perp}^{(k)}\|_{S_\infty} + \|\f{U}_{\perp}^{(k)*}\f{U}_0 \|_{S_{\infty}} \|\f{\Sigma}_0 \f{V}_0^{*}\f{V}^{(k)} \|_{F}\right),
\end{split}
\end{equation*}
using the singular value decomposition $\f{X}^0 = \f{U}_0 \f{\Sigma}_0 \f{V}_0^{*}$ of the rank-$r$ matrix $\f{X}^0$ with $\f{U}_0 \in \R^{d_1 \times r}$, $\f{V}_0 \in \R^{d_2 \times r}$. This allows us to use the singular subspace perturbation result of \Cref{lemma:Wedinbound}, so that $\|\f{V}_0^{*} \f{V}_{\perp}^{(k)}\|_{S_\infty}$ and $\|\f{U}_{\perp}^{(k)*}\f{U}_0 \|_{S_{\infty}}$ can compensate for the negative power of the $\epsilon_k$, avoiding a blow-up of term \textsc{(II)}: Indeed, using \Cref{lemma:Wedinbound} with $\f{X}=\f{X}^0$, $\widehat{\f{X}}=\f{X}\hk$, $\alpha = \sigma_r(\f{X}^0)$ and $\delta = (1-\zeta) \sigma_r(\f{X}^0)$ results in
\[
\max(\|\f{V}_0^{*} \f{V}_{\perp}^{(k)}\|_{S_\infty},\|\f{U}_{\perp}^{(k)*}\f{U}_0 \|_{S_{\infty}}) \leq \frac{\sqrt{2}\|\f{\eta}\hk\|_{S_{\infty}}}{(1-\zeta)\sigma_r(\f{X}^0)},
\]
and since $\|\f{U}^{(k)*} \f{U}_0 \f{\Sigma}_0\|_F \leq \|\f{\Sigma}_0\|_F \leq \sqrt{r} \sigma_1(\f{X}^0)$, 
$\|\f{\Sigma}_0 \f{V}_0^{*}\f{V}^{(k)}\|_F \leq \sqrt{r} \sigma_1(\f{X}^0)$, we obtain with \eqref{eq:MatrixIRLS:localconvergence:sigmarXk:bound} that
\begin{equation} \label{eq:MatrixIRLS:localconvergence:summandII:bound}
\textsc{(II)}  \leq 4r (1-\zeta)^{-2} \sigma_r(\f{X}^0)^{-1} \frac{\|\f{\eta}\hk\|_{S_\infty}}{\epsilon_k}\frac{\sigma_1(\f{X}^0)}{\sigma_r(\f{X}^0)}.
\end{equation}

It remains to bound the last term \textsc{(III)}. For \textsc{(III)}, we can use the subspace perturbation lemma \emph{twice} in the same summand such that
\begin{equation} \label{eq:MatrixIRLS:localconvergence:summandIII:bound}
\begin{split}
\textsc{(III)} &= \epsilon_k^{-2} \left\| \f{U}_{\perp}^{(k)}\f{U}_{\perp}^{(k)*} \f{X}^0 \f{V}_{\perp}^{(k)}\f{V}_{\perp}^{(k)*}\right\|_{S_1} = \epsilon_k^{-2} \|\f{U}_{\perp}^{(k)*} \f{X}^0 \f{V}_{\perp}^{(k)}\|_{S_1} \leq \sqrt{r} \epsilon_k^{-2} \|\f{U}_{\perp}^{(k)*} \f{X}^0 \f{V}_{\perp}^{(k)}\|_{F} \\
&\leq \sqrt{r} \epsilon_k^{-2} \|\f{U}_{\perp}^{(k)*} \f{U}_0\|_{S_\infty} \|\f{\Sigma}_0\|_F \|\f{V}_0^{*} \f{V}_{\perp}^{(k)}\|_{S_\infty} \\
&\leq \sqrt{r} \epsilon_k^{-2} \frac{\sqrt{2} \|\f{\eta}\hk\|_{S_\infty}}{(1-\zeta)\sigma_r(\f{X}^0)} \sqrt{r} \sigma_1(\f{X}^0)  \frac{\sqrt{2} \|\f{\eta}\hk\|_{S_\infty}}{(1-\zeta)\sigma_r(\f{X}^0)} = 2r (1-\zeta)^{-2}  \sigma_r(\f{X}^0)^{-1} \frac{\|\f{\eta}\hk\|_{S_\infty}^2}{\epsilon_k^2} \frac{\sigma_1(\f{X}^0)}{\sigma_r(\f{X}^0)}.
\end{split}
\end{equation}
Combining \eqref{eq:MatrixIRLS:localconvergence:summandI:bound,eq:MatrixIRLS:localconvergence:summandII:bound,eq:MatrixIRLS:localconvergence:summandIII:bound} finally yields the statement of \Cref{lemma:MatrixIRLS:localconvergence:2}.
\end{proof}

\subsection{Wrapping up the proof} \label{section:wrappingproof}
We can now put \Cref{lemma:MatrixIRLS:localRIP}, \Cref{lemma:MatrixIRLS:localconvergence:1} and \Cref{lemma:MatrixIRLS:localconvergence:2} together to prove the local convergence statement of \Cref{cor:MatrixIRLS:localconvergence:matrixcompletion}, showing also that we attain locally quadratic convergence.
\begin{proof}[Proof of \Cref{cor:MatrixIRLS:localconvergence:matrixcompletion}]
	
Let $k = k_0$ and $\f{X}\hk$ be the $k$-th iterate of \texttt{MatrixIRLS} with the parameters stated in \Cref{cor:MatrixIRLS:localconvergence:matrixcompletion}. Under the sampling model of \Cref{cor:MatrixIRLS:localconvergence:matrixcompletion}, if the number of samples $m$ fulfills $m \geq C \mu_0 r (d_1 + d_2) \log(d_1 + d_2)$, where $C$ is the constant of \Cref{lemma:MatrixIRLS:localRIP}, we know from \Cref{lemma:MatrixIRLS:localRIP} that with a probability of at least $1 - 2 D^{-2}$, inequality \eqref{eq:localcvg:ass:1} is satisfied with $c(\mu_0,r,d_1,d_2) =  \sqrt{\frac{\widetilde{C} d \log(D)}{\mu_0 r}}$, if furthermore $\eta\hk:=\f{X}\hk - \f{X}^0$ fulfills
\begin{equation} \label{eq:proximitybound}
\| \eta\hk \|_{S_\infty} \leq \xi \sigma_r(\f{X}^0)
\end{equation}
with 
\begin{equation} \label{eq:zeta:proof41:1}
\xi \leq C_1 \sqrt{\frac{\mu_0 r}{d}},
\end{equation}
and thus, by \Cref{lemma:MatrixIRLS:localconvergence:1},
\begin{equation} \label{eq:proximitybound2}
\|\f{X}\hkk - \f{X}^0\|_{S_\infty} \leq \frac{\widetilde{C} d \log(D)}{\mu_0 r} \epsilon_k^{2} \|W\hk(\f{X}^0)\|_{S_1}.
\end{equation}
We denote the event that this is fulfilled by $E$. Furthermore, on this event, if $\xi \leq 1/2$ in \eqref{eq:proximitybound} and denoting the condition number by $\kappa=\sigma_1(\f{X}^0)/\sigma_r(\f{X}^0)$, it follows from \Cref{lemma:MatrixIRLS:localconvergence:2} that 
\[
\|\f{X}\hkk - \f{X}^0\|_{S_\infty}  \leq \frac{\widetilde{C} d \log(D)}{\mu_0}  4 \sigma_r(\f{X}^0)^{-1} \left(\epsilon_k^2 + 4 \epsilon_k  \|\f{\eta}\hk\|_{S_\infty} \kappa + 2 \|\f{\eta}\hk\|_{S_\infty}^2
\kappa \right)
\]
Furthermore, if $\f{X}_r\hk \in \Rdd$ denotes the best rank-$r$ approximation of $\f{X}\hk$ in any unitarily invariant norm, we estimate that
\[
\epsilon_k \leq \sigma_{r+1}(\f{X}\hk) =\|\f{X}\hk-\f{X}_r\hk\|_{S_\infty} \leq \|\f{X}\hk-\f{X}^0\|_{S_\infty} = \|\eta\hk\|_{S_\infty},
\]
Inserting these two bounds into \eqref{eq:proximitybound2}, we obtain
\[
\|\eta\hkk\|_{S_\infty} = \|\f{X}\hkk - \f{X}^0\|_{S_\infty} \leq \frac{\widetilde{C} d \log(D)}{\mu_0}  4  \sigma_r(\f{X}^0)^{-1} \left(1 + 6 \kappa \right) \|\eta\hk\|_{S_\infty}^2.
\]
Finally, if, additionally, \eqref{eq:proximitybound} is satisfied for 
\begin{equation} \label{eq:zeta:proof41:2}
\xi \leq \frac{\mu_0}{4(1+6 \kappa) d \log(D) \widetilde{C}},
\end{equation}
we conclude that
\[
\| \eta\hkk \|_{S_\infty} < \| \eta\hk \|_{S_\infty}
\]
and also, we observe a quadratic decay in the spectral error such that
\[
\| \eta\hkk \|_{S_\infty} \leq \mu \| \eta\hk \|_{S_\infty}^2
\]
with a constant $\mu = \frac{4 \widetilde{C} d \log(D) \left(1 + 6 \kappa \right) }{\mu_0 \sigma_r(\f{X}^0)}$. This shows inequality \eqref{eq:MatrixIRLS:closeness:assumption} of \Cref{cor:MatrixIRLS:localconvergence:matrixcompletion}.

To show the remaining statement, we can use \Cref{lemma:etaksigmarp1Xk} to show that if $\f{X}\hk$ is close enough to $\f{X}^0$, we can ensure that the $(r+1)$-st singular value $\sigma_{r+1}(\f{X}\hk)$ of the current iterate is strictly decreasing. More precisely, assume now the stricter assumption of 
\begin{equation} \label{eq:condition:etaksigmarp1Xk}
\|\eta\hk\|_{S_\infty} \leq \sqrt{\frac{\mu_0 r}{4 \widetilde{C} d (d-r) \log(D)}} \xi  \sigma_r(\f{X}^0).
\end{equation}
In fact, if $\xi$ fulfills \eqref{eq:zeta:proof41:1} and \eqref{eq:zeta:proof41:2}, we can conclude that on the event $E$, 
\[
\begin{split}
\sigma_{r+1}(\f{X}\hkk) &\leq \|\eta\hkk\|_{S_\infty} \leq  \frac{\widetilde{C} d \log(D)}{\mu_0}  4  \sigma_r(\f{X}^0)^{-1} \left(1 + 6 \kappa \right) \|\eta\hk\|_{S_\infty}\cdot\|\eta\hk\|_{S_\infty}  \\
&< \frac{\widetilde{C} d \log(D)}{\mu_0}  4  \sigma_r(\f{X}^0)^{-1} \left(1 + 6 \kappa \right) \sqrt{\frac{\mu_0 r}{4 \widetilde{C} d (d-r) \log(D)}}  \xi \sigma_r(\f{X}^0)   \\
&\cdot \sqrt{\frac{4 \widetilde{C} d (d-r) \log(D)}{\mu_0 r}}  \sigma_{r+1}(\f{X}\hk) \leq  \sigma_{r+1}(\f{X}\hk)
\end{split}
\]
using \Cref{lemma:etaksigmarp1Xk} for one factor $\|\eta\hk\|_{S_\infty}$ and \eqref{eq:condition:etaksigmarp1Xk} for the other factor $\|\eta\hk\|_{S_\infty}$ in the third inequality, and \eqref{eq:zeta:proof41:2} in the last inequality. Taking the update rule \eqref{eq:MatrixIRLS:epsdef} for the smoothing parameter into account, this implies that $\epsilon_{k+1} = \sigma_{r+1}(\f{X}\hkk)$, which ensures that the first statement of \Cref{cor:MatrixIRLS:localconvergence:matrixcompletion} is fulfilled likewise for iteration $k+1$. By induction, this implies that $\f{X}^{(k+\ell)} \xrightarrow{\ell \to \infty} \f{X}^0$, which finishes the proof of \Cref{cor:MatrixIRLS:localconvergence:matrixcompletion}.
\end{proof}

The presented proof of  \Cref{cor:MatrixIRLS:localconvergence:matrixcompletion} has certain similarities with the proof of \emph{local superlinear convergence} of Theorem 11 in \cite{KS18} for a related IRLS algorithm designed for Schatten-$p$ quasi-norm minimization. However, that proof is not applicable to the matrix completion setting, and furthermore, is not extendable to a log-determinant objective as used in this paper. As observed in \cite{KS18}, it is not possible to obtain superlinear (or quadratic) convergence rates for the IRLS methods of \cite{Fornasier11,Mohan10}.

\section{Proof of \Cref{thm:wellconditioning}} \label{sec:proof:wellconditioning}
In this section, we provide a result about the \emph{spectrum} of the system matrix 
\begin{equation} \label{eq:systemmatrix:good}
\f{A} := \f{D}_k + P_{T_k}^* P_{\Omega}^* P_{\Omega}  P_{T_k} := \epsilon_k^{2} \left(\f{D}_{S_k}^{-1}- \epsilon_k^{2} \f{I}_{S_k}\right)^{-1} + P_{T_k}^* P_{\Omega}^* P_{\Omega}  P_{T_k}
\end{equation}
of \eqref{eq:gamma:system} in \Cref{algo:MatrixIRLS:mainstep:implementation}. We recall that solving a linear system with $\f{A}$ constitutes the main computational step in our implementation of \texttt{MatrixIRLS}. 

It is well-known that the shape of the spectrum of the system matrix $\f{A}$ plays an important role in the convergence of the conjugate gradient iterations. In particular, the CG method terminates after $\ell$ iterations (in exact arithmetic) if $\f{A}$ has $\ell$ distinct eigenvalues (cf. Theorem 5.4 of \cite{NocedalWright06}) and a bound on the error $\f{\gamma}_{\ell} - \f{\gamma}^*$ of the $\ell$-th iterate $\f{\gamma}_{\ell}$ to the exact solution $\f{\gamma}^*$ of the linear system (5.36) of \cite{NocedalWright06} can be provided by
\[
	\left\langle \f{\gamma}_{\ell}- \f{\gamma}^*, \f{A}(\f{\gamma}_{\ell}- \f{\gamma}^*)\right\rangle  \leq 2 \left(\frac{\sqrt{\kappa(\f{A})}-1}{\sqrt{\kappa(\f{A})}+1} \right)^{\ell} \left\langle \f{\gamma}_{0}- \f{\gamma}^*, \f{A}(\f{\gamma}_{0}- \f{\gamma}^*)\right\rangle,
\]
where $\kappa(\f{A}):= \lambda_{\max}(\f{A})/\lambda_{\min}(\f{A})$ is the condition number of $\f{A}$.

It has been a common problem for IRLS methods that the linear systems to be solved become ill-conditioned close to the desired (low-rank or sparse, depending on the problem) solution \cite{Daubechies10,Fornasier11,Fornasier16}. Close to the solution the smoothing parameter $\epsilon_k$ is typically very small, resulting in ``very large weights'' on large parts of the domain induced by the quadratic form
\[
\langle \f{X},W\hk(\f{X})\rangle.
\]  
For the sparse recovery problem, it has been observed \cite{Voronin12} that this blow-up can be a problem for an inexact solver of the weighted least squares system, and in \cite{Fornasier16}, an analysis was pursued for an IRLS algorithm for the sparse recovery problem about with which precision the linear system for each outer iteration $k$ needs to be solved by a conjugate gradient method to ensure overall convergence.

However, the underlying issue of bad conditioning of the IRLS system matrices was not addressed or solved in \cite{Fornasier16} (see Section 5.2 of \cite{Fornasier16} for a discussion). 

\Cref{thm:wellconditioning}, which we now show, argues that by computing the weighted least squares update via \Cref{algo:MatrixIRLS:mainstep:implementation}, these issues do not arise for \texttt{MatrixIRLS} in the same manner.

\begin{proof}[{Proof of \Cref{thm:wellconditioning}}]
	Recall the definitions $\f{D}_k = \epsilon_k^{2} \left(\f{D}_{S_k}^{-1}- \epsilon_k^{2} \f{I}_{S_k}\right)^{-1}$ and $\f{A}= \f{D}_k + P_{T_k}^* P_{\Omega}^* P_{\Omega}  P_{T_k}$. We know that the eigenvalues of $\f{D}_{S_k}^{-1}$ are just the inverses of the entries of the matrices $\f{H}\hk$, $\f{H}_{1,2}\hk$ and $\f{H}_{2,1}\hk$ in the block decomposition of the matrix $\f{H}_k \in \Rdd$ that defines the weight operator $W\hk$. By \eqref{eq:MatrixIRLS:localconvergence:Hbound:1} and \eqref{eq:MatrixIRLS:localconvergence:Hbound:2}, we can lower bound these eigenvalues by $\sigma_r(\f{X}\hk) \epsilon_k=\sigma_r\hk \epsilon_k$, and therefore,
	
	\begin{equation} \label{eq:spectrum:Dk}
	\|\f{D}_{k}\|_{S_\infty} \leq \frac{\epsilon_k^{2}}{\sigma_r(\f{X}\hk) \epsilon_k-\epsilon_k^{2}}  = \frac{\epsilon_k}{\sigma_r(\f{X}\hk) -\epsilon_k} \leq \frac{\epsilon_k}{3/4\sigma_r(\f{X}_0) - \epsilon_k},
	\end{equation}
	using that $\sigma_r(\f{X}\hk) \geq (1- 1/4) \sigma_r(\f{X}_0)$ since $\|\f{X}\hk - \f{X}_0\|_{S_\infty} \leq \frac{1}{4} \sigma_{r}(\f{X}_0)$, see also \eqref{eq:MatrixIRLS:localconvergence:sigmarXk:bound}. Also, since $\epsilon_k= \sigma_{r+1}(\f{X}\hk) \leq \|\f{X}\hk - \f{X}_0\|_{S_\infty} \leq \frac{1}{4} \sigma_{r}(\f{X}_0)$ and further $\epsilon_k= \sigma_{r+1}(\f{X}\hk) \leq C_1 \left(\frac{\mu_0 r}{d}\right)\sigma_{r}(\f{X}_0)$, we have that
	\[
	\frac{\epsilon_k}{(3/4\sigma_r(\f{X}_0))-\epsilon_k} \leq \frac{C_1 \frac{\mu_0 r}{d} }{(3/4-1/4)} \frac{\sigma_r(\f{X}_0)}{\sigma_r(\f{X}_0)} \leq 2 C_1\frac{\mu_0 r}{d}.
	\]
	This implies that
	\[
	0 \leq \lambda_{\min}(\f{D}_k) \leq \lambda_{\max}(\f{D}_k)= \|\f{D}_{k}\|_{S_\infty} \leq 2 C_1 \frac{\mu_0 r}{d} \leq 2 C_1 \frac{m}{C d D \log(D) } \leq \frac{m}{d_1 d_2},
	\]
	if the constant $C_1 >0$ is small enough, using the lower bound on the sample complexity $m$.
	
	The second summand in $\f{A}$, the matrix $P_{T_k}^* P_{\Omega}^*P_{\Omega} P_{T_k}$, is positive semidefinite already due to its factorized form. We note that by following the proof of \Cref{cor:MatrixIRLS:localconvergence:matrixcompletion}, we see that under the assumptions of \Cref{thm:wellconditioning}, we have that for $\y{P}_{T_k}: \Rdd \to \Rdd, \f{Z} \mapsto P_{T_k} P_{T_k}^*(\f{Z})$ and $\y{P}_{\Omega}:\Rdd \to \Rdd, \f{Z} \mapsto P_{\Omega}^* P_{\Omega}(\f{Z})$,
	\[
	\begin{split}
	&\frac{d_1 d_2}{m} \left\|P_{T_k} \left[ P_{T_k}^*P_{\Omega}^*P_{\Omega} P_{T_k}- \frac{m}{d_1 d_2} \f{I} \right] P_{T_k}^*\right\|_{S_\infty} = \left\| \frac{d_1 d_2}{m} P_{T_k} P_{T_k}^*P_{\Omega}^*P_{\Omega} P_{T_k} P_{T_k}^*-P_{T_k} P_{T_k}^*\right\|_{S_\infty} \\
	&=\left\| \frac{d_1 d_2}{m} \y{P}_{T}\y{P}_{\Omega}\y{P}_{T}- \y{P}_{T}\right\|_{S_\infty}
\leq \frac{4}{10}
	\end{split}
	\]
	on an event $E$ that holds with high probability.
	
	As $ P_{T_k}$ is a matrix with orthonormal columns such that $P_{T_k}^* P_{T_k} = \f{I}$, this implies that 
	\[
	 \left\|P_{T_k}^*P_{\Omega}^*P_{\Omega} P_{T_k}- \frac{m}{d_1 d_2} \f{I}\right\|_{S_\infty} \leq \frac{4 m}{10 d_1 d_2}
	\]
	on the event $E$. Thus, the bound on the spectrum of $\f{A}$ follows from this and \eqref{eq:spectrum:Dk} since
	\[
	\begin{split}
	\left\| \f{A}- \frac{3}{2} \frac{m}{d_1 d_2} \f{I} \right\|_{S_\infty} &= \left\| \f{D}_k + P_{T_k}^* P_{\Omega}^*P_{\Omega} P_{T_k}- \frac{3}{2} \frac{m}{d_1 d_2} \f{I} \right\|_{S_\infty} \leq \left\| \f{D}_k - \frac{1}{2} \frac{m}{d_1 d_2} \f{I} \right\|_{S_\infty} + \left\| P_{T_k}^* P_{\Omega}^*P_{\Omega} P_{T_k} -  \frac{m}{d_1 d_2} \f{I} \right\|_{S_\infty}  \\
	&\leq \frac{1}{2} \frac{m}{d_1 d_2} + \frac{4}{10}  \frac{m}{d_1 d_2} = \frac{9}{10}  \frac{m}{d_1 d_2}. 
	\end{split}
	\]
	The condition number bound follows immediately since $\kappa(\f{A})= \frac{\lambda_{\max}(\f{A})}{\lambda_{\min}(\f{A})} \leq 4$.
\end{proof}

As a summary, since \Cref{thm:wellconditioning} gives a bound on the condition number of the linear system matrix $\f{A}$  that is a \emph{small constant}, the theory of the conjugate gradient methods suggests that very good solutions can be found already after few, in particular, after 
\[
N_{\text{CG\_inner}} = \text{cst.}
\]
CG iterations (where \text{cst.} is small), for each IRLS iteration, at least in the neighborhood of a low-rank matrix $\f{X}_0$ that is compatible with the measurements.

 Taking into account the statement of \Cref{thm:MatrixIRLS:computationalcost:Xkk}, this suggests that at least locally, a new iterate $\f{X}\hkk$ can be calculated with a time complexity of 
\[
O \left( (m r + r^2 D) \cdot N_{\text{CG\_inner}} \right) = O \left( m r + r^2 D \right).
\]

\section{Remarks to MatrixIRLS as a saddle-escaping smoothing Newton method} \label{sec:remarks:smoothing:newton}
We briefly elaborate on the interpretation of \texttt{MatrixIRLS} as a saddle-escaping smoothing method.

If $\epsilon_k > 0$ and if $F_{\epsilon_k}:\Rdd \to R$ is the $\epsilon_k$-smoothed log-det objective of \eqref{eq:smoothing:Fpeps}, it can be shown that $F_{\epsilon_k}$ is continuously differentiable with $\epsilon_k^{-2}$-Lipschitz gradient 
 $\nabla F_{\epsilon_k}(\f{X}) = \f{U} \dg \bigg(\frac{\sigma_i(\f{X})}{\max(\sigma_i(\f{X}),\epsilon_k)^{2}}\bigg)_{i=1}^d \f{V}^*$
for any matrix $\f{X}$ with singular value decomposition $\f{X} = \f{U} \dg\big(\sigma(\f{X})\big) \f{V}^* = \f{U} \dg\big(\sigma \big) \f{V}^*$. This can be shown by using results from \cite{Lewis05_Nonsm1,AnderssonCarlssonPerfekt16}. Additionally, it holds that $\nabla F_{\epsilon_k}$ is differentiable at $\f{X}$ if and only if the second derivative $f_{\epsilon_k}'':\R \to \R$ of $f_{\epsilon_k}$ from \eqref{eq:smoothing:Fpeps} exists at all $\sigma=\sigma_i(\f{X})$, $i \in [d]$, which is the case if $\f{X} \in \mathcal{D}_{\epsilon_k}:=\big\{ \f{X}: \sigma_i (\f{X}) \neq \epsilon_k \text{ for all } i \in [d] \big\}$. The latter statement follows from the calculus of \emph{non-Hermitian L\"owner functions} \cite{Yang09,Ding18}, also called \emph{generalized matrix functions} \cite{Noferini17}, as $\f{X} \mapsto \nabla F_{\epsilon_k}(\f{X})$ is such a function.

Let now $\f{X}\hk \in \mathcal{D}_{\epsilon_k}:=\big\{ \f{X}: \sigma_i (\f{X}) \neq \epsilon_k \text{ for all } i \in [d] \big\}$ with singular value decomposition as in \eqref{eq:Xk:bothsvds}, and $r_k := |\{i \in [d]: \sigma_i(\f{X}\hk) > \epsilon_k\}| = |\{i \in [d]: \sigma_i\hk > \epsilon_k\}|$. In this case, it can be calculated that the Hessian $\nabla^2 F_{\epsilon_k}(\f{X}\hk)$ at $\f{X}\hk$, which is a function that maps $\Rdd$ to $\Rdd$ matrices, satisfies in the case of $d_1=d_2$
\begin{equation} \label{eq:smoothedranksurrogate:Hessianformula}
\nabla^2 F_{\epsilon_k}(\f{X}\hk)(\f{Z}) =  \f{U}_k \begin{bmatrix}\f{M}^{\text{S}} \circ S(\f{U}_k^* \f{Z} \f{V}_{k})+ 	
	\f{M}^{\text{T}} \circ T(\f{U}_k^* \f{Z} \f{V}_k) \end{bmatrix}
		 \f{V}_k^*, 
\end{equation}
for any $\f{Z} \in \Rdd$, where $S:\R^{d \times d} \to \R^{d \times d}$ and  $T:\R^{d \times d} \to \R^{d \times d}$ are the \emph{symmetrization operator} and  \emph{antisymmetrization operator}, respectively, that map any $\f{X} \in \R^{d \times d}$ to
\begin{equation*} 
	S(\f{X})=\frac{1}{2}(\f{X} + \f{X}^*), \quad \text{ and } \quad T(\f{X})= \frac{1}{2}(\f{X} - \f{X}^*)
\end{equation*}
for any $\f{X} \in \R^{d \times d}$, and  $\f{M}^{\text{S}}, \f{M}^{\text{T}}  \in \Rdd$ fulfill 
\[
\f{M}^{\text{S}}=
\left[
\begin{array}{c|c}
-\f{H}\hk & \f{M}_{1,2}^{-} \\
\hline
\f{M}_{2,1}^{-} & \epsilon_k^{-2} \mathbf{1}
\end{array}
\right] \quad \quad 
\f{M}^{\text{T}}=
\left[
\begin{array}{c|c}
-\f{H}\hk & \f{M}_{1,2}^{+} \\
\hline
\f{M}_{2,1}^{+} & \epsilon_k^{-2} \mathbf{1}
\end{array}
\right]
\]
with $\f{H}\hk \in \R^{r_k \times r_k}$ as in \eqref{eq:H:def:small} and the $(d_1-r_k) \times (d_2- r_k)$-matrix of ones $\mathbf{1}$. Furthermore, the matrices $\f{M}_{1,2}^{-},\f{M}_{1,2}^{+} \in (d_1- r_k) \times r_k$ are such that
\[
\left(\f{M}_{1,2}^{\pm}\right)_{ij} = \frac{(\sigma_i^{(k)})^{-1} \pm \sigma_{j+r_k}^{(k)}\epsilon_k^{-2}}{\sigma_i^{(k)} \pm \sigma_{j+r_k}^{(k)}} 
\] 
for $i \in [r_k]$, $j \in [d_2-r_k]$ and
\[
\left(\f{M}_{2,1}^{\pm}\right)_{ij} = \frac{(\sigma_j^{(k)})^{-1} \pm \sigma_{i+r_k}^{(k)}\epsilon_k^{-2}}{\sigma_j^{(k)} \pm \sigma_{i+r_k}^{(k)}} 
\] 
for$j \in [r_k]$, $i \in [d_1-r_k]$. The formula \eqref{eq:smoothedranksurrogate:Hessianformula} for $\nabla^2 F_{\epsilon_k}(\f{X}\hk)$ follows by inserting the operator $\nabla F_{\epsilon_k}$ into Theorem 2.2.6 of \cite{Yang09}, Corollary 3.10 \cite{Noferini17} or Theorem 4 of \cite{Ding18}.

By realizing that $0 \leq \sigma_{\ell}\hk \leq \epsilon_k$ for all $\ell > r_k$, we see that
\[
\frac{1}{(\sigma_i\hk)^{2}} \leq \left(\f{M}_{1,2}^{+}\right)_{ij} = \left(\f{M}_{2,1}^{+}\right)_{ji} \leq  \frac{1}{\sigma_i\hk \epsilon_k}
\]
and 
\[
-\frac{1}{\sigma_i\hk \epsilon_k} \leq \left(\f{M}_{1,2}^{-}\right)_{ij} = \left(\f{M}_{2,1}^{-}\right)_{ji} \leq  \frac{1}{(\sigma_i\hk)^{2}}
\]
for all $i$ and $j$.

Now, comparing $M^{\text{S}}$ and $M^{\text{T}}$ with $\f{H}_k$, see \eqref{eq:Hk:blockstructure}, of the weight operator $W\hk$, we see that the upper left blocks of $M^{\text{S}}$ and $M^{\text{T}}$ are just the \emph{negative} of the upper left block $\f{H}\hk$ of $\f{H}_k$, while the lower right blocks coincide. Furthermore, the lower left and the upper right blocks are related such that
\[
\left|\left(\f{M}_{1,2}^{\pm}\right)_{ij}\right| \leq \frac{1}{\sigma_i\hk \epsilon_k} = (\f{H}_{1,2}\hk)_{ij}
\]
for all $i \in [r_k]$, $j \in [d_2-r_k]$, and 
\[
\left|\left(\f{M}_{2,1}^{\pm}\right)_{ij}\right| \leq  \frac{1}{\sigma_j\hk \epsilon_k} = (\f{H}_{2,1}\hk)_{ij}
\]
for all $i \in [d_2-r_k]$, $j \in [r_k]$.

We now point out the relationship of these considerations to an analysis that was performed in \cite{PaternainMokhtariRibeiro19} for the case of an \emph{unconstrained minimization} of $F_{\epsilon_k}$, assuming furthermore that $F_{\epsilon_k}$ was smooth:

In this case, \cite{PaternainMokhtariRibeiro19} considers using \emph{modified Newton steps} 
\[
\f{X}\hkk := \f{X}\hk - \eta_k \left|\nabla^2 F_{\epsilon_k}(\f{X}\hk)\right|_{c}^{-1} \nabla F_{\epsilon_k}(\f{X}\hk)
\]
where the Hessian $\nabla^2 F_{\epsilon_k}(\f{X}\hk)$ is replaced by a positive definite truncated eigenvalue matrix $\left|\nabla^2 F_{\epsilon_k}(\f{X}\hk)\right|_{c}$, which replaces the large negative eigenvalues of $\nabla^2 F_{\epsilon_k}(\f{X}\hk)$ by their modulus for eigenvalues that have large modulus and eigenvalues of small modulus by an appropriate constant $c$. \cite{PaternainMokhtariRibeiro19} shows that such steps are, unlike conventional Newton steps, which often are \emph{attracted by saddle points}, able to \emph{escape} saddle points with an exponential rate that does \emph{not} depend on the conditioning of the problem. Experimental observations of such behavior has been reported also in other works \cite{Murray2010,Dauphin14}.

In view of this, we observe that the weight operator $W\hk$ is nothing but a refined variant of $\left|\nabla^2 F_{\epsilon_k}(\f{X}\hk)\right|_{c}$, as the eigenvalues of $\nabla^2 F_{\epsilon_k}(\f{X}\hk)$ from \eqref{eq:smoothedranksurrogate:Hessianformula} are simply $\{ (\f{M}^{\text{S}}_{ij}, i \leq j\} \cup \{ (\f{M}^{\text{T}}_{ij}, i < j\}$, c.f., e.g., Theorem 4.5 of \cite{Noferini17}. In particular, the refinement is such that the small eigenvalues of $\nabla^2 F_{\epsilon_k}(\f{X}\hk)$, which can be found in the entries of $\f{M}_{1,2}^{\pm}$ and $\f{M}_{2,1}^{\pm}$, are replaced not by a uniform constant, but by \emph{different} upper bounds $(\sigma_i\hk \epsilon_k)^{-1}$ and $(\sigma_j\hk \epsilon_k)^{-1}$ that depend either on the row index $i$ or the column index $j$.

Besides this connection, there are important differences of our algorithm to the algorithm analyzed in \cite{PaternainMokhtariRibeiro19}. While that paper considers the minimization of a fixed smooth function, we update the smoothing parameter $\epsilon_k$ and thus the function $F_{\epsilon_k}$ at each iteration. Furthermore, Algorithm 1 of \cite{PaternainMokhtariRibeiro19} uses backtracking for each modified Newton step, which would be prohibitive to perform as evaluations of $F_{\epsilon_k}$ are very expensive for our smoothed log-det objectives, as they would require the calculation of all singular values. On the other hand, \texttt{MatrixIRLS} uses full modified Newton steps, and we can assure that these are always a descent direction in our case, as we explain in an upcoming paper. Lastly, we do not add noise to the iterates.

As mentioned in \Cref{surrogate}, \texttt{MatrixIRLS} is by no means the \emph{first} algorithm for low-rank matrix recovery that can be considered as an iteratively reweighted least squares algorithm. However, the IRLS algorithms \cite{Fornasier11,Mohan10,Lai13,KS18} are different from \texttt{MatrixIRLS} not only in their computational aspects, but also since they do \emph{not} allow for a close relationship between their weight operator $W\hk$ and the Hessian $\nabla^2 F_{\epsilon_k}(\f{X}\hk)$ at $\f{X}\hk$ as described above.

\section{Experimental Details} \label{sec:experimental:details}
In this section, we specify some details of the setup and the algorithmic parameters for the experiments presented in \Cref{sec:numerics}. The sample complexity experiments of Figures \ref{fig:sampcomp:1}, \ref{fig:sampcomp:2}, \ref{fig:sampcomp:3} and \ref{fig:sampcomp:4} were conducted on a Linux node with Intel Xeon E5-2690 v3 CPU with 28 cores and 64 GB RAM, using MATLAB R2019a. The experiment of Figure \ref{running_time_MatrixIRLS_vs_R2RILS} was conducted on a Windows 10 laptop with Intel i7 7660U with 2 cores and 8 GB RAM, also using MATLAB R2019a, and all other experiments were conducted on a iMac with 4 GHz Quad-Core Intel Core i7 CPU, using MATLAB R2020b. In the main text we divided the algorithms into three main categories according to the main optimization philosophy behind. But now, for the purpose of our experiment, we categorize the algorithms into algorithms of \emph{first-order} type and of \emph{second-order} type based on whether an algorithm exhibits empirically observed locally superlinear convergence rates or not.

\subsection{Algorithmic Parameter Choice}

All the methods are provided with the true rank $r$ of $\f{X}^0$ as an input parameter. If possible, we use the MATLAB implementation provided by the authors of the respective papers. Below we point out the links from which one can download such implementations. We do not make use of explicit parallelization for any of the methods, but most methods use complied C subroutines to efficiently implement sparse evaluations of matrix factorizations. We base our choice of algorithms on the desire to obtain a representative picture of state-of-the-art algorithms for matrix completion, including in particular those that are scalable to problems with dimensionality in the thousands or more, those that come with the best theoretical guarantees, and those that claim to perform particularly well to complete \emph{ill-conditioned} matrices. 

We set a maximal number of outer iterations for the second-order methods as $N_0 = 400$. The second-order type algorithms considered for this paper, including their parameter choices, are:
\begin{itemize}
	\item \texttt{MatrixIRLS}, as described in \Cref{algo:MatrixIRLS} or, minutely, in \Cref{sec:implementation}. As a stopping criterion, we choose a threshold of $10^{-9}$ for the relative change of the Frobenius norm $\frac{\|\f{X}\hkk - \f{X}\hk\|_F}{\|\f{X}\hk\|_F}$. We use the CG method for solving the linear system \eqref{eq:gamma:system} without any preconditioning. We terminate the CG method if a maximum number of $N_{\text{CG\_{inner}}} = 500$ inner iterations is reached or if a relative residual of $\text{tol}_{\text{inner}} = 10^{-9}$ is reached, whichever happens first.\footnote{While this stopping condition uses the condition number $\kappa$, which will probably be unknown in practice, it can be generally chosen independently of $\kappa$ without any problems of convergence.} For the weight operator update step, we use a variant of the randomized Block Krylov method \cite{MuscoMusco15} based on the implementation provided by the authors\footnote{\url{https://github.com/cpmusco/bksvd}}, setting the parameter for the maximal number of iterations to $20$.
	\item \texttt{R2RILS} \cite{BauchNadler20} or \emph{rank $2r$ iterative least squares}, a method that optimizes a least squares data fit objective $\|P_{\Omega}(\f{X}_{0}) - P_{\Omega}(\f{X})\|_F$ over $\f{X} \in T_{\f{Z}\hk}\mathcal{M}_{r}$, where $T_{\f{Z}\hk}\mathcal{M}_{r}$ is a tangent space onto the manifold of rank-$r$ matrices, while iteratively updating this tangent space. As above, we stop the outer iterations a threshold of $10^{-9}$ is reached for the relative change of the Frobenius norm $\frac{\|\f{X}\hkk - \f{X}\hk\|_F}{\|\f{X}\hk\|_F}$. At each outer iteration, \texttt{R2RILS} solves an overdetermined least squares problem of size $(m \times r(d_1 +d_2))$ via the iterative solver LSQR, for which we choose the maximal number of inner iterations as $N_{\text{LSQR\_{inner}}} = 500$ and a termination criterion based on a relative residual of $10^{-10}$. We use the implementation based on the code provided by the authors, but adapted for these stopping criteria.\footnote{\url{https://github.com/Jonathan-WIS/R2RILS}}
	\item  \texttt{RTRMC}, the preconditioned Riemannian trust-region method called RTRMC 2p of \cite{boumal_absil_15}, which was reported to achieve the best performance among a variety of matrix completion algorithms for the task of completing matrices of a condition number of up to $\kappa = 150$. We use the implementation provided by the authors\footnote{RTRMC v3.2 from \url{http://web.math.princeton.edu/~nboumal/RTRMC/index.html}, together with the toolbox Manopt 6.0 (\url{https://www.manopt.org/}) \cite{manopt}.} with default options except from setting the maximal number of inner iterations to $N_{\text{inner}}=500$ and setting the parameter for the tolerance on the gradient norm to $10^{-15}$. Furthermore, as the algorithm otherwise would often run into certain submatrices that are not positive definite for $\rho$ between $1$ and $1.5$, we set the regularization parameter $\lambda = 10^{-8}$, which is small enough not to deter high precision approximations of $\f{X}^0$ if enough samples are provided.
\end{itemize}

Furthermore, we consider the following first-order algorithms, setting the maximal number of outer iterations to $N_0=4000$:

\begin{itemize}
	\item \texttt{LRGeomCG} \cite{Vandereycken13}, a local optimization method for a quadratic data fit term based on gradients with respect to the Riemannian manifold of fixed rank matrices. We use the author's implementation\footnote{\url{http://www.unige.ch/math/vandereycken/matrix_completion.html}} while setting the parameters related to the stopping conditions \texttt{abs\_grad\_tol}, \texttt{rel\_grad\_tol}, \texttt{abs\_f\_tol}, \texttt{rel\_f\_tol}, \texttt{rel\_tol\_change\_x} and \texttt{rel\_tol\_change\_res} each to $10^{-9}$.
	The rank-adaptive variant of \texttt{LRGeomCG}, called \texttt{LRGeomCG Pursuit} \cite{Uschmajew_Vandereycken,Tan2014}, is used with the same algorithmic parameters as  \texttt{LRGeomCG} for the inner iterations, and with a rank increase of $1$ each outer iteration.
	\item  \texttt{LMaFit} or low-rank matrix fitting \cite{Wen12}, a nonlinear successive over-relaxation algorithm based on matrix factorization. We use the implementation provided by the authors\footnote{\url{http://lmafit.blogs.rice.edu}}, setting the tolerance threshold for the stopping condition (which is based on a relative data fit error $\|P_{\Omega}(\f{X}^{(k)})-\f{y}\|_2/\|\f{y}\|_2$) to $5 \cdot 10^{-10}$.
	\item \texttt{ScaledASD} or scaled alternating steepest descent \cite{TannerWei16}, a gradient descent method based on matrix factorization which scales the gradients in a quasi-Newton fashion. We use the implementation provided by the authors\footnote{\label{algWei}\url{http://www.sdspeople.fudan.edu.cn/weike/code/mc20140528.tar}} with the stopping condition of $\|P_{\Omega}(\f{X}^{(k)})-\f{y}\|_2/\|\f{y}\|_2 \leq 10^{-9}$.
	\item  \texttt{ScaledGD} or scaled gradient descent \cite{tong_ma_chi}, a method that is very similar to \texttt{ScaledASD}, but for which a non-asymptotic local convergence analysis has been achieved for the case of a matrix recovery problem related to matrix completion, and which has been investigated experimentally in \cite{tong_ma_chi} in the light of the completion of ill-conditioned low-rank matrices. We use an adapted version of the author's implementation\footnote{\url{https://github.com/Titan-Tong/ScaledGD}}: We choose a step size of $\eta = 0.5$, but increase the normalization parameter $p$ by a factor of $1.5$ in case the unmodified algorithm \texttt{ScaledGD} leads to divergent algorithmic iterates, using the same stopping condition as for \texttt{ScaledASD}.
	\item \texttt{NIHT} or normalized iterative hard thresholding \cite{TannerWei13}, which performs iterative hard thresholding steps with adaptive step sizes. We use the implementation provided by the authors \footref{algWei} with a stopping threshold of $10^{-9}$ for the relative data fit error $\|P_{\Omega}(\f{X}^{(k)})-\f{y}\|_2/\|\f{y}\|_2$ and the convergence rate threshold parameter $1- 10^{-9}$. 
	\item \texttt{R3MC} \cite{MishraS14}, a Riemannian nonlinear conjugate-gradient that also optimizes a least squares data fit objective $\|P_{\Omega}(\f{X}_{0}) - P_{\Omega}(\f{X})\|_F$ by exploiting a three-factor matrix factorization similar to the SVD and performs a search on a quotient manifold defined from the manifold of rank r matrices, this factorization and symmetries from the action of the orthogonal group. We use the author's implementation\footnote{\url{https://bamdevmishra.in/codes/r3mc/}. We used the version from Sep. 2020 which already includes the rank updating strategy.} while choosing the Polyak-Ribier rule for the nonlinear CG and the Armijo line search with a maximum of 50 line searches allowed at each iteration. Also, we set the tolerance parameter for stopping criterion to $10^{-9}$. \texttt{R3MC w/ Rank Upd} corresponds to the method described in the section on \emph{rank updating} of \cite{MishraS14}.
\end{itemize}

\subsection{Remark to Experiment of Figure \ref{fig:plateau}}
Tracking the relative Frobenius error to gauge the performance of methods for the recovery of highly ill-conditioned matrices without taking account the condition number $\kappa$ might not provide a full picture, as a recovery of the singular spaces corresponding to the smallest singular values can only be expected once the relative error is smaller than $1 / \kappa$.

\begin{centering}
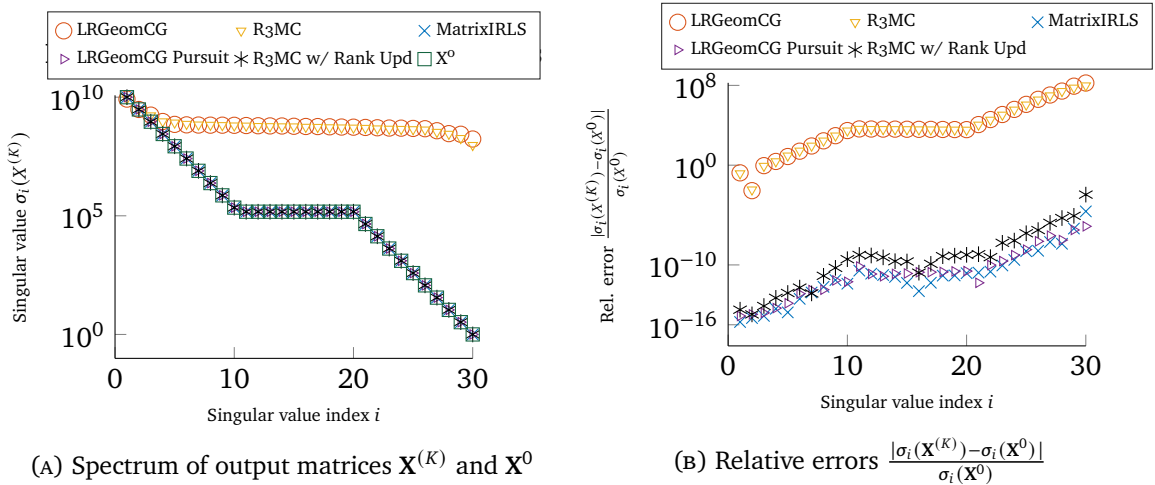
\begin{figure}[!h]
\begin{subfigure}[b]{0.49\textwidth}
\setlength\figureheight{35mm} 
\setlength\figurewidth{50mm}
%
%
\definecolor{mycolor1}{rgb}{0.85000,0.32500,0.09800}%
\definecolor{mycolor2}{rgb}{0.92900,0.69400,0.12500}%
\definecolor{mycolor3}{rgb}{0.00000,0.44700,0.74100}%
\definecolor{mycolor4}{rgb}{0.49400,0.18400,0.55600}%
\definecolor{mycolor5}{rgb}{0.08000,0.39200,0.25100}%
\begin{tikzpicture}

\begin{axis}[%
width=0.951\figurewidth,
height=\figureheight,
at={(0\figurewidth,0\figureheight)},
scale only axis,
xmin=0,
xmax=30,
xlabel style={font=\color{white!15!black}},
xlabel={Singular value index $i$},
ymin=0,
ymax=12000000000,
ymode=log,
ytick={1e0,1e5,1e10},
ylabel style={font=\color{white!15!black}},
ylabel={Singular value $\sigma_i(X^{(K)})$},
axis background/.style={fill=white},
title style={font=\bfseries},
title={First singular values of reconstuctions},
axis x line*=bottom,
axis y line*=left,
legend style={legend cell align=left, align=left, draw=white!15!black},
xlabel style={font=\tiny},ylabel style={font=\tiny},legend style={font=\fontsize{7}{30}\selectfont, anchor=south, legend columns = 3, at={(0.5,1.06)}}
]
\addplot[only marks, mark=o, mark options={}, mark size=3.0619pt, draw=mycolor1] table[row sep=crcr]{%
x	y\\
1	8300650878.21071\\
2	3030735012.34533\\
3	1753273560.83877\\
4	943225564.43873\\
5	713673108.585779\\
6	681131128.047702\\
7	669646417.326314\\
8	661507439.662571\\
9	660182288.104973\\
10	653050345.123534\\
11	638056995.36337\\
12	625402219.076393\\
13	614799440.481466\\
14	610646025.181806\\
15	593855585.301553\\
16	581564752.74762\\
17	577407158.756374\\
18	564048119.227538\\
19	551129452.222437\\
20	544290271.272534\\
21	530135312.753894\\
22	513868680.215522\\
23	500424847.333545\\
24	494608722.158205\\
25	481779388.556324\\
26	468704838.612706\\
27	385192187.808027\\
28	290636767.752205\\
29	273770750.803956\\
30	172425268.845832\\
};
\addlegendentry{LRGeomCG}

\addplot[only marks, mark=triangle, mark options={rotate=180}, mark size=2.0412pt, draw=mycolor2] table[row sep=crcr]{%
x	y\\
1	8541724788.07118\\
2	3028570546.24131\\
3	1679713240.2173\\
4	940207775.904359\\
5	835188816.238836\\
6	734894071.333871\\
7	688924478.079158\\
8	671196310.828307\\
9	650890704.079185\\
10	628515050.170801\\
11	620274729.465496\\
12	596166669.192016\\
13	588281583.941\\
14	579170198.147236\\
15	574615974.284317\\
16	553196635.524704\\
17	546237692.266927\\
18	527528449.569069\\
19	519001312.376597\\
20	510087923.658902\\
21	496978655.990136\\
22	472708993.764719\\
23	459206909.807403\\
24	449815823.535699\\
25	438832540.13134\\
26	423776988.9535\\
27	338366616.219407\\
28	280572696.788111\\
29	186033685.996832\\
30	94635854.5364943\\
};
\addlegendentry{R3MC}

\addplot[only marks, mark=x, mark options={}, mark size=3.0619pt, draw=mycolor3] table[row sep=crcr]{%
x	y\\
1	10000000000\\
2	3039195382.3132\\
3	923670857.187387\\
4	280721620.394118\\
5	85316785.2417278\\
6	25929437.9740477\\
7	7880462.81567125\\
8	2395026.61998901\\
9	727895.384400529\\
10	221221.629106755\\
11	148735.210732655\\
12	148735.210732347\\
13	148735.210730429\\
14	148735.210730285\\
15	148735.210729622\\
16	148735.210729175\\
17	148735.210728977\\
18	148735.210728578\\
19	148735.210728275\\
20	148735.210726687\\
21	45203.5365628424\\
22	13738.237959123\\
23	4175.31893692741\\
24	1268.96100357367\\
25	385.662043032337\\
26	117.210230061718\\
27	35.6224781925662\\
28	10.8263674819052\\
29	3.29034310040211\\
30	0.999976646513891\\
};
\addlegendentry{MatrixIRLS}

\addplot[only marks, mark=triangle, mark options={rotate=270}, mark size=2.0412pt, draw=mycolor4] table[row sep=crcr]{%
x	y\\
1	10000000000\\
2	3039195382.3132\\
3	923670857.187388\\
4	280721620.394116\\
5	85316785.2417269\\
6	25929437.9740501\\
7	7880462.81567225\\
8	2395026.6199867\\
9	727895.384396282\\
10	221221.629106592\\
11	148735.210739713\\
12	148735.210731019\\
13	148735.210728521\\
14	148735.210727604\\
15	148735.210727017\\
16	148735.21072693\\
17	148735.210726789\\
18	148735.210726346\\
19	148735.210725954\\
20	148735.210725369\\
21	45203.5365635298\\
22	13738.237960155\\
23	4175.31893757462\\
24	1268.96100438283\\
25	385.662040802903\\
26	117.210232481682\\
27	35.6224758182314\\
28	10.8263670023975\\
29	3.29034351651083\\
30	0.999999260682693\\
};
\addlegendentry{LRGeomCG Pursuit}

\addplot[only marks, mark=asterisk, mark options={}, mark size=2.9580pt, draw=black] table[row sep=crcr]{%
x	y\\
1	10000000000\\
2	3039195382.31319\\
3	923670857.187379\\
4	280721620.394131\\
5	85316785.2417148\\
6	25929437.9740329\\
7	7880462.81566881\\
8	2395026.62000689\\
9	727895.384361368\\
10	221221.62919\\
11	148735.210878246\\
12	148735.210860637\\
13	148735.210813757\\
14	148735.210766134\\
15	148735.210756372\\
16	148735.210732606\\
17	148735.210710703\\
18	148735.210617663\\
19	148735.210603069\\
20	148735.210576134\\
21	45203.5365099504\\
22	13738.2379505741\\
23	4175.31900516407\\
24	1268.96096830659\\
25	385.661984487046\\
26	117.210266286598\\
27	35.6224203854896\\
28	10.8264215464759\\
29	3.29037211538973\\
30	0.998880604917674\\
};
\addlegendentry{R3MC w/ Rank Upd}

\addplot[only marks, mark=square, mark options={}, mark size=2.5000pt, draw=mycolor5] table[row sep=crcr]{%
x	y\\
1	10000000000\\
2	3039195382.31319\\
3	923670857.187388\\
4	280721620.394117\\
5	85316785.241728\\
6	25929437.9740466\\
7	7880462.81566991\\
8	2395026.61998749\\
9	727895.384398315\\
10	221221.629107045\\
11	148735.210729352\\
12	148735.210729351\\
13	148735.210729351\\
14	148735.210729351\\
15	148735.210729351\\
16	148735.210729351\\
17	148735.210729351\\
18	148735.210729351\\
19	148735.210729351\\
20	148735.21072935\\
21	45203.5365636025\\
22	13738.2379588326\\
23	4175.3189365604\\
24	1268.96100316792\\
25	385.662042116348\\
26	117.21022975335\\
27	35.6224789026184\\
28	10.8263673387385\\
29	3.29034456231756\\
30	0.999999999999992\\
};
\addlegendentry{$\text{X}^\text{0}$}

\end{axis}
\end{tikzpicture}%
\caption{Spectrum of output matrices $\f{X}^{(K)}$ and $\f{X}^0$}
\label{fig:singplateau}
\end{subfigure}
\begin{subfigure}[b]{0.49\textwidth}
\setlength\figureheight{35mm} 
\setlength\figurewidth{50mm}
%
%
\definecolor{mycolor1}{rgb}{0.85000,0.32500,0.09800}%
\definecolor{mycolor2}{rgb}{0.92900,0.69400,0.12500}%
\definecolor{mycolor3}{rgb}{0.00000,0.44700,0.74100}%
\definecolor{mycolor4}{rgb}{0.49400,0.18400,0.55600}%
\begin{tikzpicture}

\begin{axis}[%
width=0.951\figurewidth,
height=\figureheight,
at={(0\figurewidth,0\figureheight)},
scale only axis,
xmin=0,
xmax=30,
xlabel style={font=\color{white!15!black}},
xlabel={Singular value index $i$},
ymin=0,
ymax=180000000,
ymode=log,
ytick={1e-16,1e-10,1e0,1e8},
ylabel style={font=\color{white!15!black}},
ylabel={Rel. error $\frac{|\sigma_i(X^{(K)})-\sigma_i(X^{0})|}{\sigma_i(X^{0})}$},
axis background/.style={fill=white},
title style={font=\bfseries},
title={Rel. errors to singular value of X0},
axis x line*=bottom,
axis y line*=left,
legend style={legend cell align=left, align=left, draw=white!15!black},
xlabel style={font=\tiny},ylabel style={font=\tiny},legend style={font=\fontsize{7}{30}\selectfont, anchor=south, legend columns = 3, at={(0.5,1.05)}}
]
\addplot[only marks, mark=o, mark options={}, mark size=3.0619pt, draw=mycolor1] table[row sep=crcr]{%
x	y\\
1	0.169934912178929\\
2	0.00278375323189352\\
3	0.898158361494222\\
4	2.36000327696347\\
5	7.36497890261252\\
6	25.2686421791888\\
7	83.9755189498207\\
8	275.200453949871\\
9	905.974137019273\\
10	2951.01851536649\\
11	4288.88530849243\\
12	4203.80272297066\\
13	4132.51645159664\\
14	4104.59155553947\\
15	3991.70342502943\\
16	3909.0677633481\\
17	3881.11477245334\\
18	3791.29717335675\\
19	3704.44035618648\\
20	3658.45809740347\\
21	11726.7397534589\\
22	37403.2640515733\\
23	119852.083066701\\
24	389773.564327375\\
25	1249225.87727454\\
26	3998838.00576798\\
27	10813176.5124621\\
28	26845270.2399901\\
29	83204279.1657456\\
30	172425267.844439\\
};
\addlegendentry{LRGeomCG}

\addplot[only marks, mark=triangle, mark options={rotate=180}, mark size=2.0412pt, draw=mycolor2] table[row sep=crcr]{%
x	y\\
1	0.145827521192882\\
2	0.00349593715945922\\
3	0.818519256233866\\
4	2.34925316612366\\
5	8.78926730387808\\
6	27.3420748289817\\
7	86.4218296810268\\
8	279.245866675091\\
9	893.209137783196\\
10	2840.11030511702\\
11	4169.32877705192\\
12	4007.24166832185\\
13	3954.22742097348\\
14	3892.96831662903\\
15	3862.34864129737\\
16	3718.3387686211\\
17	3671.55130502286\\
18	3545.76237712802\\
19	3488.43138502026\\
20	3428.50348580939\\
21	10993.2427909566\\
22	34407.2694724911\\
23	109980.277307092\\
24	354474.687127302\\
25	1137867.1130329\\
26	3615528.03569326\\
27	9498681.47923829\\
28	25915681.33457\\
29	56539270.9437402\\
30	94635853.5357299\\
};
\addlegendentry{R3MC}

\addplot[only marks, mark=x, mark options={}, mark size=3.0619pt, draw=mycolor3] table[row sep=crcr]{%
x	y\\
1	1.9073486328125e-16\\
2	4.70687565180684e-16\\
3	7.74362135320225e-16\\
4	4.03420388191858e-15\\
5	1.74656852712256e-15\\
6	3.97966748722748e-14\\
7	1.71953650155042e-13\\
8	6.34419996756901e-13\\
9	3.04130759352308e-12\\
10	1.30454505725264e-12\\
11	2.68032328643198e-11\\
12	1.32677744189958e-11\\
13	9.66049735367334e-12\\
14	6.32207836742621e-12\\
15	1.57968797116072e-12\\
16	2.34223523037209e-13\\
17	1.63291169569382e-12\\
18	8.43145980296599e-12\\
19	8.81380964758981e-12\\
20	1.48529413029761e-11\\
21	1.681548280817e-11\\
22	2.11375406510832e-11\\
23	8.78992702325313e-11\\
24	3.19749891292884e-10\\
25	2.37511213736133e-09\\
26	2.63091568977508e-09\\
27	1.99329766277485e-08\\
28	1.32227135833514e-08\\
29	4.4429879486134e-07\\
30	2.33534977656437e-05\\
};
\addlegendentry{MatrixIRLS}

\addplot[only marks, mark=triangle, mark options={rotate=270}, mark size=2.0412pt, draw=mycolor4] table[row sep=crcr]{%
x	y\\
1	7.62939453125e-16\\
2	1.09827098542159e-15\\
3	1.54872427064045e-15\\
4	4.67118344222151e-15\\
5	1.4321861922405e-14\\
6	1.30165297596682e-13\\
7	2.98762080819207e-13\\
8	3.25667635478949e-13\\
9	2.79261000686298e-12\\
10	2.04838307194671e-12\\
11	6.96997937588778e-11\\
12	1.22665031012168e-11\\
13	7.24938436214077e-12\\
14	1.26803566945532e-11\\
15	1.39041110078989e-11\\
16	1.84656972811181e-11\\
17	1.87312288774101e-11\\
18	1.98667335592688e-11\\
19	2.2579773782252e-11\\
20	2.39369787578461e-11\\
21	1.6087943971251e-12\\
22	9.62540934721529e-11\\
23	2.42907369744542e-10\\
24	9.57400705351611e-10\\
25	3.40567081901003e-09\\
26	2.32772800789919e-08\\
27	8.65856022456131e-08\\
28	3.10678570128408e-08\\
29	3.17842303751958e-07\\
30	7.39322125958423e-07\\
};
\addlegendentry{LRGeomCG Pursuit}

\addplot[only marks, mark=asterisk, mark options={}, mark size=2.9580pt, draw=black] table[row sep=crcr]{%
x	y\\
1	3.24249267578125e-15\\
2	1.09827098542159e-15\\
3	7.61456099731554e-15\\
4	5.01090587438308e-14\\
5	1.53523373534073e-13\\
6	5.31292793060188e-13\\
7	1.39690181775436e-13\\
8	8.10183305695987e-12\\
9	5.07573381947232e-11\\
10	3.74992239752988e-10\\
11	1.00136641605146e-09\\
12	8.81652366619007e-10\\
13	5.7467749105787e-10\\
14	2.35677391688977e-10\\
15	2.20997427490733e-10\\
16	1.67441446030217e-11\\
17	1.27587247154191e-10\\
18	7.50630036802396e-10\\
19	8.46024757632375e-10\\
20	1.02674767536364e-09\\
21	1.1868997763132e-09\\
22	6.01129703633233e-10\\
23	1.64307607442186e-08\\
24	2.74723406206838e-08\\
25	1.49429529172878e-07\\
26	3.1168997266501e-07\\
27	1.64270255291886e-06\\
28	5.00700979206902e-06\\
29	8.37392105801331e-06\\
30	0.00111939509022503\\
};
\addlegendentry{R3MC w/ Rank Upd}

\end{axis}
\end{tikzpicture}%
\caption{Relative errors $\frac{| \sigma_i(\f{X}^{(K)}) - \sigma_i(\f{X}^0)|}{\sigma_i(\f{X}^0)}$}
\label{fig:singplateau:2}
\end{subfigure}
\caption{Spectrum information of algorithmic output $\f{X}^{(K)}$ after convergence, experiment of Figure \ref{fig:plateau} ($1000 \times 1000$ matrix, $r=30$, $\kappa=10^{10}$, $\rho= 1.5$)}
\end{figure}
\end{centering}

For this reason, we report in Figure \ref{fig:singplateau} the singular values of the recovered matrices $\f{X}^{(K)}$ and the relative error on a basis of individual singular values in Figure \ref{fig:singplateau:2} for the very experiment conducted in Section \ref{sec:rank_update} and illustrated in Figure \ref{fig:plateau}. We observe that  \texttt{MatrixIRLS}, \texttt{LRGeomCG Pursuit} and \texttt{R3MC w/ Rank Upd} each are able to recover even the smallest singular values such (with indices $i=28,29,30$) with a high precision of a relative error between $10^{-7}$ and  $10^{-3}$.

This shows that despite a not too restrictive choice of the tolerance on of the inner conjugate gradient iterations (such as $\text{tol}_{\text{inner}} = 10^{-3}$), \texttt{MatrixIRLS} is successful in recovering the complete spectrum of $\f{X}^0$, indicating that an implementation of \texttt{MatrixIRLS} that solves \eqref{eq:gamma:system} via conjugate gradient method together with weight updates based on a randomized block Krylov method can be very precise even without requiring an very high precision on the iterative solver.

These observations suggest that \texttt{MatrixIRLS} and Riemannian optimization methods with adaptive rank updates such as \texttt{LRGeomCG Pursuit} and \texttt{R3MC w/ Rank Upd} are good alternatives to solve hard matrix recovery problems in a numerically efficient way, warranting further investigations for a better theoretical understanding.

\bibliographystyle{alphaabbr}
\bibliography{MatrixIRLS_wellconditioned}


\end{document}